\documentclass[reqno,a4paper,12pt]{amsart}
\usepackage{enumerate,amsmath,amssymb,stmaryrd} 
\usepackage{pstricks,pst-node,pst-coil,pst-plot} 
\usepackage{hyperref}
\usepackage{esint}
\usepackage{xfrac}
\usepackage{accents}
\newtheorem{theorem}{Theorem}[section]
\newtheorem{proposition}[theorem]{Proposition}
\newtheorem{cor}[theorem]{Corollary}
\newtheorem{lemma}[theorem]{Lemma}

\theoremstyle{definition}
\newtheorem{definition}   [theorem]  {Definition}

\theoremstyle{definition}
\newtheorem{remark}   [theorem]  {Remark}

\newcommand{\definedas}{\mathrel{\raise.095ex\hbox{\rm :}\mkern-5.2mu=}}

\let\oldmarginpar\marginpar
\renewcommand\marginpar[1]{\-\oldmarginpar[\raggedleft\tiny #1]%
{\raggedright\tiny #1}}

\newcounter{mnotecount}[section]

\DeclareMathOperator{\Hess}{Hess}
\DeclareMathOperator{\graph}{graph}

\DeclareMathOperator{\tr}{tr}
\DeclareMathOperator{\const}{const}
\DeclareMathOperator{\proj}{proj}

\DeclareMathOperator{\supp}{supp}

\newcommand{\bR}{\mathbb{R}}

\newcommand{\bS}{\mathbb{S}}

\newcommand{\bH}{\mathbb{H}}

\newcommand{\cL}{\mathcal{L}}

\newcommand{\gbar}{\bar{g}}
\newcommand{\ghat}{\hat{g}}

\renewcommand{\hbar}{\overline{h}}

\newcommand{\gtil}{\tilde{g}}

\newcommand{\fbar}{\overline{f}}
\newcommand{\funder}{\underline{f}}

\newcommand{\rhotil}{\widetilde{\rho}}

\newcommand{\Sigmatil}{\widetilde{\Sigma}}

\newcommand{\gcheck}{\check{g}}

\DeclareMathOperator{\dist}{dist}

\DeclareMathOperator{\divg}{div}

\DeclareMathOperator{\vol}{Vol}

\newcommand{\ric}{\mathrm{Ric}}

\newcommand{\ricdd}[2]{\ric_{#1 #2}}

\newcommand{\scal}{\mathrm{Scal}}

\newcommand{\hess}{\mathrm{Hess}}

\begin{document} 
%%%%%%%%%%%%%%%%%%%%%%%%%%%%%%%%%%%%%%%%%%%%%%%%%%%%%%%%%%%%%%%%%%%%%%%%%

\author{Anna Sakovich}
\address{Uppsala University \\ Department of Mathematics \\
Box 480, 751 06 UPPSALA}
\email{anna.sakovich@math.uu.se}

\title[Jang equation and positive mass theorem in AH setting]
{The Jang equation and the positive mass theorem in the asymptotically hyperbolic setting}

\begin{abstract}
We solve the Jang equation with respect to asymptotically hyperbolic  ``hyperboloidal''  initial data. The results are applied to give a non-spinor proof of the positive mass theorem in the asymptotically hyperbolic setting. This work focuses on the case when the spatial dimension is equal to three.
\end{abstract}

\date{\today}

\maketitle

\section{Introduction}\label{secIntro}

The classical positive mass theorem has its roots in general relativity and asserts that for a nontrivial isolated physical system, the energy of the gravitational field is nonnegative. Considered from the point of view of differential geometry, the theorem is a statement about initial data  for the Einstein equations. Such initial data is a triple $(M,g,K)$, where $(M,g)$ is a Riemannian manifold and $K$ is a symmetric 2-tensor. In the context of the positive mass theorem it is standard to assume that $(M,g,K)$ satisfies the so-called dominant energy condition, a condition on the stress energy tensor of the matter or electromagnetic fields which is satisfied by almost all ``reasonable''  fields.  

Roughly speaking, a manifold $(M,g)$ is asymptotically Euclidean if outside some compact set it consists of a finite number of components $M_k$ such that each $M_k$ is diffeomorphic to a complement of a compact set in Euclidean space. Moreover, it is required that under these diffeomorphisms, the geometry at infinity of each end $M_k$ tends to that of the Euclidean space. In this setup, with each $M_k$ one can associate the so-called Arnowitt-Deser-Misner (ADM) mass which is the limit of surface integrals taken over large 2-spheres  in $M_k$\footnote{Note that the quantity that we, following the terminology of \cite{PMT2}, call ADM mass in this work is more commonly referred to as ADM energy.}. An initial data set $(M,g,K)$ is called asymptotically Euclidean if  $(M,g)$ is an asymptotically Euclidean manifold and $K$ falls off to zero sufficiently fast near infinity. The positive mass theorem for asymptotically Euclidean initial data sets states that the ADM mass  for each $M_k$ is nonnegative provided that the dominant energy condition is satisfied, and if the mass is zero then $M$ arises as a hypersurface in Minkowski spacetime, with the induced metric $g$ and second fundamental form $K$.

A complete proof of this theorem was first obtained by Schoen and  Yau in \cite{PMT1} for the special case when $K\equiv 0$. This result is also known as the Riemannian positive mass theorem: if $\scal_g\geq 0$ (the dominant energy condition when $K\equiv 0$) holds then the ADM mass of $(M,g)$ is positive unless $(M,g)$ is isometric to Euclidean space. Shortly after this important case was resolved, Schoen and Yau were able to address the general case in \cite{PMT2}  using a certain reduction argument. The key idea is to consider a smooth function $f$ on $M$ whose  graph $\Sigma$ in $M\times \bR$ equipped with the standard product metric has mean curvature equal to the trace of $K$ (trivially extended to be a tensor defined over $M\times \bR$) on $\Sigma$. Schoen and Yau observed that, as long as the dominant energy condition is satisfied, $\Sigma$ can be equipped with an asymptotically Euclidean metric such that its scalar curvature vanishes and its ADM mass does not exceed the ADM mass of $(M,g,K)$.  All in all, it follows from the Riemannian positive mass theorem that the ADM mass of $(M,g,K)$ is nonnegative, and in the case when the mass is zero the function $f$ provides the graphical embedding into the Minkowski spacetime.

The prescribed mean curvature equation that plays a central role in Schoen and Yau's argument is known as the Jang equation. It first appeared in the eponymous paper of Jang \cite{Jang} where it was motivated by a question related to the characterization of the case when the mass is zero: Which conditions ensure that an initial data set $(M,g,K)$ arises as a hypersurface in Minkowski spacetime such that the induced metric is $g$ and the second fundamental form is $K$? A substantial part of \cite{PMT2} is devoted to the construction of a solution and careful analysis of its geometric and analytic properties. In fact, it turns out that the hypersurface $\Sigma \subset M\times \bR$ as described above  is not necessarily a graph as it might have asymptotically cylindrical components. Nevertheless, its structure and asymptotics are well understood so that the reduction argument described above can be applied. Importantly, the reduction argument of \cite{PMT2} was shown to work in dimensions $3<n\leq 7$, see Eichmair \cite{EichmairJang}. Furthermore, in the light of Schoen and Yau's recent work \cite{PMTfinal}  it is natural to anticipate the extension of these results to dimensions $n>7$. For other important developments concerning spacetime positive mass theorem in higher dimensions  see   \cite{EHLS},  \cite{HuangLeeRigidity}, \cite{LohkampNew}.

The current work has been largely motivated by another paper of Schoen and Yau \cite{Bondi}, which contains a sketch of the proof that the Bondi mass, representing the total mass of an isolated physical system measured after the loss due to the gravitational radiation, is positive. The idea of the argument is to pick a suitable asymptotically null hypersurface in the radiating Bondi spacetime and use the Jang equation for deforming it to an asymptotically Euclidean manifold with ``almost nonnegative'' scalar curvature and the ADM mass equal to the positive multiple of the Bondi mass. Completing all steps in this argument would require one to analyze the Jang equation in the asymptotically null setting, and the preliminary analysis carried out in \cite{Bondi} (see also \cite{HYZ}) indicates that this can be somewhat problematic in the radiating regime. Therefore in the current paper we turn to the non-radiating -- but still rather general -- setting of asymptotically hyperbolic initial data sets.  

Roughly speaking, a manifold $(M,g)$ is asymptotically hyperbolic if its geometry at infinity approaches that of the hyperbolic space. The definition of mass for such manifolds is due to Wang \cite{Wang}, and Chrusciel and Herzlich \cite{CH}; see also \cite{Herzlich} where the relation between these two approaches is discussed. The respective (Riemannian) positive mass theorem stating that an $n$-dimensional asymptotically hyperbolic manifold $(M,g)$ with $\scal^g \geq -n(n-1)$ has positive mass unless it is isometric to hyperbolic space was proven under spinor assumption in \cite{Wang} and \cite{CH}. In \cite{AnderssonCaiGalloway} the spinor assumption was replaced by the restriction on dimension and the geometry at infinity. These assumptions have recently been removed in \cite{CGNP}, \cite{ChruscielDelay}, and \cite{HJM}. 

An asymptotically hyperbolic manifold  $(M,g)$ with $\scal^g \geq -n(n-1)$ can be viewed as either a spacelike totally geodesic hypersurface in an asymptotically anti-de Sitter spacetime (in which case $K= 0$) or as an umbilic ``hyperboloidal'' hypersurface in an asymptotically Minkowski spacetime (in which case $K=g$). Consequently, an initial data set $(M,g,K)$ is called asymptotically hyperbolic if  $(M,g)$ is an asymptotically hyperbolic manifold and either $K\to 0$ or $K\to g$ sufficiently fast near infinity. There is a vast literature devoted to spinor proofs of positive mass theorem in both cases, see e.g. \cite{ChruscielBondi}, \cite{ChruscielMaerten}, \cite{ChruscielMaertenTod}, \cite{Maerten}, \cite{WangXu},  \cite{XieZhang}, \cite{ZhangAngular}, \cite{Zhang}.  The initial data sets we are considering in this paper are ``hyperboloidal'', that is we assume $K\to g$ at infinity.  

In this work we apply Schoen and Yau's reduction argument using the Jang equation to deform an asymptotically hyperbolic initial data set satisfying the dominant energy condition to an asymptotically Euclidean manifold with ``almost nonnegative'' scalar curvature which in particular yields a proof of the positive mass conjecture in the ``hyperboloidal'' setting. In the current paper we focus on the case when $n=3$. In this case, similarly to \cite{PMT2}, the Jang equation can be solved without resorting to techniques from geometric measure theory that are required for dealing with higher dimensions, see \cite{EichmairJang}.  Furthermore, we could rely on the findings of \cite{Bondi} and \cite{HYZ} to get some intuition about the asymptotics of solutions. Our main result is the following theorem.

\begin{theorem}\label{thMain} 
 Let $(M,g,K)$ be a 3-dimensional asymptotically hyperbolic initial data set of type  
$(l,\beta, \tau,\tau_0)$ for $l\geq 6$, $0<\beta<1$, $\tfrac{3}{2} < \tau < 3$  and
$\tau_0>0$. Assume that the dominant energy condition $\mu \geq |J|_g$
holds. Then the mass vector $(E,\vec{P})$ is causal future directed, that is $E \geq |\vec{P}|$.

Suppose in addition that $(M,g,K)$ has Wang's asymptotics. If $E=0$ then $(M,g)$ can be embedded
isometrically into Minkowski space as a spacelike graphical  hypersurface with second fundamental form $K$.
\end{theorem}

When working towards the proof of this result we encountered a few difficulties that are not present in the asymptotically Euclidean setting of \cite{PMT2} and \cite{EichmairJang}. One problem is that barriers for the Jang equation are required to have more complicated asymptotics which makes it difficult to find them by inspection. See Section \ref{secBarriers}, where our construction of barriers is described, for more details. Another difficulty is that the rescaling technique -- which is a commonly used method for proving estimates for solutions of geometric PDEs in the asymptotically Euclidean setting -- does not work on asymptotically hyperbolic manifolds. Consequently, we had to devise a new method for proving that the Jang graph is an asymptotically Euclidean manifold, see Section \ref{secJangAE} for details. An additional issue  that requires some further adjustments is the fact that the asymptotics of the asymptotically Euclidean metric induced on the Jang graph are worse than in the  setting of  \cite{PMT2} and \cite{EichmairJang}, see Section \ref{secConformal}.

Of course, the result of Theorem \ref{thMain} is essentially covered by some of the aforementioned spinor proofs (see also \cite{ChenWangYau} where $E \geq |\vec{P}|$ is proven under an additional assumption on the asymptotic expansion of the initial data). In this connection we would like to point out that our result is currently being extended to the case $3<n\leq 7$  in \cite{Lundberg}. Interestingly, this case turns out to be different from the case $n=3$ in a few respects. The extension to dimensions $n>7$ might also be possible in the view of Schoen and Yau's recent work \cite{PMTfinal}.

We would  also like to stress that the Jang equation has many important applications besides proving positive mass theorems.  Among them are existence results for marginally outer trapped surfaces obtained by Andersson, Eichmair and Metzger (see \cite{AEM} for an overview) and  reduction arguments for the spacetime Penrose conjecture of Bray and Khuri (see e.g. \cite{BrayKhuri}). Other important works where the Jang equation  plays a prominent role include (but do not restrict to) \cite{ADGP} of Andersson, Dahl, Galloway and Pollack on topological censorship, \cite{BourniMoore} of Bourni and Moore on the null mean curvature flow, of Wang and Yau \cite{WangYau} on the notion of quasilocal mass, as well as the recent work of Bryden, Khuri, and Sormani \cite{BKS} on the stability of the spacetime positive mass theorem. In the view of these results, we hope that our study of the Jang equation in the asymptotically hyperbolic setting will be useful in other contexts that are out of the scope of the current paper.

The paper is organized as follows. Section \ref{secPrelim} contains some preliminaries and heuristics behind our arguments.  In Section \ref{secBarriers} we construct barriers for the Jang equation that will later be used to ensure that the solution has certain asymptotic behavior at infinity. In Section \ref{secBVP} we solve a sequence of  regularized boundary value problems for the Jang equation and in Section \ref{secExistence} we construct the geometric limit of the respective solutions when the domain grows and the regularization parameter tends to zero. This gives us the so-called geometric solution of the Jang equation. In Section \ref{secJangAE} we study the asymptotic behavior of this solution in more depth and in Section \ref{secConformal} we analyze its conformal properties. Finally, we prove Theorem \ref{thMain}  in Section \ref{secPositivity} and Section \ref{secRigidity}.

\subsection*{Acknowledgments}   I would like to thank Mattias Dahl for the suggestion to work on this problem, for stimulating discussions at the early stages of work, and for collaborating with me on the companion paper \cite{DahlSakovich}.  Thank you to  Romain Gicquaud for all your help  and interesting discussions on the topics related to this paper and to Michael Eichmair for patiently  answering my questions and encouragement. I would also like to thank Ye Sle Cha and Marcus Khuri for their interest and the pleasure of joint work on a related paper \cite{ChaKhuriSakovich}.  I am also grateful to  Piotr Chru\'sciel, Greg Galloway, Ulrich Menne, Rick Schoen, Mu-Tao Wang, and Eric Woolgar for
their interest in this work. A part of this paper is based upon work supported by the National Science Foundation under Grant No. 0932078 000, while the author was in residence at the Mathematical Science Research Institute in Berkeley, California.  The author acknowledges  support from Knut and Alice Wallenberg Foundation and Swedish Research Council (Vetenskapsr{\aa}det). Finally, I would like to thank the Institute Mittag-Leffler and the organizers of the program ``General Relativity, Geometry and Analysis: beyond the first 100 years after Einstein''  during which this paper was put in its final form and the two anonymous referees for their constructive comments.

\section{Preliminaries}\label{secPrelim}

\subsection{Initial data sets}

\begin{definition} 
An \emph{initial data set} $(M,g,K)$ for the Einstein equations of general relativity consists of a 3-dimensional Riemannian manifold $(M,g)$ and a symmetric  2-tensor $K$. The local mass density $\mu$ and the local current density $J$ of $(M,g,K)$ are defined via the constraint equations by
\begin{eqnarray}
2 \mu&\definedas& \scal^g+(\tr^g K)^2-|K|_g^2 \,, \label{eqConstraints1}\\
J&\definedas&\divg^g K-d (\tr^g K)\,, \label{eqConstraints2}
\end{eqnarray}
where $\scal^g$ is the scalar curvature of the metric $g$, and $\tr^g K$ and $|K|_g$ are respectively the trace and the norm of $K$ with respect to $g$. We say that $(M,g,K)$ satisfies the \emph{dominant energy condition} if
\begin{equation}\label{DEC}
\mu\geq |J|_g.
\end{equation}
\end{definition}

In this article, we denote the 3-dimensional hyperbolic space by $\bH^3$ and the hyperbolic metric by $b$. We will almost exclusively work with the hyperboloidal model of the hyperbolic space where $(\bH^3, b)$ is viewed as the unit upper hyperboloid 
\begin{equation}\label{eqHyperboloid}
\left\{(x^0, x^1, x^2, x^3): x^0 = \sqrt{1- (x^1)^2-(x^2)^2-(x^3)^2} \right\}
\end{equation}
 in Minkowski spacetime $\bR^{3,1}=(\bR\times \bR^3, -(dx^0)^2 + (dx^1)^2 + (dx^2)^2 + (dx^3)^2)$. In this case we have $b=\frac{dr^2}{1+r^2} + r^2 \sigma$ on $(0,\infty) \times \bS^2$, where $\sigma$ is the standard round metric on $\bS^2$ and $r^2 = (x^1)^2+(x^2)^2+(x^3)^2$. 
 
 Our definition of asymptotically hyperbolic initial data sets is the same as in \cite{DahlSakovich}:
\begin{definition}\label{defAHdataGeneral}  
We say that an initial data set $(M,g,K)$ is \emph{asymptotically hyperbolic of type $(l,\beta,\tau,\tau_0)$} for $l\geq 2$, $0\leq \beta <1$, $\tau>3/2$, and $\tau_0>0$ if $g\in C^{l,\beta}(M)$, $K\in C^{l-1,\beta}(M)$, and if there exists a compact set $\mathcal{C}$ and a diffeomorphism $\Phi: M\setminus \mathcal{C} \to (R,\infty)\times  \bS^2$ for some $R>0$ such that  
\begin{itemize}
\item $e \definedas \Phi_*g-b\in C^{l,\beta}_\tau (\bH^3)$,
\item $\eta\definedas \Phi_* (K-g) \in C^{l-1,\beta}_\tau (\bH^3)$,
\item $\Phi_*\mu \in C^{l-2,\beta}_{3+\tau_0}(\bH^3)$, and $\Phi_* J \in C^{l-2,\beta}_{3+\tau_0}(\bH^3)$.
\end{itemize}
%as $r\to \infty$. 
\end{definition}
For the definition of \emph{weighted H\"older spaces} $C^{l,\beta}_\tau$, see  \cite{DahlSakovich}.  

In the view of the density result proven in \cite{DahlSakovich} (see Theorem \ref{thPerturb} below), for the purposes of  this article it will mostly suffice to work with initial data having simpler asymptotics, as described in the following definition. 

\begin{definition}\label{defAHdata}  We say that an asymptotically hyperbolic initial data set $(M,g,K)$ of type $(l,\beta,\tau,\tau_0)$ for $l\geq 2$, $0\leq \beta <1$, $\tau>3/2$, and $\tau_0>0$, has \emph{Wang's asymptotics}\footnote{The study of mass of asymptotically hyperbolic manifolds was initiated by  Xiaodong Wang  in \cite{Wang}. The asymptotic behavior of the metric considered here is essentially the same as in \cite{Wang}, hence the name.} if $\tau=3$ and the chart at infinity $\Phi$ is such that 
\begin{eqnarray*}
\Phi_*g & = &\frac{dr^2}{1+r^2}+r^2 \left(\sigma + \mathbf{m} \, r^{-3} + O^{l,\beta}(r^{-4})\right) \\
\Phi_* (K - b)_{|_{T\bS^2 \times T\bS^2}} & = & \mathbf{p} \, r^{-1} + O^{l-1,\beta}(r^{-2})
\end{eqnarray*}
where $\sigma$ is the standard round metric on  $\mathbb{S}^2$, and $\mathbf{m}\in C^{l,\beta}(\mathbb{S}^2)$ and $\mathbf{p}\in C^{l-1,\beta}(\mathbb{S}^2)$ are symmetric 2-tensors on $\bS^2$.  The expression $O^{l,\beta}(r^{-\tau})$ stands for a tensor in the weighted H\"older space $C^{l,\beta}_\tau (\bH^3)$.
\end{definition}

We will now recall the notion of mass in the asymptotically hyperbolic setting. Let $\mathcal{N} \definedas \{ V \in C^{\infty}(\bH^3) \mid \hess^b V = V b \}$. This is a vector space with a basis of the functions 
\[
V_{(0)} = \sqrt{1+r^2}, \quad 
V_{(i)} = x^i r, \quad
i=1,2,3,
\]
where $x^1, x^2, x^3$ are the coordinate functions
on $\bR^{3}$ restricted to $\bS^{2}$. In the hyperboloidal model of the hyperbolic space,  the functions $V_{(a)}$, $a=0,\ldots, 3$, have natural interpretation as  the restrictions to the upper unit hyperboloid \eqref{eqHyperboloid} of the coordinate functions $x^a$ of $\bR^{3,1}$. In fact, there is a natural correspondence between functions in $\mathcal{N}$ and the isometries of Minkowski space preserving the geometry of the hyperboloid, see e.g. \cite[Section 2.2]{DahlSakovich} for details. 

Given an asymptotically hyperbolic initial data set as in Definition \ref{defAHdataGeneral} the \emph{mass functional} $H_\Phi : \mathcal{N} \to \bR$ is well-defined by the formula
\[
H_{\Phi} (V)
=
\lim_{R \to \infty} \int_{\{r=R\}} \left(
V (\operatorname{div}^b e- d \tr^b e) + (\tr^b e) dV - (e+2\eta)(\nabla^b V, \cdot)
\right) (\nu_r) \, d \mu^b,
\]
where $\nu_r = \sqrt{1+r^2}\partial_r$.  
If $\Phi$ is a chart at infinity as in Definition \ref{defAHdataGeneral}
and $\mathcal{I}$ is an isometry of the hyperbolic metric $b$ then $\mathcal{I} \circ \Phi$ 
is again such a chart and it is not complicated to verify that
\[
H_{\mathcal{I} \circ \Phi} (V) = H_{\Phi} (V \circ \mathcal{I}^{-1}).
\]
The components of the \emph{mass vector} $(E,\vec{P})$, where $\vec{P} = (P_1, P_2, P_3)$, are given by 
\[
E = \tfrac{1}{16 \pi} H_\Phi (V_{(0)}),\quad P_i = \tfrac{1}{16 \pi} H_\Phi (V_{(i)}), \quad i=1,2,3.
\]
In what follows we will refer to $E$ as the \emph{energy} of the initial data set $(M,g,K)$. A computation shows that in the case when the initial data has Wang's asymptotics the energy is given by
\begin{equation}\label{eqMassAH}
E = \tfrac{1}{16 \pi} \int_{\bS^2} (\tr^\sigma \mathbf{m} + 2 \tr ^\sigma \mathbf{p})\, d\mu^\sigma.
\end{equation}
The Minkowskian length of the mass vector is a coordinate invariant which is usually referred to as the \emph{mass}.  We note that this definition of mass is essentially the one introduced in \cite{ChruscielBondi} and refer the reader to \cite{Michel} for the proof of well-definiteness and coordinate invariance.

The following density result was proven in \cite{DahlSakovich}. 

\begin{theorem}\label{thPerturb}
Let $(M,g,K)$ be an asymptotically hyperbolic initial data set of type
$(l,\beta, \tau,\tau_0)$ for $l\geq 3$, $0<\beta<1$, $\tfrac{3}{2} < \tau < 3$ and
$\tau_0>0$. % with respect to a chart at infinity denoted by $\Phi$. 
Assume that the dominant energy condition $\mu \geq |J|_g$
holds. Then for every $\varepsilon >0$ there exists an asymptotically hyperbolic initial data set $(M,\bar{g},\bar{K})$ of type $(l-1,\beta, 3,\tau'_0)$ for some $\tau'_0>0$ with Wang's asymptotics (possibly with respect to a different chart at infinity) %denoted by $\Psi$ such that
%\[
%\|\Phi^*g - \Psi^* \bar{g}\|_{C_{\tau'}^{l-1,\beta}(\bH^3)}<\varepsilon,
%\qquad
%\|\Phi^*K - \Psi^*\bar{K}\|_{C_{\tau'}^{l-2,\beta}(\bH^3)} < \varepsilon,
%\]
such that the strict dominant energy condition 
\[
\bar{\mu} > |\bar{J}|_{\bar{g}}
\]
holds, and the energies of the two initial data sets satisfy
\[
|E-\bar{E}|
< \varepsilon.
\]
\end{theorem}

For future reference we also recall the following well-known definition.

\begin{definition}\label{defAEManifolds}
Let $(M,g)$ be a 3-dimensional Riemannian manifold. We say that $(M,g)$ is \emph{asymptotically Euclidean} if there is a compact $\mathcal{C} \subset M$ and a diffeomorphism  $\Phi: M\setminus \mathcal{C} \to \mathbb{R}^3 \setminus \overline{\mathbb{B}}_R$ such that in the coordinates $(x^1,x^2,x^3)$ induced by this diffeomorphism we have 
\begin{equation*}
|g_{ij} - \delta_{ij}|  +|x||\partial g_{ij}| + |x|^2|\partial\partial g_{ij}| = O_2(|x|^{-1}) \quad \text{as} \quad |x| \to \infty.
\end{equation*}
If the scalar curvature $\scal^g$ is integrable then the \emph{ADM mass} of the metric $g$ is defined by 
\begin{equation*}
\mathcal{M} (g) =\frac{1}{16\pi}\lim_{r\to\infty}\int_{|x|=r}\sum_{i,j=1}^3 (\partial_i g_{ij} - \partial_j g_{ii}) \frac{x^j}{|x|} \, d\mu^\delta.
\end{equation*}
If, in addition, $g$ has the following asymptotic expansion near infinity
\[
g_{ij} = \left(1+\frac{m}{2|x|}\right)^4 \delta_{ij} + O_2(|x|^{-2})  \quad \text{for $m\in\mathbb{R}$} \, \quad \text{as} \quad  |x| \to \infty
\] 
then $(M,g)$ is called \emph{asymptotically Schwarzschildean}. In this case $\mathcal{M}(g)=m.$
\end{definition}

Note that the asymptotics considered in this definition are not the most general ones, however they are sufficient for the purpose of this paper. For a more detailed treatment of asymptotically Euclidean manifolds and their mass see e.g. \cite{Bartnik} or \cite{Michel}. 
 
\subsection{The Jang equation}\label{secJf}

Let $(M,g,K)$ be an initial data set. Let $(x^1,x^2,x^3)$ be local coordinates on $M$, then we can write $g=g_{ij}dx^i\otimes dx^j$ and $K=K_{ij}dx^i\otimes dx^j$. We use the Einstein summation convention and define $g^{ij}$ by $g^{ik}g_{kj}=\delta^i_j$.  In the chosen coordinates the Jang equation reads
\begin{equation}\label{eqJang}
\left(g^{ij}-\frac{f^if^j}{1+|df|_g^2}\right)\left(\frac{\Hess^g_{ij}f}{\sqrt{1+|df|_g^2}} - K_{ij}\right)=0,
\end{equation}
where $f^i=g^{ij}f_j$ (with $f_j=\partial_j f$) are the components of the gradient and $|df|_g^2=g^{ij}f_i f_j$ is the square of its norm. Recall that the components of the second covariant derivative (or Hessian) of $f$ are computed as $\Hess^g _{ij} f=\partial_i\partial_j f-\Gamma^k_{ij}\partial _k f$,  where $ \Gamma^k_{ij}$ are the Christoffel symbols of the metric $g$ in the coordinates $(x^1,x^2,x^3)$.  

The geometric interpretation of the Jang equation is as follows. Consider a function $f:M\rightarrow \bR$.  Its graph $\Sigma\definedas\{(x,f(x)):x\in M\}$ can be seen as a submanifold in $(M\times \bR, g+dt^2)$, where $t$ is the coordinate along the $\bR$-factor, with local coordinates $(x^1,x^2,x^3)$. It is easy to check that the downward pointing unit normal of $\Sigma$ is $\nu=\frac{f^i\partial _i -\partial_t}{\sqrt{1+|df|_g^2}}$ and that the vectors $e_i = \partial_{x^i} + f_i \partial_t$ are tangent to $\Sigma$. Consequently, we may use the base coordinates $(x^1,x^2,x^3)$ to compute that the components of the induced metric on $\Sigma$ are $\gbar_{ij} =g_{ij}+\partial_i f \partial_j f$ with the inverse $\gbar^{ij}=g^{ij}-\frac{f^if^j}{1+|df|_g^2}$ defined by $\gbar^{ik} \gbar _{kj}=\delta^i_j$.  Similarly, the components of the second fundamental form are $A_{ij}=\frac{\Hess^g_{ij}f}{\sqrt{1+|df|_g^2}}$. It follows that
\begin{equation*}
H_{g}(f)\definedas H^{\Sigma}=\left(g^{ij}-\frac{f^if^j}{1+|df|_g^2}\right)\frac{\Hess^g_{ij}f}{\sqrt{1+|df|_g^2}}
\end{equation*}
is the mean curvature of $\Sigma$. Now let us extend $K$ to be a symmetric tensor on $M\times\bR$ by setting $K(\cdot,\partial_t)= 0$. Then the trace of $K$ with respect to the induced metric on $\Sigma$ is 
\begin{equation*}
\tr_{g}(K)(f)\definedas \tr^{\Sigma}K= \left(g^{ij}-\frac{f^if^j}{1+|df|_g^2}\right)K_{ij}.
\end{equation*}
We conclude that the Jang equation \eqref{eqJang} is a  prescribed mean curvature equation 
\[
H^{\Sigma}=\tr^{\Sigma}K 
\]
which we will also write as
\begin{equation*}
H_{g}(f)-\tr_{g}(K)(f)=0
\end{equation*}  
whenever we need to make reference to the graphing function.

\subsection{Preliminary considerations}\label{secHeuristic}

In this section we make an educated guess about the asymptotics of solutions of the Jang equation in the asymptotically hyperbolic setting. The existence of solutions having the desired asymptotics will be proven rigorously in Sections \ref{secBarriers}--\ref{secJangAE}. 

In \cite{Bondi} it was observed that if the initial data is taken to be the unit hyperboloid in the Minkowski spacetime, that is, if $(M,g,K) = (\bH^3,b,b)$ where $b=\frac{dr^2}{1+r^2} + r^2 \sigma$ is the hyperbolic metric, then the Jang equation \eqref{eqJang} is satisfied by the function $f(r,\theta,\varphi) = \sqrt{1 + r^2}$. Based on this observation, in the case of initial data arising as an asymptotically null slice in Bondi radiating spacetime,  it was suggested in \cite{Bondi} and \cite{HYZ} to look for solutions in the form 
\begin{equation}\label{eqAnsatzHYZ}
f(r,\theta,\varphi)=\sqrt{1+r^2}+\alpha(\theta,\varphi)\ln r +O_3(r^{-1 + \varepsilon})
\end{equation}
where $\alpha\in C^3({\mathbb{S}}^2)$ and $\varepsilon>0$. At the same time, a computation carried out in  \cite[Proposition 4.1]{HYZ} shows that this asymptotic behavior cannot be expected unless the initial data satisfies some additional conditions, see Remark \ref{HYZcomputation} below.  

In the case when $(M,g,K)$ is initial data with Wang's asymptotics (see Definition \ref{defAHdata}), the above considerations have served as motivation to look for solutions of \eqref{eqJang} with asymptotics
\begin{equation}\label{eqAnsatz}
f(r,\theta,\varphi)=\sqrt{1+r^2}+\alpha(\theta,\varphi)\ln r +\psi(\theta,\varphi)+ O_3(r^{-1 + \varepsilon})
\end{equation}
for $\alpha,\psi\in C^3({\mathbb{S}}^2)$ and $\varepsilon>0$. A lengthy but rather straightforward computation shows that in this case we have
\begin{equation}\label{eqJangEqExpansion}
\mathcal{J}(f) = \frac{\alpha + \Delta^{\mathbb{S}^2}\psi-(\frac{1}{2} \tr^\sigma \mathbf{m} + \tr^\sigma \mathbf{p})}{r^3} +\frac{\Delta^{\mathbb{S}^2}\alpha \, \ln r}{r^3}+O_1(r^{-4+\varepsilon})
\end{equation}
where $\mathcal{J}(f)$ denotes the left hand side of the Jang equation \eqref{eqJang}. As it turns out, it is possible to make the leading order terms in this expansion vanish without imposing any restrictions on the initial data $(M,g,K)$.

\begin{proposition}
If $(M,g,K)$ is asymptotically hyperbolic in the sense of Definition \ref{defAHdata}, then there exists a constant 
\begin{equation}\label{eqAlphaMass}                                                          
\alpha = \frac{1}{8\pi}\int_{\mathbb{S}^2} (\tr^\sigma \mathbf{m} + 2 \tr^\sigma \mathbf{p}) \, d\mu^\sigma = 2E
\end{equation}
and $\psi : \mathbb{S}^2\to \mathbb{R}$ such that 
\begin{equation}\label{eqPsi}
\Delta^{\mathbb{S}^2}\psi  = \tfrac{1}{2} \tr^\sigma \mathbf{m} + \tr^\sigma \mathbf{p} - \alpha, 
\end{equation}
\end{proposition}

\begin{proof} 
This follows from standard existence theory for linear elliptic equations on closed manifolds (see e.g. \cite[Section I in Appendix]{Besse}). If we define $\alpha$ by \eqref{eqAlphaMass}, then 
\begin{equation*}
\int_{\mathbb{S}^2} (\tfrac{1}{2} \tr^\sigma \mathbf{m}  + \tr^\sigma \mathbf{p} - \alpha) \, d\mu^\sigma=0,
\end{equation*}
which implies the  existence of a solution $\psi$ to \eqref{eqPsi}. Note that $\psi$ is uniquely defined up to an additive constant which is reminiscent of the fact that the Jang equation \eqref{eqJang} is invariant with respect to vertical translations $f\to f+C$, where $C$ is a constant.
\end{proof}

\begin{remark}\label{HYZcomputation}
In \cite[Section 4]{HYZ} it was suggested to seek a solution in the form \eqref{eqAnsatz} with $\psi \equiv 0$. From the above discussion it is clear that this approach might only work for initial data which satisfies the additional condition $\tfrac{1}{2} \tr^\sigma \mathbf{m} + \tr^\sigma \mathbf{p} \equiv \const$.
\end{remark}

%\section{Heuristic analysis of the Jang equation}\label{secMotivation}
%\input{motivation}

\section{Construction of barriers}\label{secBarriers}
 
In this section we construct barriers for the Jang equation \eqref{eqJang} in the case when the asymptotically hyperbolic initial data set $(M,g,K)$ has Wang's asymptotics as in Definition \ref{defAHdata}.

\begin{definition}\label{defBarriers}
We say that functions $f_+$ and $f_-$, which are locally $C^2$ on the subset $\{r\geq r_0\}\subset M$, are respectively an {\it upper} and a {\it lower barrier} for the Jang equation $\mathcal{J}(f)=0$ if 
\begin{equation}\label{eqBar1}
(\partial _r f_+)|_{r=r_0}= - \infty, \hspace{0.5cm} (\partial _r f_-)|_{r=r_0}= + \infty
\end{equation}
and 
\begin{equation}\label{eqBar2}
\hspace{1cm} \mathcal{J}(f_+)< 0, \hspace{0.5cm} \mathcal{J}(f_-)> 0 \hspace{0.5cm}\text{for } r > r_0.
\end{equation}
\end{definition}

Such functions $f_+$ and $f_-$ with prescribed asymptotic behavior at infinity will be needed for our construction of a geometric solution of the Jang equation, a hypersurface $\Sigma \subset M\times \mathbb{R}$ satisfying $H^\Sigma =\tr^\Sigma K$. In fact, in Section \ref{secBVP} and Section \ref{secExistence} we will see that near infinity $\Sigma$ is given as the graph of a function $f$ satisfying the Jang equation  \eqref{eqJang} such that $f_- \leq f \leq f_+$ on $\{r\geq r_0\}$. Our construction of barriers will ensure that $f$ behaves at infinity as 
\begin{equation}\label{eqAsSoln}
f(r,\theta,\varphi) = \sqrt{1+r^2}+\alpha\ln r +\psi(\theta,\varphi)+O(r^{-1+\varepsilon}),
\end{equation}
where $\alpha$ and $\psi$ are as in \eqref{eqAlphaMass} and \eqref{eqPsi}. 

While in the asymptotically Euclidean setting of  \cite{PMT2} the barriers with the required fall off $O(r^{-\varepsilon})$ for $\varepsilon>0$ are constructed explicitly, it appears difficult to find the functions $f_+$ and $f_-$ satisfying \eqref{eqBar1}, \eqref{eqBar2} and  \eqref{eqAsSoln} by inspection. Instead, in our construction of barriers we rely on the fact that in the spherically symmetric case there is a substitution which allows to rewrite the Jang equation as a first order ordinary differential equation, see e.g. \cite[Section 2]{MM}. The rough idea is to use this substitution and rewrite the Jang equation as an ordinary differential equation modulo correction terms and then construct sub- and supersolutions of this ordinary differential equation with prescribed boundary values on $\{r \geq r_0\}$.

More specifically, we will look for barriers in the form 
\begin{equation}\label{eqBarrierForm}
f(r,\theta,\varphi)=\phi(r)+\psi(\theta,\varphi),
\end{equation}
where $\psi$ is a solution of \eqref{eqPsi}. 
For $f$ as in \eqref{eqBarrierForm} we define (cf. \cite[Equation (4)]{MM})
\begin{equation}\label{eqK}
k(r) \definedas \frac{\phi'(r)\sqrt{1+r^2}}{\sqrt{1+(1+r^2)(\phi'(r))^2}}.
\end{equation}
Note that $-1\leq k \leq 1$, and that $k(r_0)= \pm 1$ if and only if  $\phi'(r_0)=\pm\infty$, cf. \eqref{eqBar1}.

For $f$ as in \eqref{eqBarrierForm}, we would like to rewrite the left hand side of the Jang equation $\mathcal{J}(f) = 0$ in terms of $k$. For this purpose it is convenient to introduce 
\begin{equation*}
\beta \definedas \frac{1+(1+r^2)(\phi')^2}{1+|df|_g^2}. 
\end{equation*}
Note that 
\begin{equation*}
\beta = \frac{1}{1+\frac{g^{\mu\nu}\psi_\mu\psi_\nu}{1+(1+r^2)(\phi')^2}} = 1 + O(r^{-2})
\end{equation*}
 in the sense that $|1-\beta|\leq Cr^{-2}$, where the constant $C$ does not depend on $\phi$. Set $c \definedas K-g$, then $c_{rr}=O(r^{-5})$,  $c_{r\mu}=O(r^{-3})$, $c_{\mu\nu}=O(r^{-1})$, and
\begin{equation}\label{eqTraceC}
g^{\mu\nu} c_{\mu\nu}= \frac{\tr^\sigma \mathbf{p} - \tr^\sigma \mathbf{m}}{r^3}+O(r^{-4}).
\end{equation}

\begin{lemma}\label{lemODE}
There exist constants $C_i$, $i=1,2,\ldots,8$, depending only on $(M,g,K)$ such that 
\begin{equation*}
\begin{split}
\frac{\mathcal{J}(f)}{\sqrt{1+r^2}(1+|d\psi|_g^2)\beta^{\frac{3}{2}}}\leq & k'+\frac{2}{r}\left(k-\frac{r}{\sqrt{1+r^2}}\right)-\frac{1-k^2}{\sqrt{1+r^2}}-\frac{\alpha\sqrt{1-k^2}}{r^2\sqrt{1+r^2}}\\&+\frac{C_1}{r^2}\left|\sqrt{\frac{1-k^2}{1+r^2}}-\frac{3k}{r^2}+\frac{2}{r^2}\right|+\frac{C_2}{r^2}\left|\sqrt{\frac{1-k^2}{1+r^2}}-\frac{1}{r^2}\right|\\&+\frac{C_3}{r^3}\left|k-\frac{r}{\sqrt{1+r^2}}\right| +C_4 r^{-3}|k|(1-k^2)\\&+C_5r^{-3}(1-k^2)+C_6r^{-5},
\end{split}
\end{equation*}
and
\begin{equation*}
\begin{split}
\frac{\mathcal{J}(f)}{\sqrt{1+r^2}(1+|d\psi|_g^2)\beta^{\frac{3}{2}}}\geq & k'+\frac{2}{r}\left(k-\frac{r}{\sqrt{1+r^2}}\right)-\frac{1-k^2}{\sqrt{1+r^2}}-\frac{\alpha\sqrt{1-k^2}}{r^2\sqrt{1+r^2}}\\&-\frac{C_1}{r^2}\left|\sqrt{\frac{1-k^2}{1+r^2}}-\frac{3k}{r^2}+\frac{2}{r^2}\right|-\frac{C_2}{r^2}\left|\sqrt{\frac{1-k^2}{1+r^2}}-\frac{1}{r^2}\right|\\&-\frac{C_3}{r^3}\left|k-\frac{r}{\sqrt{1+r^2}}\right| -C_4 r^{-3}|k|(1-k^2) -C_5r^{-3}(1-k^2)\\ & -C_6r^{-5}-\frac{C_7(3-k^2)}{r}\left(\left(1+C_8 r^{-2}(1-k^2)\right)^{\frac{3}{2}}-1\right)
\end{split}
\end{equation*}
holds for any $f$ as in \eqref{eqBarrierForm}.
\end{lemma}

\begin{proof}
As in Section \ref{secJf} we write $\mathcal{J}(f)= H_g(f)-\tr_g(K)(f)$ and compute the two terms in the right hand side separately. In the computations below, for all tensors the indices are lowered and raised with respect to the metric $g$, unless stated  otherwise. The Christoffel symbols of the metric $g$ can be found in Appendix \ref{appChristoffel}. We have
\begin{equation*}
\begin{split}
& \tr_g(K)(f) \\ & \quad =\left(g^{rr}-\frac{f^r f^r}{1+|df|_g^2}\right)\left(g_{rr}+c_{rr}\right)-\frac{2f^r f^\nu c_{r\nu}}{1+|df|_g^2}+\left(g^{\mu\nu}-\frac{f^\mu f^\nu}{1+|df|_g^2}\right)\left(g_{\mu\nu}+c_{\mu\nu}\right).
\end{split}
\end{equation*}
It is easy to see that the radial term is
\begin{equation*}
\begin{split}
\left(g^{rr}-\frac{f^r f^r}{1+|df|_g^2}\right)\left(g_{rr}+c_{rr}\right)&=\left(1+r^2-\frac{(1+r^2)^2(\phi')^2}{1+|df|_g^2}\right)\left(\frac{1}{1+r^2}+c_{rr}\right)\\&=\left(1-\frac{(1+r^2)(\phi')^2}{1+|df|_g^2}\right)\left(1+(1+r^2)c_{rr}\right)\\&=(1-\beta k^2 )(1+(1+r^2)c _{rr})\\%&=(1-\beta k^2)(1+r^2 c_{rr})+O(r^{-5})\\
%&=(1-k^2)(1+r^2c_{rr})+k^2(1-\beta)(1+r^2c_{rr})+O(r^{-5})\\
&=(1-k^2)(1+r^2c_{rr})+k^2(1-\beta)+O(r^{-5}).
\end{split}
\end{equation*} 
We use the fact that $1-k^2=\frac{1}{1+(1+r^2)(\phi')^2}$ and \eqref{eqTraceC} to find that the sum of the mixed terms is
\begin{equation*}
-\frac{2f^r f^\nu c_{r\nu}}{1+|df|_g^2}=-\frac{2(1+r^2)\phi' \psi^\nu c_{r\nu}\beta}{1+(1+r^2)(\phi')^2}%=-\frac{2k\sqrt{1+r^2}\, \psi^\nu c_{r\nu}\beta}{\sqrt{1+(1+r^2)(\phi')^2}}
=-2 k \sqrt{1-k^2} \sqrt{1+r^2}\, \psi^\nu c_{r\nu}\beta=O(r^{-4}),
\end{equation*} 
and that the sum of the tangential terms is
\begin{equation*}
\begin{split}
\left(g^{\mu\nu}-\frac{f^\mu f^\nu}{1+|df|_g^2}\right)\left(g_{\mu\nu}+c_{\mu\nu}\right)& =\left(g^{\mu\nu}-\frac{\psi^\mu \psi^\nu}{1+|df|_g^2}\right)\left(g_{\mu\nu}+c_{\mu\nu}\right)\\&=2+g^{\mu\nu}c_{\mu\nu}-\frac{|d\psi|_g^2}{1+|df|^2_g}-\frac{c_{\mu\nu}\psi^\mu\psi^\nu}{1+|df|_g^2}\\& = 2 + \frac{\tr^\sigma \mathbf{p} - \tr^\sigma \mathbf{m}}{r^3} -(1-\beta)+O(r^{-4}).
\end{split}
\end{equation*}
Consequently,
\begin{equation*}
\tr_g(K)(f) = (1-k^2)(1 + r^2 c_{rr}) - (1-k^2)(1-\beta) + 2 + \frac{\tr^\sigma \mathbf{p} - \tr^\sigma \mathbf{m}}{r^3} + O(r^{-4}). 
\end{equation*}
Similarly, we compute $H_g(f)$ by splitting it into the sum of the radial, mixed, and tangential terms.
To compute the radial term we note that
\begin{equation*}
\begin{split}
k'%=\frac{\frac{r}{\sqrt{1+r^2}}\,\phi'+\phi''\sqrt{1+r^2}}{\sqrt{1+(1+r^2)(\phi')^2}}-\frac{\phi'\sqrt{1+r^2}\left(r(\phi')^2+(1+r^2)\phi'\phi''\right)}{(1+(1+r^2)(\phi')^2)^{\frac{3}{2}}}
=\frac{\frac{r}{\sqrt{1+r^2}}\,\phi'+\phi''\sqrt{1+r^2}}{\left(1+(1+r^2)(\phi')^2\right)^{\frac{3}{2}}},
\end{split}
\end{equation*}
which yields
\begin{equation*}
\begin{split}
\left(g^{rr}-\frac{f^r f^r}{1+|df|^2_{g}}\right)\frac{\Hess^{g}_{rr}f}{\sqrt{1+|df|^2_{g}}}&=\left(1+r^2 -\frac{(1+r^2)^2(\phi')^2}{1+|df|_g^2}\right)\frac{\phi''+\frac{r}{1+r^2}\phi'}{\sqrt{1+|df|_g^2}}\\&=\frac{\left(1+r^2+(1+r^2)|d\psi|^2_g\right)\left(\phi''+\frac{r}{1+r^2}\phi'\right)}{(1+|df|^2_g)^{\frac{3}{2}}}\\&=\frac{\sqrt{1+r^2}(1+|d\psi|_g^2)\left(\phi''\sqrt{1+r^2}+\frac{r}{\sqrt{1+r^2}}\phi'\right)\beta^{\frac{3}{2}}}{(1+(1+r^2)(\phi')^2)^{\frac{3}{2}}}\\&=\sqrt{1+r^2}(1+|d\psi|_g^2)\beta^{\frac{3}{2}}k'.
\end{split}
\end{equation*}
As for the mixed terms, a straightforward computation shows that 
\begin{equation*}
\begin{split}
-\frac{2f^r f^{\mu}}{1+|df|^2_{g}}\frac{\Hess^{g}_{r\mu}f}{\sqrt{1+|df|^2_{g}}}&=\frac{2(1+r^2)\phi'\psi^\mu (\Gamma _{\mu r}^r\phi'+\Gamma^\nu_{\mu r}\psi_\nu)}{(1+|df|_g^2)^{\frac{3}{2}}}\\&%=\frac{ (1+r^2) \phi'\psi^\mu g^{\nu\lambda}\partial_r g_{\mu\lambda}\psi_\nu}{(1+|df|_g^2)^{\frac{3}{2}}}\\&
=\frac{(1+r^2)\phi'\partial_r g_{\mu\nu}\psi^\mu\psi^\nu}{(1+|df|_g^2)^{\frac{3}{2}}}\\&=\frac{(1+r^2)\phi'\partial_r g_{\mu\nu}\psi^\mu\psi^\nu\beta^{\frac{3}{2}}}{(1+(1+r^2)(\phi')^2)^{\frac{3}{2}}}\\&=\sqrt{1+r^2}\,k(1-k^2)\partial_r g_{\mu\nu} \psi^{\mu}\psi^{\nu} \beta^{\frac{3}{2}}.
\end{split}
\end{equation*}
Further, it is easy to check that $g^{\mu\nu}\partial_r g_{\mu\nu} = 4 r^{-1} - 3r^{-4} \tr^\sigma \mathbf{m} + O(r^{-5})$ and  that $\Delta^g \psi = r^{-2}\Delta^{\mathbb{S}^2}\psi+O(r^{-5})$.
%\begin{equation}%\label{gparg}
%\begin{split}
%g^{\mu\nu}\partial_r g_{\mu\nu}%=\left((g_0)^{\mu\nu}-b^{\mu\nu}+O(r^{-8})\right)\left(\frac{2}{r}(g_0)_{\mu\nu}-\frac{m_{\mu\nu}}{r^2}+O(r^{-3})\right)\\&=\frac{4}{r}-\frac{2(g_0)^{\mu\nu}b_{\mu\nu}}{r}-\frac{(g_0)^{\mu\nu}m_{\mu\nu}}{r^2}+O(r^{-5})\\&
%=\frac{4}{r}-\frac{3\tr_{h_0}m}{r^4}+O(r^{-5}),
%\end{split} 
%\end{equation}
%and
%\begin{equation*}
%\Delta^g \psi = r^{-2}\Delta^{\mathbb{S}^2}\psi+O(r^{-5}) 
%\end{equation*}
Hence the sum of the tangential terms is 
\begin{equation*}
\begin{split}
&\left(g^{\mu\nu}-\frac{ \psi^{\mu} \psi^{\nu}}{1+|df|^2_{g}}\right)\frac{\partial^2_{\mu\nu}\psi-\Gamma^\lambda_{\mu\nu}\psi_\lambda - \Gamma^r_{\mu\nu}\phi'}{\sqrt{1+|df|^2_{g}}}\\ &\quad =\left(g^{\mu\nu}-\frac{ \psi^{\mu} \psi^{\nu}}{1+|df|^2_{g}}\right)\frac{\Hess^g_{\mu\nu}\psi+\frac{1}{2}(1+r^2)\partial_r g_{\mu\nu}\phi'}{\sqrt{1+|df|^2_{g}}}\\ &\quad =\frac{\Delta^g\psi}{\sqrt{1+|df|^2_g}}-\frac{\psi^{\mu}\psi^\nu\Hess^g_{\mu\nu}\psi}{(1+|df|_g^2)^{\frac{3}{2}}}+\frac{g^{\mu\nu}\partial_r g_{\mu\nu}(1+r^2)\phi'}{2\sqrt{1+|df|_g^2}}-\frac{\partial_r g_{\mu\nu}\psi^\mu\psi^\nu (1+r^2)\phi'}{2(1+|df|_g^2)^{\frac{3}{2}}}\\&\quad =\frac{\Delta^g\psi \,\beta^{\frac{1}{2}}}{\sqrt{1+(1+r^2)(\phi')^2}}-\frac{\psi^{\mu}\psi^\nu\Hess^g_{\mu\nu}\psi\,\beta^{\frac{3}{2}}}{(1+(1+r^2)(\phi')^2)^{\frac{3}{2}}}+\frac{(1+r^2)\phi'g^{\mu\nu}\partial_r g_{\mu\nu}\beta^{\frac{1}{2}}}{2\sqrt{1+(1+r^2)(\phi')^2}}\\ & \qquad-\frac{(1+r^2)\phi'\partial_r g_{\mu\nu}\psi^\mu\psi^\nu\beta^{\frac{3}{2}}}{2(1+(1+r^2)(\varphi')^2)^{\frac{3}{2}}
}\\&\quad =\Delta^g \psi \sqrt{1-k^2}\,\beta^{\frac{1}{2}}-\psi^\mu\psi^\nu \Hess_{\mu\nu} ^g \psi (1-k^2)^{\frac{3}{2}}\beta^{\frac{3}{2}}+\frac{1}{2}\sqrt{1+r^2}\,k g^{\mu\nu} \partial_r g_{\mu\nu} \beta^{\frac{1}{2}}\\&\qquad-\frac{1}{2}\sqrt{1+r^2}\,k(1-k^2)\partial_r g_{\mu\nu} \psi^\mu\psi^\nu\beta^{\frac{3}{2}}\\&\quad =r^{-2}\Delta^{\mathbb{S}^2}\psi\sqrt{1-k^2}\beta^{\frac{1}{2}}+\sqrt{1+r^2}\,k \beta^{\frac{1}{2}}\left(\frac{2}{r}-\frac{3\tr^\sigma \mathbf{m}}{2r^4}\right)\\ & \qquad-\frac{1}{2}\sqrt{1+r^2}\,k(1-k^2)\partial_r g_{\mu\nu} \psi^\mu\psi^\nu\beta^{\frac{3}{2}}+O(r^{-4}).
\end{split}
\end{equation*}

Using the fact that $\Delta^{\mathbb{S}^2}\psi=\frac{1}{2}\tr^\sigma \mathbf{m} + \tr^\sigma \mathbf{p} -\alpha$ by \eqref{eqPsi}, we can now compute
\begin{equation*}
\begin{split}
&\frac{H_g(f)-\tr_g(K)(f)}{\sqrt{1+r^2}(1+|d\psi|_g^2)\beta^{\frac{3}{2}}}\\&\quad = k'+\frac{k(1-k^2)\partial_r g_{\mu\nu}\psi^\mu\psi^\nu}{2(1+|d\psi|_g^2)}+\frac{\sqrt{1-k^2}\left(\frac{1}{2}\tr^\sigma \mathbf{m} + \tr^\sigma \mathbf{p} -\alpha \right)}{r^2\sqrt{1+r^2}(1+|d\psi|_g^2)\beta}\\ & \qquad+\frac{k\left(\frac{2}{r}-\frac{3\tr^\sigma \mathbf{m}}{2r^4}\right)}{(1+|d\psi|_g^2)\beta}-\frac{(1-k^2)(1+r^2 c_{rr})}{\sqrt{1+r^2}(1+|d\psi|_g^2)\beta^{\frac{3}{2}}}+\frac{(1-k^2)(1-\beta)}{\sqrt{1+r^2}(1+|d\psi|_g^2)\beta^{\frac{3}{2}}}\\&\qquad -\frac{2+(\tr^\sigma \mathbf{p} - \tr^\sigma \mathbf{m})r^{-3}}{\sqrt{1+r^2}(1+|d\psi|_g^2)\beta^{\frac{3}{2}}}+O(r^{-5})\\&\quad =k'+\frac{1}{2}k(1-k^2)\partial_r g_{\mu\nu}\psi^\mu\psi^\nu+\sqrt{\frac{1-k^2}{1+r^2}}\left(\frac{\tr^\sigma \mathbf{m}}{2r^2} + \frac{\tr^\sigma \mathbf{p}}{r^2} - \frac{\alpha}{r^2}\right)\\ & \qquad +\frac{2k}{r(1+|d\psi|_g^2)\beta}-\frac{3k\tr^\sigma \mathbf{m}}{2r^4}-\frac{1-k^2}{\sqrt{1+r^2}(1+|d\psi|^2_g)\beta^{\frac{3}{2}}}-\frac{(1-k^2)r^2 c_{rr}}{\sqrt{1+r^2}}\\& \qquad +\frac{(1-k^2)(1-\beta)}{\sqrt{1+r^2}(1+|d\psi|_g^2)\beta^{\frac{3}{2}}}-\frac{2}{\sqrt{1+r^2}(1+|d\psi|_g^2)\beta^{\frac{3}{2}}}+
\frac{\tr^\sigma \mathbf{m} - \tr^\sigma \mathbf{p}}{r^4}+O(r^{-5}).
\end{split}
\end{equation*}
 We use the simple identities
\begin{eqnarray*}
\frac{1}{r(1+|d\psi|^2_g)\beta}&=&\frac{1}{r}-\frac{|d\psi|_g^2}{r(1+|d\psi|_g^2)}+\frac{1-\beta}{r(1+|d\psi|_g^2)\beta},
%\\
%\frac{1-k^2}{\sqrt{1+r^2}(1+|d\psi|^2_g)\beta^{\frac{3}{2}}}&=&\frac{1-k^2}{\sqrt{1+r^2}}-\frac{(1-k^2)|d\psi|_g^2}{\sqrt{1+r^2}(1+|d\psi|_g^2)}+\frac{(1-k^2)\left(1-\%beta^{\frac{3}{2}}\right)}{\sqrt{1+r^2}(1+|d\psi|_g^2)\beta^{\frac{3}{2}}},
\\
\frac{1}{\sqrt{1+r^2}(1+|d\psi|^2_g)\beta^{\frac{3}{2}}}&=&\frac{1}{\sqrt{1+r^2}}-\frac{|d\psi|_g^2}{\sqrt{1+r^2}(1+|d\psi|_g^2)}+\frac{1-\beta^{\frac{3}{2}}}{\sqrt{1+r^2}(1+|d\psi|_g^2)\beta^{\frac{3}{2}}}
\end{eqnarray*}
to rewrite this as 
\begin{equation*}
\begin{split}
\frac{H_g(f)-\tr_g(K)(f)}{\sqrt{1+r^2}(1+|d\psi|_g^2)\beta^{\frac{3}{2}}}=
& k'+\frac{2}{r}\left(k-\frac{r}{\sqrt{1+r^2}}\right)-\frac{1-k^2}{\sqrt{1+r^2}}-\frac{\alpha\sqrt{1-k^2}}{r^2\sqrt{1+r^2}}\\&+\frac{\tr^\sigma \mathbf{m}}{2r^2}\left(\sqrt{\frac{1-k^2}{1+r^2}}-\frac{3k}{r^2}+\frac{2}{r^2}\right) + \frac{\tr^\sigma \mathbf{p}}{r^2} \left(\sqrt{\frac{1-k^2}{1+r^2}} - \frac{1}{r^2}\right)\\ &-\frac{2}{r}\left(k-\frac{r}{\sqrt{1+r^2}}\right)\frac{|d\psi|_g^2}{1+|d\psi|_g^2}+\frac{2k(1-\beta)}{r(1+|d\psi|_g^2)\beta}\\&-\frac{2\left(1-\beta^{\frac{3}{2}}\right)}{\sqrt{1+r^2}(1+|d\psi|_g^2)\beta^{\frac{3}{2}}}+\frac{(1-k^2)|d\psi|_g^2}{\sqrt{1+r^2}(1+|d\psi|_g^2)}\\ &-\frac{(1-k^2)\left(1-\beta^{\frac{3}{2}}\right)}{\sqrt{1+r^2}(1+|d\psi|_g^2)\beta^{\frac{3}{2}}}+\frac{1}{2}k(1-k^2)\partial_r g_{\mu\nu}\psi^\mu\psi^\nu\\&-\frac{(1-k^2)r^2c_{rr}}{\sqrt{1+r^2}} +\frac{(1-\beta)(1-k^2)}{\sqrt{1+r^2}(1+|d\psi|_g^2)\beta^{\frac{3}{2}}}+O(r^{-5}).
\end{split}
\end{equation*}
Finally, we note that
\begin{equation*}
\frac{1-\beta}{\beta}=\frac{|d\psi|_g^2}{1+(1+r^2)(\phi')^2}=(1-k^2)|d\psi|_g^2, 
\end{equation*}
and
\begin{equation*}
\frac{1-\beta^{\frac{3}{2}}}{\beta^{\frac{3}{2}}}=\left(1+\frac{|d\psi|_g^2}{1+(1+r^2)(\phi')^2}\right)^{\frac{3}{2}}-1 =\left(1+|d\psi|_g^2(1-k^2)\right)^{\frac{3}{2}}-1,
\end{equation*}
hence 
\begin{equation}\label{eqLHSJang}
\begin{split}
\frac{H_g(f)-\tr_g(K)(f)}{\sqrt{1+r^2}(1+|d\psi|_g^2)\beta^{\frac{3}{2}}}= & k'+\frac{2}{r}\left(k-\frac{r}{\sqrt{1+r^2}}\right)-\frac{1-k^2}{\sqrt{1+r^2}}-\frac{\alpha\sqrt{1-k^2}}{r^2\sqrt{1+r^2}}\\&+\frac{\tr^\sigma \mathbf{m}}{2r^2}\left(\sqrt{\frac{1-k^2}{1+r^2}}-\frac{3k}{r^2}+\frac{2}{r^2}\right) + \frac{\tr^\sigma \mathbf{p}}{r^2} \left(\sqrt{\frac{1-k^2}{1+r^2}} - \frac{1}{r^2}\right) \\ & -\frac{2}{r}\left(k-\frac{r}{\sqrt{1+r^2}}\right)\frac{|d\psi|_g^2}{1+|d\psi|_g^2} +\frac{2k(1-k^2)|d\psi|_g^2}{r(1+|d\psi|_g^2)}\\ & -\frac{3-k^2}{\sqrt{1+r^2}}\left(\left(1+|d\psi|_g^2(1-k^2)\right)^{\frac{3}{2}}-1 \right) \frac{1}{1+|d\psi|_g^2}\\&+\frac{(1-k^2)|d\psi|_g^2}{\sqrt{1+r^2}(1+|d\psi|_g^2)}+\frac{1}{2}k(1-k^2)\partial_r g_{\mu\nu}\psi^\mu\psi^\nu \\ & -\frac{(1-k^2)r^2 c_{rr}}{\sqrt{1+r^2}}+\frac{(1-\beta)(1-k^2)}{\sqrt{1+r^2}(1+|d\psi|_g^2)\beta^{\frac{3}{2}}}+O(r^{-5}).
\end{split}
\end{equation}
Estimating the right hand side from above and from below, the result follows.
\end{proof}

Lemma \ref{lemBarriers} and Lemma \ref{lemBarrierAsympt} below concern two initial value problems whose solutions will be used to define the barriers via \eqref{eqBarrierForm} and \eqref{eqK}.  To prove these two lemmas we will need the following simple comparison result for ordinary differential equations.
\begin{lemma}\label{lemComparison}
Let $F:[r_0, +\infty)\times [-1,\, 1] \rightarrow \bR$ be continuous in both variables. If functions $l=l(r)$ and $k=k(r)$ satisfy $l'+F(r,l)<k'+F(r,k)$ and $l(r_0)\leq k(r_0)$ then $l(r)\leq k(r)$ for $r\geq r_0$.
\end{lemma}
\begin{proof}
Assume that $l(r)>k(r)$ for some $r> r_0$. Set $r_*\definedas \inf \{r> r_0: l(r)>k(r)\}$, then $r_*\geq r_0$ and $l(r_*)=k(r_*)$. But then $(l-k)'(r_*)<0$ and $(l-k)(r_*)=0$ so $(l-k)(r_*+\varepsilon)<0$ for any sufficiently small $\varepsilon>0$. %Hence $r_*$ is an isolated point of the set $\{r> r_0: l(r)>k(r)\}$, which can never happen since $l$ and $k$ are continuous. 
Since $l$ and $k$ are continuous we conclude that $l(r)\leq k(r)$ for $r\geq r_0$.
\end{proof}

\begin{lemma}\label{lemBarriers}
Let $C_i$, $i=1,2,\ldots,8$, be as in Lemma \ref{lemODE}. For any sufficiently large $r_0>0$ there exists $k_+: [r_0, +\infty) \to \bR$ such that $k_+(r_0) = -1$ and $|k_+| < 1$ for $r>r_0$ satisfying  
\begin{equation}\label{eqUpperBarrier} 
\begin{split}
 k_+'&+\frac{2}{r}\left(k_+-\frac{r}{\sqrt{1+r^2}}\right)-\frac{1-k_+^2}{\sqrt{1+r^2}}-\frac{\alpha\sqrt{1-k_+^2}}{r^2\sqrt{1+r^2}}\\&+\frac{C_1}{r^2}\left|\sqrt{\frac{1-k_+^2}{1+r^2}}-\frac{3k_+}{r^2}+\frac{2}{r^2}\right|+\frac{C_2}{r^2}\left|\sqrt{\frac{1-k_+^2}{1+r^2}}-\frac{1}{r^2}\right|+\frac{C_3}{r^3}\left|k_+-\frac{r}{\sqrt{1+r^2}}\right| \\& +C_4 r^{-3}|k_+|(1-k_+^2)+C_5r^{-3}(1-k_+^2)+C_6r^{-5} = 0.
\end{split}
\end{equation}
Similarly, there exists $k_-: [r_0, +\infty) \to \bR$ such that $k_-(r_0) = 1$ and $|k_-|<1$ for $r>r_0$ satisfying 
\begin{equation}\label{eqLowerBarrier} 
\begin{split}
k_-'&+\frac{2}{r}\left(k_--\frac{r}{\sqrt{1+r^2}}\right)-\frac{1-k_-^2}{\sqrt{1+r^2}}-\frac{\alpha\sqrt{1-k_-^2}}{r^2\sqrt{1+r^2}}\\&-\frac{C_1}{r^2}\left|\sqrt{\frac{1-k_-^2}{1+r^2}}-\frac{3k_-}{r^2}+\frac{2}{r^2}\right|-\frac{C_2}{r^2}\left|\sqrt{\frac{1-k_-^2}{1+r^2}}-\frac{1}{r^2}\right|\\&-\frac{C_3}{r^3}\left|k_--\frac{r}{\sqrt{1+r^2}}\right| -C_4 r^{-3}|k_-|(1-k_-^2) -C_5r^{-3}(1-k_-^2) -C_6r^{-5}\\&-\frac{C_7(3-k_-^2)}{r}\left(\left(1+C_8 r^{-2}(1-k_-^2)\right)^{\frac{3}{2}}-1\right)=0.
\end{split}
\end{equation}
\end{lemma}

\begin{proof}
We shall only prove the existence of $k_+$, as the same argument applies in the case of $k_-$. It is clear that $k^+_+(r)\equiv 1$ and $k^-_+(r)\equiv -1$ are respectively a super- and a subsolution of \eqref{eqUpperBarrier} provided that $r_0$ is sufficiently large. Hence by Lemma \ref{lemComparison}, and the existence theory for ordinary differential equations (see e.g. \cite[Chapter II]{Hartman})  we conclude that the solution $-1\leq k_+\leq 1$ of \eqref{eqUpperBarrier} exists for $r\geq r_0$.

Note also that our choice of $r_0$ guarantees that at a point $r^*>r_0$ where $k_+(r^*)=1$ we have $k_+'(r^*)<0$, meaning that $k_+(r^*-\varepsilon)>1$ for any sufficiently small $\varepsilon>0$, which contradicts $-1\leq k_+\leq 1$. That there are no points $r_*>r_0$ where $k_+(r_*)=-1$ is proven similarly.
\end{proof}

\begin{lemma}\label{lemBarrierAsympt}
For any sufficiently small $\varepsilon > 0$ there exists $r_0>0$ such that $k_+$ and $k_-$ as in Lemma \ref{lemBarriers} satisfy
\begin{equation}\label{eqKpm}
k_\pm(r)=\frac{r}{\sqrt{1+r^2}}+\frac{\alpha}{r^3}+O(r^{-4+\varepsilon}).
\end{equation}
\end{lemma}

\begin{proof} First, we will confirm \eqref{eqKpm} in the case of $k_+$ by gradually improving its asymptotics. Then we will briefly comment on the case of $k_-$, which is very similar.

\vspace{0.5cm}

{\it Step 1.} We will prove that $k_+(r)=1+O(r^{-2+\varepsilon})$. For a sufficiently large $r_0>1$ set $k^-_+(r)=1-\frac{2r_0^{2-\varepsilon}}{r^{2-\varepsilon}}$. Then $k^-_+(r_0)=-1$. We also have $(k^-_+)'=\frac{2(2-\varepsilon)r_0^{2-\varepsilon}}{r^{3-\varepsilon}}$,  and
\begin{equation*}
\frac{2}{r}\left(k_+^--\frac{r}{\sqrt{1+r^2}}\right)=\frac{2}{r}\left(-\frac{2r_0^{2-\varepsilon}}{r^{2-\varepsilon}}+\frac{1}{\sqrt{1+r^2}(r+\sqrt{1+r^2})}\right)\leq-\frac{4r_0^{2-\varepsilon}}{r^{3-\varepsilon}}+\frac{1}{r^3}.
\end{equation*}
It is also easy to check that
\begin{equation*}
\begin{split}
&-\frac{1-(k^-_+)^2}{\sqrt{1+r^2}}-\frac{\alpha\sqrt{1-(k^-_+)^2}}{r^2\sqrt{1+r^2}}+\frac{C_1}{r^2}\left|\frac{\sqrt{1-(k^-_+)^2}}{\sqrt{1+r^2}}-\frac{3k^-_+}{r^2}+\frac{2}{r^2}\right|\\& \quad +\frac{C_2}{r^2}\left|\sqrt{\frac{1-(k^-_+)^2}{1+r^2}}-\frac{1}{r^2}\right| +\frac{C_3}{r^3}\left|k^-_+-\frac{r}{\sqrt{1+r^2}}\right|+C_4 r^{-3}|k^-_+|\left(1-(k^-_+)^2\right)\\ & \quad  +C_5 r^{-3}(1-(k^-_+)^2)+C_6 r^{-5}<Cr^{-3},
\end{split}
\end{equation*}
where the constant $C>0$ depends only on $C_i$, $i=1,2,\ldots,6$, and does not depend on $r_0$. We conclude that $k_+^-$ is a subsolution of \eqref{eqUpperBarrier} provided that $-\frac{2\varepsilon r_0^{2-\varepsilon}}{r^{3-\varepsilon}}+\frac{C+1}{r^3}<0$ for $r\geq r_0$, which is true if $r_0\geq \sqrt{\frac{C+1}{2\varepsilon}}$. The claim follows by Lemma \ref{lemComparison} since $k_+\leq 1$.

\vspace{0.5cm}

{\it Step 2.} For a chosen $\varepsilon>0$ we fix $r_0$ as in Step 1, and prove that $k_+(r)=\frac{r}{\sqrt{1+r^2}}+O\left(r^{-3+\frac{\varepsilon}{2}}\right)$. 
Write $k_+=1+k_1$, then $k_1=O(r^{-2+\varepsilon})$ by Step 1. Then $k_+'=k_1'$, 
\begin{equation*}
\frac{2}{r}\left(k_+-\frac{r}{\sqrt{1+r^2}}\right)=\frac{2k_1}{r}+\frac{2}{r}\left(1-\frac{r}{\sqrt{1+r^2}}\right)=\frac{2k_1}{r}+\frac{1}{r^3}+O(r^{-5}),
\end{equation*} 
\begin{equation*}
\frac{1-k_+^2}{\sqrt{1+r^2}}=-\frac{2k_1+k_1^2}{\sqrt{1+r^2}}=-\frac{2k_1}{r}+O(r^{-5+2\varepsilon}),
\end{equation*}
and it is easy to check that the sum of the remaining terms in the left hand side of \eqref{eqUpperBarrier} is of order $O(r^{-4+\frac{\varepsilon}{2}})$.
%\begin{equation*}
%\begin{split}
%&-\frac{\alpha\sqrt{1-k_+^2}}{r^2\sqrt{1+r^2}}+\frac{C_1}{r^2}\left|\sqrt{\frac{1-k_+^2}{1+r^2}} -\frac{3k_+}{r^2}+\frac{2}{r^2}\right| %+\frac{C_2}{r^2}\left|\sqrt{\frac{1-k_+^2}{1+r^2}}-\frac{1}{r^2}\right|  +\frac{C_3}{r^3}\left|k_+-\frac{r}{\sqrt{1+r^2}}\right|\\ & \quad  +C_4 r^{-3}|k_+|(1-k_+^2)+C_5r^{-3}(1-k_+^2)+C_6r^{-5}= O(r^{-4+\frac{\varepsilon}{2}}).
%\end{split}
%\end{equation*}
Consequently, $k_1$ is a solution of the equation
\begin{equation*}
k_1'+\frac{4k_1}{r}+\frac{1}{r^3}=p,
\end{equation*}
where $p(r)=O(r^{-4+\frac{\varepsilon}{2}})$. Then $\left(k_1 r^4\right)'=-r+r^4 p$, and integrating from $r_0$ to $r$ we obtain 
\begin{equation*}
\begin{split}
k_1(r)&=-\frac{1}{2r^2}+r^{-4}\left(\frac{r_0^2}{2}+k_1(r_0)r_0^4\right)+r^{-4}\int_{r_0}^r s^4 p(s) ds \\ &= -\frac{1}{2r^2} +O(r^{-3+\frac{\varepsilon}{2}}).
\end{split}
\end{equation*}
%Note that there exists $C>0$ such that $|r^4 p(r)|\leq C r^{\frac{\varepsilon}{2}} $ for $r\geq r_0$. Hence 
%\begin{equation*}
%\left|\int_{r_0}^r \rho^4 p(\rho) d\rho\right|\leq C\int_{r_0}^r \rho^{\frac{\varepsilon}{2}}d\rho \leq C r^{\frac{\varepsilon}{2}}(r-r_0)+O(r^{-3+\frac{\varepsilon}{2}}). 
%\end{equation*}
%We conclude that $r^{-4}\int_{r_0}^r \rho^4 p(\rho) d\rho=O(r^{-3+\frac{\varepsilon}{2}})$. This yields that $k_1(r)=-\frac{1}{2r^2}+O(r^{-3+\frac{\varepsilon}{2}})$, 
It follows that $k_+(r)=\frac{r}{\sqrt{1+r^2}}+O\left(r^{-3+\frac{\varepsilon}{2}}\right)$.

\vspace{0.5cm}

{\it Step 3.} Finally, we prove that $k_+(r)=\frac{r}{\sqrt{1+r^2}}+\frac{\alpha}{r^3}+O(r^{-4+\frac{\varepsilon}{2}})$. By Step 2, we can write $k_+(r)=\frac{r}{\sqrt{1+r^2}}+k_2$, where $k_2(r)=O(r^{-3+\frac{\varepsilon}{2}})$. Then $k_+'=\frac{1}{(1+r^2)^{\frac{3}{2}}}+k_2'$ and $\frac{2}{r}\left(k_+-\frac{r}{\sqrt{1+r^2}}\right)=\frac{2k_2}{r}$. It is also straightforward to check that
\begin{equation*}
\frac{1-k_+^2}{\sqrt{1+r^2}}=\frac{1}{(1+r^2)^{\frac{3}{2}}}-\frac{2k_2}{r}+O(r^{-6+\frac{\varepsilon}{2}}),
\end{equation*}
and
\begin{equation*}
\frac{\alpha \sqrt{1-k_+^2}}{r^2\sqrt{1+r^2}} =\frac{\alpha}{r^4}+O\left(r^{-5+\frac{\varepsilon}{2}}\right),
\end{equation*}
while the remaining terms in the left hand side of \eqref{eqUpperBarrier} are of order $O(r^{-5+\frac{\varepsilon}{2}})$.
%\begin{equation*}
%\begin{split}
%\frac{C_1}{r^2}\left|\frac{\sqrt{1-k_+^2}}{\sqrt{1+r^2}}-\frac{3k_+}{r^2}+\frac{2}{r^2}\right|&=\frac{C_1}{r^2}\left|\frac{\sqrt{\frac{1}{1+r^2}-\frac{2rk_2}{\sqrt{1+r^2}}-k_2^2}}{\sqrt{1+r^2}}-\frac{3(1+k_2)}{r^2}+\frac{2}{r^2}\right|\\&=\frac{C_1}{r^2}\left|\frac{1+O(r^{-1+\frac{\varepsilon}{2}})}{1+r^2}-\frac{1+O(r^{-3+\frac{\varepsilon}{2}})}{r^2}\right|\\&=O(r^{-5+\frac{\varepsilon}{2}}).
%\end{split}
%\end{equation*}
%It is also easy to show that
%\begin{equation*}
%\frac{C_2}{r^3}\left|k_+-\frac{r}{\sqrt{1+r^2}}\right| +C_3 r^{-3}|k_+|(1-k_+^2)+C_4r^{-3}(1-k_+^2)+C_5r^{-5}=O(r^{-5}).
%\end{equation*}
We conclude that $k_2$ satisfies 
\begin{equation*}
k_2'+\frac{4k_2}{r}-\frac{\alpha}{r^4}=q,
\end{equation*}
where $q(r)=O(r^{-5+\frac{\varepsilon}{2}})$. Equivalently, we have $\left(k_2 r^4\right)'=\alpha+q r^4$. It follows that 
\begin{equation*}
\begin{split}
k_2(r)& =\alpha r^{-3}+\left(k_2(r_0)r_0^4-\alpha r_0\right)r^{-4}+r^{-4}\int_{r_0}^r s^4 q(s) ds \\ & = \alpha r^{-3}+O\left(r^{-4+\frac{\varepsilon}{2}}\right),
\end{split}
\end{equation*}
%Recall that there exists $C>0$ such that $\left|r^4 q(r)\right|\leq Cr^{-1+\frac{\varepsilon}{2}}$ for $r\geq r_0$. Therefore we can estimate 
%\begin{equation*}
%\begin{split}
%\left|\int_{r_0}^r \rho^4 q(\rho) d\rho\right|\leq C\int_{r_0}^r \rho^{-1+\frac{\varepsilon}{2}}d\rho = \frac{2C}{\varepsilon}\left(r^{\frac{\varepsilon}{2}}-r_0^{\frac{\varepsilon}{2}}\right) 
%\end{split}
%\end{equation*}
%Consequently, $k_2(r)=\alpha r^{-3}+O\left(r^{-4+\frac{\varepsilon}{2}}\right)$ and 
hence $k_+(r)=\frac{r}{\sqrt{1+r^2}}+\frac{\alpha}{r^3}+O(r^{-4+\frac{\varepsilon}{2}})$.
 
\vspace{0.5cm}

This argument can also be applied to prove \eqref{eqKpm} in the case of $k_-$. The only difference is that the last term on the left hand side of \eqref{eqLowerBarrier}, is not present in \eqref{eqUpperBarrier}. On Step 1, this term can be simply estimated from above by zero. On Step 2, the contribution of this term is of order $O(r^{-5+\varepsilon})$, and on Step 3 it is of order $O(r^{-5})$.
\end{proof}

\begin{proposition}\label{propBarExist}
Given $\varepsilon>0$ there exists $r_0 > 0$ and $f_+, f_-: [r_0, \infty)\to \bR$ such that 
\begin{itemize}
\item $f_+$ (respectively $f_-$) are an upper (respectively lower) barrier for the Jang equation in the sense of Definition \ref{defBarriers}.
\item  When $r\to \infty$ we have 
\begin{equation}\label{eqAsBarriers}
f_\pm(r,\theta,\varphi) = \sqrt{1+r^2}+\alpha\ln r +\psi(\theta,\varphi)+O(r^{-1+\varepsilon}).
\end{equation}
\item $f_- \leq f_+$.
\end{itemize}
\end{proposition}

\begin{proof}
Given $\varepsilon>0$ let $r_0$, $k_+$, and $k_-$ be as in Lemma \ref{lemBarriers} and Lemma \ref{lemBarrierAsympt}. Recall that $r_0>0$ was chosen so that $|k_\pm'(r_0)|>0$. Hence for some $\delta>0$ we also have $|k_\pm'(r)|>0$  on $[r_0, r_0+\delta]$, so that $ 1\pm k_\pm (r) \geq C(r-r_0)$ for some positive constant $C$ when $r \in [r_0, r_0 + \delta]$.  It follows that \eqref{eqK} or, equivalently,
\[
\phi'_\pm (r)= \frac{k_\pm(r)}{\sqrt{(1-k_\pm^2(r))(1+r^2)}}\,
\]  
defines (up to an additive constant) the continuous functions $\phi_\pm$ on $[r_0,\infty)$, which are $C^2$ for $r>r_0$. Since $\phi'_\pm(r)=1+\frac{\alpha}{r}+O(r^{-2+\varepsilon})$ and since the Jang equation is invariant with respect to vertical translations we can assume that $\phi_\pm(r)=r+\alpha \ln r +O(r^{-1+\varepsilon})$. By Lemma \ref{lemODE} and Lemma \ref{lemBarriers} the functions $f_\pm = \phi_\pm + \psi$ will satisfy \eqref{eqBar2}. Since $k_+(r_0) = -1$ and $k_-(r_0) = 1$ they will also have the property \eqref{eqBar1}. 

It only remains to show that $f_- \leq f_+$. For this we use a version of the well-known Bernstein trick as in the proof of \cite[Proposition 3]{PMT2}. Note that the difference $f_+ - f_- $ depends only on $r$ and is of order $O(r^{-1+\varepsilon})$.  Clearly, there exists a constant $L \geq 0$ such that $f_+ - f_- > - L$ for $r \geq r_0$. We denote by $L_0$  the infimum of all such constants $L$. Then we have \begin{equation}\label{eqMin1}
(f_+ - f_-)(r) \geq - L_0 \text{ for all } r \in [r_0, +\infty)
\end{equation} 
and either there exists $r_* \in [r_0, +\infty)$ such that                     
\begin{equation}\label{eqMin2}
(f_+ - f_-)(r_*) = -L_0
\end{equation}
or else
\[
\lim _{r \to \infty } (f_+ - f_-)(r) = -L_0.
\] 
In the later case we obviously must have $L_0=0$ and hence $f_+ \geq f_-$ on $\{r \geq r_0\}$. We will complete the proof by showing that the former case is not possible. If we assume that $r_* = r_0$, then by \eqref{eqMin1} and \eqref{eqMin2} it follows that $(f_+ - f_-)'(r_0) \geq 0$, which contradicts \eqref{eqBar1}.  Now suppose that $r_* > r_0$ and let $x_* \in \{r > r_0\}$ be any point such that $r(x_*) = r_*$. In this case $x_*$ is an interior minimum point for the function $f_+-f_-$. Let $(x^1,x^2,x^3)$ denote coordinates in the neighborhood of $x_*$.  Using \eqref{eqBar2} and the fact that the first order partial derivatives of $f_+$ and $f_-$ coincide at $x_*$ we obtain
\begin{equation*}
(1+|df_+|_g^2(x_*))^{-\frac{1}{2}}\left(g^{ij}(x_*)-\frac{(f_+)^i(x_*)(f_+)^j(x_*)}{1+|df_+|_g^2}\right)\frac{\partial^2 (f_+-f_-)}{\partial x^i \partial x^j}(x_*)< 0,
\end{equation*}
which contradicts the fact that  $g^{ij}(x_*)-\frac{(f_+)^i(x_*)(f_+)^j(x_*)}{1+|df_+|_g^2}$ and $\frac{\partial^2 (f_+-f_-)}{\partial x^i \partial x^j}(x_*)$ are nonnegative definite. 
\end{proof}

\section{A boundary value problem for the regularized Jang equation}\label{secBVP}
A distinctive feature of the Jang equation  $\mathcal{J}(f) = 0$ is the lack of a priori estimates for $\sup_M |f|$: in fact, the solutions may blow up for general initial data. In order to construct solutions, Schoen and Yau introduced in \cite{PMT2} the so called {\it capillarity regularization}, that is the equation $\mathcal{J}(f) = \tau f$ for $\tau>0$ for which a ($\tau$-dependent) a priori estimate is available. This section is concerned with the existence of a solution to a certain boundary value problem for the regularized equation, see Proposition \ref{propExBVP}. In Section \ref{secExistence} we will construct the so-called \emph{geometric} solution to the Jang equation by letting the regularization parameter go to zero as the domain grows in a controlled way. 

The following result has been established in \cite[Theorem 3.1]{AEM}, \cite[Lemma 2.2]{Eichmair}, \cite[Corollary 3.6]{AM}, \cite[Proposition 2]{SYBlackHole}, \cite[Section 5]{Yau}. 

\begin{theorem}\label{existenceBVP} 
Let $\Omega$ be a bounded domain in the initial data set $(M,g,K)$ with $C^{2,\alpha}$ boundary $\partial\Omega$. Let $H_{\partial \Omega}$ denote the mean curvature of $\partial \Omega$ computed as the tangential divergence of the outward unit normal to $\partial \Omega$, and let $\tr_{\partial \Omega}K$ be the trace of the restriction of $K$ to $\partial\Omega$ with respect to the induced metric on $\partial \Omega$. Suppose that
\begin{equation}\label{eqTrap}
H_{\partial \Omega}-|\tr_{\partial \Omega}K|>0.
\end{equation}
If $\tau\in (0,1)$ is sufficiently small and $\phi \in C^{2,\alpha} (\partial \Omega)$, then there exists $f\in C^{2,\alpha}(\overline{\Omega})\cap C^3(\Omega)$ such that
\begin{subequations}
\begin{eqnarray}
H_g(f)-\tr_g(K)(f)=\tau f &\text{ in }& \Omega  \label{eqBVP1}\\
f=\phi &\text{ on }& \partial\Omega \label{eqBVP2}.
\end{eqnarray} 
\end{subequations}
\end{theorem}

The proof, which we include for the sake of self-consistency, is very similar to that of \cite[Lemma 3]{PMT2} and is based on the continuity method. For $s\in [0,1]$ we consider the supplementary boundary value problem 
\begin{subequations}
\begin{eqnarray}
H_g(f_s)-s\tr_g(K)(f_s)=\tau f_s &\text{ in }& \Omega \label{eqAux1}\\
f_s=s\phi &\text{ on }& \partial \Omega \label{eqAux2}.
\end{eqnarray} 
\end{subequations}
The first step is to obtain uniform a priori estimates for the solutions.

\begin{lemma}\label{lemApriori}
Let $\Omega$ and $\phi$ be as in Theorem \ref{existenceBVP}. Suppose that $\tau>0$ is sufficiently small and $f_s\in C^{2,\alpha}(\overline{\Omega})$ satisfies \eqref{eqAux1}-\eqref{eqAux2}. Then there exists a constant $C$ depending only on $\alpha$, $\tau$, $\phi$, $\Omega$, and the initial data $(M,g,K)$, such that $\left\|f_s\right\|_{C^{2,\alpha}(\overline{\Omega})}\leq C$.
\end{lemma}
\begin{proof} The proof is divided into the following steps.

\vspace{0.5cm}

1) {\it $C^0$ bound for $f_s$.} Suppose that $f_s$  attains its maximum at an interior point $p\in \Omega$, then from \eqref{eqAux1} 
it follows that 
\begin{equation*}
%\begin{split}
\tau f_s(p) = \left(g^{ij} \Hess_{ij} (f_s)\right)(p)-s \left(\tr^g K\right)(p)\\ \leq-s\left(\tr^g K\right)(p)\\ \leq \max_{\Omega} \left|\tr^g K\right|.
%\end{split}
\end{equation*}
Similarly, if $q\in \Omega$ is an interior minimum point we have  
$\tau f_s(q)%\frac{g^{ij}(q) \left(\Hess_{ij} (f_s)\right)(q)}{\sqrt{1+|df_s|_g^2(q)}}-s \left(\tr_g K\right)(q)\geq-s\left(\tr_g K\right)(q)
\geq-\max_{\Omega} \left|\tr^g K\right|,$
thus
\begin{equation*}
\tau |f_s|<\mu_1\definedas\max\left\{\max_{\Omega} \left|\tr^g K\right|,\max_{\partial\Omega}|\phi|\right\} \hspace{7pt}\text{ on } \hspace{7pt} \overline{\Omega}.
\end{equation*}

\vspace{0.5cm}

2) {\it Interior gradient estimates for $f_s$.} It is straightforward to check that 
\begin{equation*}
H_g(f_s)=\nabla_i \left(\frac{(f_s)^i}{\sqrt{1+|df_s|_g^2}}\right).
\end{equation*}
Applying the covariant derivative $\nabla_k$ to the both sides of \eqref{eqAux1} and commuting the covariant derivatives we thereby obtain 
\begin{equation*}
\begin{split}
\tau (f_s)_k & =\nabla_k \left(H_g(f_s)\right)-s\nabla_k \left(\tr_g(K)(f_s)\right)\\ & = \nabla_i\left(\left(g^{ij}-\frac{(f_s)^i (f_s)^j}{1+|df_s|_g^2}\right)\frac{\Hess_{jk}(f_s)}{\sqrt{1+|df_s|_g^2}}\right)-\frac{\ricdd ik(f_s)^i}{\sqrt{1+|df_s|_g^2}}\\ & \quad +\frac{2 s K_{il} (f_s)^l \Hess_{kj}(f_s)}{1+|df_s|_g^2}\left(g^{ij}-\frac{(f_s)^i (f_s)^j}{1+|df_s|_g^2}\right) \\ & \quad - s \left(g^{ij}-\frac{(f_s)^i(f_s)^j}{1+|df_s|_g^2} \right)\nabla_k K_{ij}.
\end{split}
\end{equation*}
As a consequence, we have
\begin{equation*}
\begin{split}
\tau|d f_s|_g^2=&(f_s)^k\nabla_i\left(\left(g^{ij}-\frac{(f_s)^i (f_s)^j}{1+|df_s|^2_g}\right)\frac{\Hess_{jk}(f_s)}{\sqrt{1+|df_s|^2_g}}\right)- \frac{\ricdd ik(f_s)^i(f_s)^k}{\sqrt{1+|df_s|^2_g}}\\&+\frac{2 s K_{il} (f_s)^l (f_s)^k\Hess_{kj}(f_s)}{1+|df_s|_g^2}\left(g^{ij}-\frac{(f_s)^i (f_s)^j}{1+|df_s|_g^2}\right)\\ & -s \left(g^{ij}-\frac{(f_s)^i(f_s)^j}{1+|df_s|_g^2} \right)\nabla_k K_{ij}(f_s)^k.
\end{split}
\end{equation*} 
Let $u_s = |df_s|_g^2$. Then
\begin{equation*}
- \frac{\ricdd ik(f_s)^i(f_s)^k}{\sqrt{1+|df_s|^2_g}}\leq \frac{|\ric|_g|df_s|_g^2}{\sqrt{1+|df_s|_g^2}}\leq C_1 (u_s)^{\frac{1}{2}}
\end{equation*}
and
\begin{equation*}
\begin{split}
\left(g^{ij}-\frac{(f_s)^i (f_s)^j}{1+|df_s|_g^2}\right)\frac{2 K_{il} (f_s)^l (f_s)^k\Hess_{kj}(f_s)}{1+|df_s|_g^2}=\left(g^{ij}-\frac{(f_s)^i (f_s)^j}{1+|df_s|_g^2}\right)\frac{K_{il} (f_s)^l \nabla_j (u_s)}{1+|df_s|_g^2}. 
\end{split}
\end{equation*}
%, then $\nabla_i u_s=2(f_s)^k \Hess_{ik} (f_s)$. We now estimate the right hand side of \eqref{covarksummed} from above in order to rewrite it as a differential inequality for $u_s$. 
We also have
\begin{equation*}
\begin{split}
&(f_s)^k\nabla_i\left(\left(g^{ij}-\frac{(f_s)^i (f_s)^j}{1+|df_s|^2_g}\right)\frac{\Hess_{jk}(f_s)}{\sqrt{1+|df_s|^2_g}}\right)\\&=
\nabla_i\left(\left(g^{ij}-\frac{(f_s)^i (f_s)^j}{1+|df_s|_g^2}\right)\frac{(f_s)^k \Hess_{jk}(f_s)}{\sqrt{1+|df_s|_g^2}}\right)\\ & \quad -\left(g^{ij}-\frac{(f_s)^i(f_s)^j}{1+|df_s|_g^2}\right)\frac{g^{kl}\Hess_{jk}(f_s)\Hess_{il}(f_s)}{\sqrt{1+|df_s|_g^2}}\\
&=\frac{1}{2}\nabla_i\left(\left(g^{ij}-\frac{(f_s)^i (f_s)^j}{1+|df_s|_g^2}\right)\frac{\nabla_j\left(|df_s|_g^2\right)}{\sqrt{1+|df_s|_g^2}}\right)\\& \quad -\left(g^{ij}-\frac{(f_s)^i(f_s)^j}{1+|df_s|_g^2}\right)\frac{(f_s)^k(f_s)^l}{1+|df_s|_g^2}\frac{\Hess_{jk} (f_s) \Hess_{il}(f_s)}{\sqrt{1+|df_s|_g^2}}\\&\hspace{12pt}-\left(g^{ij}-\frac{(f_s)^i(f_s)^j}{1+|df_s|_g^2}\right)\left(g^{kl}-\frac{(f_s)^k(f_s)^l}{1+|df_s|_g^2}\right)\frac{\Hess_{jk}(f_s)\Hess_{il}(f_s)}{\sqrt{1+|df_s|_g^2}}\\
&= \frac{1}{2}\nabla_i\left(\frac{\gbar^{ij}\nabla_j (u_s)}{\sqrt{1+|df_s|_g^2}}\right)-\frac{|d(u_s)|^2_{\gbar}}{4(1+|df_s|_g^2)^{\frac{3}{2}}}-\frac{|\Hess (f_s)|^2_{\gbar}}{\sqrt{1+|df_s|_g^2}} \\
&\leq  \frac{1}{2}\nabla_i\left(\frac{\gbar^{ij}\nabla_j (u_s)}{\sqrt{1+|df_s|_g^2}}\right)
\end{split}
\end{equation*}
where $\gbar$ is the metric induced on the graph of the function $f_s : \Omega \to \bR$ in the product manifold $(M\times \bR, g+dt^2)$, cf. Section \ref{secJf}. Finally, we can estimate 
\begin{equation*}
-\left(g^{ij}-\frac{(f_s)^i(f_s)^j}{1+|df_s|_g^2} \right)s \nabla_k K_{ij}(f_s)^k\leq |\gbar^{ij} \nabla K_{ij}|_g |df_s|_g\leq C_2 (u_s)^{\frac{1}{2}}.
\end{equation*}
We conclude that $u_s=|df_s|_g^2$ satisfies the differential inequality
\begin{equation*}
\nabla_i(A^{ij}\nabla_j (u_s))+B^j \nabla_j (u_s)+C (u_s)^{\frac{1}{2}}\geq \tau u_s,
\end{equation*}
where $A^{ij}=\frac{\gbar^{ij}}{2\sqrt{1+|df_s|_g^2}}$ is nonnegative definite, $B^j=\frac{s\gbar^{ij}K_{il} (f_s)^l}{1+|df_s|_g^2}$ is bounded, and $C>0$ is a constant that only depends on the initial data $(M,g,K)$. If $u_s$ attains its maximum at an interior point $p\in \Omega$, then the above inequality implies that $C (u_s(p))^{\frac{1}{2}}\geq \tau u_s(p)$. Recalling the definition of $u_s$ we conclude that $\tau |df_s|_g( p)\leq \mu_2$  where $\mu_2$ depends only on the initial data $(M,g,K)$.

\vspace{0.5cm}

3) {\it Boundary gradient estimates.} The bounds for $|df_s|_g$ restricted to $\partial \Omega$ can be obtained by means of the so-called barrier method. This method is  described in \cite[Chapter 14]{GT}, and its application to the boundary value problem \eqref{eqAux1}-\eqref{eqAux2} is summarized in Appendix \ref{appBarMet}.

Since \eqref{eqTrap} holds,  by choosing $\tau>0$ to be sufficiently small, we may ensure that $H_{\partial \Omega}-|\tr_{\partial \Omega} K|-\tau |\phi|>0$. Using the function $\rho=\dist(\cdot,\partial \Omega)$ we can foliate a neighborhood $U$ of $\partial \Omega$ by the hypersurfaces $\mathcal{E}_\rho$ of constant $\rho$. If $\{x^1,x^2\}$ are coordinates on $\partial \Omega$ then $(\rho, x^1, x^2)$ are coordinates on $U$, and we can write the metric on $U$ as $g=d\rho^2+g_{\rho}$, where $g_\rho$ is the induced metric on $\mathcal{E}_{\rho}$. From now on it will be assumed that $U=\{0\leq \rho < \rho_0\}$, where $\rho_0>0$ is as small as to ensure that 
\begin{equation}\label{eqSupertrap}
H_{\mathcal{E}_\rho}-|\tr_{\mathcal{E}_\rho}K|-\tau|\phi|>0
\end{equation}
holds for any $\rho\in [0,\rho_0)$. 

We will show that for a sufficiently large constant $B>0$ the functions $\fbar=s \phi +B \rho$ and $\funder=s\phi -B\rho$ are boundary barriers for \eqref{eqAux1}-\eqref{eqAux2}, in the sense that they satisfy the conditions of Proposition \ref{propBoundBar}. The mean curvature of the hypersurfaces $\mathcal{E}_\rho$ computed with respect to the normal $\partial _\rho$ (chosen so that the orientations of $\partial \Omega$ and $\mathcal{E}_\rho$ agree) is
\begin{equation*}
H_{\mathcal{E}_\rho}=(g_{\rho})^{\mu\nu}\left(A_{\mathcal{E}_\rho}\right)_{\mu\nu}=(g_{\rho})^{\mu\nu}g((\nabla_\mu \partial_\nu)|_{\mathcal{E}_\rho},\partial_\rho)=(g_{\rho})^{\mu\nu}g(\Gamma^\rho_{\mu\nu}\partial_\rho,\partial_\rho)=(g_{\rho})^{\mu\nu}\Gamma_{\mu\nu}^{\rho}.
\end{equation*}
 Using the fact that $\Gamma_{\rho\rho}^{\rho}=\Gamma_{\rho\rho}^\mu=\Gamma_{\rho\mu}^\rho=0$, one computes
\begin{equation*}
\begin{split}
& H_g(s\phi\pm B\rho)\\ & \quad =\pm\frac{2B s^2\phi^\mu \Gamma_{\rho\mu}^\nu\phi_\nu}{(1+B^2+s^2|d\phi|^2_{g_\rho})^{\frac{3}{2}}}+\left((g_\rho)^{\mu\nu}-\frac{s^2\phi^\mu \phi^\nu}{1+B^2+s^2|d\phi|^2_{g_\rho}}\right)\frac{s\Hess_{\mu\nu}\phi\mp B\Gamma^\rho_{\mu\nu}}{\sqrt{1+B^2+s^2|d\phi|^2_{g_\rho}}} \\ & \quad = \mp H_{\mathcal{E}_\rho}+O(B^{-1}),
\end{split}
\end{equation*}
and 
\begin{equation*}
\begin{split}
\tr_g(K)(s\phi\pm B\rho) = & \left(1-\frac{B^2}{1+B^2+s^2|d\phi|^2_{g_\rho}}\right)K_{\rho\rho}\mp\frac{2Bs \phi^{\mu}K_{\rho\mu}}{1+B^2+s^2|d\phi|^2_{g_\rho}}\\ & \quad +(g_{\rho})^{\mu\nu}K_{\mu\nu} -\frac{s^2\phi^\mu\phi^\nu K_{\mu\nu}}{1+B^2+s^2|d\phi|^2_{g_\rho}}\\= & \tr_{\mathcal{E}_\rho}K+O(B^{-1}).
\end{split}
\end{equation*}
Consequently, in the view of \eqref{eqSupertrap}, we have 
\begin{equation*}
\begin{split}
H_g(\fbar)-s\tr_g (K)(\fbar)-t\fbar&=-H_{\mathcal{E}_\rho}-s\tr_{\mathcal{E}_\rho} K-\tau s\phi-\tau B\rho+O(B^{-1})\\&<-H_{\mathcal{E}_\rho}+\left|\tr_{\mathcal{E}_\rho} K \right|+\tau |\phi|+CB^{-1}<0,
\end{split}
\end{equation*}
and 
\begin{equation*}
\begin{split}
H_g(\funder)-s\tr_g (K)(\funder)-t\funder&=H_{\mathcal{E}_\rho}-s\tr_{\mathcal{E}_\rho} K-\tau s\phi+\tau B\rho+O(B^{-1})\\&>H_{\mathcal{E}_\rho}-\left|\tr_{\mathcal{E}_\rho} K \right|-\tau |\phi|-CB^{-1}>0,
\end{split}
\end{equation*}
for any $0\leq\rho<\rho_0$, provided that $B>0$ is large enough. Finally, recall that the functions $f_s$ are uniformly bounded in $C^0$ norm and satisfy \eqref{eqAux2}. Hence, by increasing $B$ if needed, we can ensure that $\funder<f_s <\fbar$ holds on $\mathcal{E}_{\rho_0}$. 

Since the first order partial derivatives of $\fbar$ and $\funder$ in $U$ are bounded by a constant independent of $s$, by Proposition \ref{propBoundBar} there exists a constant $\mu_3>0$ such that the uniform estimate $|df_s|_g<\mu_3$ holds on $\partial \Omega$.

4) {\it $C^{2,\alpha}$ bounds on $f_s$.} Let $x=(x^1,x^2,x^3)$ be coordinates on $\Omega$. We may write \eqref{eqAux1} as 
\begin{equation}\label{eqQuasilinear}
a^{ij}(x,Df_s) \partial^2_{ij}f_s+b(x,f_s,Df_s)=0,
\end{equation} 
where $Df_s$ denotes the Euclidean gradient of $f_s$,  and 
\begin{equation*}
\begin{split}
a^{ij}(x,Df_s)&=\left(g^{ij}-\frac{g^{ik}g^{jl}(f_s)_k(f_s)_l}{1+g^{kl}(f_s)_k(f_s)_l}\right)\frac{1}{\sqrt{1+g^{kl}(f_s)_k(f_s)_l}},\\
b(x,f_s,Df_s)&=-\left(g^{ij}-\frac{g^{ik}g^{jl}(f_s)_k(f_s)_l}{1+g^{kl}(f_s)_k(f_s)_l}\right)\left(\frac{\Gamma_{ij}^k(f_s)_k}{\sqrt{1+g^{kl}(f_s)_k(f_s)_l}}+sK_{ij}\right)-\tau f_s.
\end{split}
\end{equation*}
 Note that 2) and 3) imply that $|df_s|_g \leq \max \{\frac{\mu_2}{\tau},\mu_3 \}$. Suppose that $K_\tau$ is a positive constant depending on $\tau$ such that $\sup_{\Omega}|f_s|+\sup_{\Omega}|D f_s|<K_\tau$. Then the differential operator in the left hand side of \eqref{eqQuasilinear} is strictly elliptic with uniform ellipticity constant $\lambda_{K_\tau}$ for all $s$. It is also obvious that we can choose a constant $\mu_{K_\tau}$ so that 
\begin{equation*}
|a^{ij}(x,p)|+|\partial_{x^k}a^{ij}(x,p)|+|\partial_{p^k} a^{ij}(x,p)|+|b(x,z,p)|\leq \mu_{K_\tau}
\end{equation*} 
for $x\in\Omega$, $|z|+|p|<K_\tau$, and $k=1,2,3$.
From the fundamental global H\"older estimate of Ladyzhenskaya and Ural'tseva \cite[Theorem 13.7]{GT} we conclude that there exists $\beta=\beta(K_\tau,\mu_{K_\tau}/\lambda_{K_\tau},\Omega)$ such that $Df_s$ is bounded in $C^{0,\beta}$ norm by a constant $C=C(K_\tau,\mu_{K_\tau}/\lambda_{K_\tau},\Omega, \Phi)$, where $\Phi$ is a $C^{2,\alpha}$ norm of $\phi$. That is, $|Df_s|_{C^{0,\beta}}<C$ uniformly in $s$.

We are now in a position to treat the Jang equation as a linear elliptic equation for $f_s\in C^{2,\alpha}(\overline{\Omega})$, namely,
\begin{equation}\label{eqJangLinEll}
\alpha^{ij}\partial^2_{ij}f_s+\beta^i \partial_i f_s+\gamma f_s=F,
\end{equation}
where the coefficients 
\begin{eqnarray*}
\alpha^{ij}&=&\left(g^{ij}-\frac{(f_s)^i(f_s)^j}{1+|df_s|_g^2}\right)\frac{1}{\sqrt{1+|df_s|_g^2}},\\
\beta^k&=&-\left(g^{ij}-\frac{(f_s)^i(f_s)^j}{1+|df_s|_g^2}\right)\frac{\Gamma_{ij}^k}{\sqrt{1+|df_s|_g^2}},\\
\gamma&=&-\tau,\\
F&=&sK_{ij}\left(g^{ij}-\frac{(f_s)^i(f_s)^j}{1+|df_s|_g^2}\right)
\end{eqnarray*}
are uniformly bounded in $C^{0,\beta}(\overline{\Omega})$. Applying \cite[Theorem 6.6]{GT} we deduce that $f_s$ are uniformly bounded in $C^{2,\beta}(\overline{\Omega})$. Then $f_s$ are uniformly bounded in $C^{1,\alpha}(\overline{\Omega})$. One more application of  \cite[Theorem 6.6]{GT} completes the proof.
\end{proof} 

\begin{proof}[Proof of Theorem \ref{existenceBVP}]
The proof is very similar to \cite[Lemma 3]{PMT2} and consists in applying the continuity method to \eqref{eqAux1}-\eqref{eqAux2}. Let $S$ be the set of $s\in[0,1]$ such that \eqref{eqAux1}-\eqref{eqAux2} has a solution $f_s\in C^{2,\alpha}(\overline{\Omega})$. Clearly, $S$ is non-empty, since $0 \in S$. Hence if we show that $S$ is both open and closed in $[0,1]$ it will follow that $S=[0,1]$. 

That $S$ is closed is an immediate consequence of Lemma \ref{lemApriori}. Suppose that $s_n\in S$ for $n=1,2,\ldots$ are such that $s_n\rightarrow s$ as $n\rightarrow \infty$. By Lemma \ref{lemApriori} we have a uniform bound $\|f_{s_n}\|_{C^{2,\alpha}(\overline{\Omega})}\leq C$. Hence by Arzela-Ascoli theorem there is a subsequence of $f_{s_n}$ which converges uniformly along with its first and second derivatives to a limit $f_s$. Thus $s\in S$, so $S$ is closed.

That $S$ is open will follow from implicit function theorem. Consider a $C^1$ map 
\begin{equation*}
T:C^{2,\alpha}(\overline{\Omega})\times \bR\rightarrow C^{0,\alpha}(\overline{\Omega})\times C^{2,\alpha}(\partial \Omega) \times \bR
\end{equation*}
defined by
\begin{equation*}
T(f,s)=(H_g(f)-s\tr_g(K)(f)-\tau f, f|_{\partial \Omega}-s\phi,s).
\end{equation*} 
Suppose that $s_0\in S$ and that $f_0\in C^{2,\alpha}(\overline{\Omega})$ is the respective solution of  \eqref{eqAux1}-\eqref{eqAux2}. The linearization of $T$ at $(f_0,s_0)$ is a map  
\begin{equation}\label{eqLinearization}
\cL_{(f_0,s_0)} T:  C^{2,\alpha}(\overline{\Omega})\times \bR\rightarrow C^{0,\alpha}(\overline{\Omega})\times C^{2,\alpha}(\partial \Omega) \times \bR
\end{equation}
given by
\begin{equation*}
\cL_{(f_0,s_0)} T\,(\eta,\varsigma)=\left(\cL_{(f_0,s_0)}\left(H_g(f)-s\tr_g(K)(f)-\tau f\right)(\eta,\varsigma),\eta|_{\partial \Omega}-\varsigma \phi,\varsigma \right).
\end{equation*}
It is straightforward to check that
\begin{multline*}
\cL _{(f_0,s_0)}\left(H_g(f)-s\tr_g(K)(f)\right)(\eta,\varsigma) \\=G^{ij} \Hess_{ij}\eta+\left(\nabla_j G^{kj}+2s_0G^{kj}(f_0)^iK_{ij}(1+|df_0|_g^2)^{-\frac{1}{2}}\right)\eta_k-\varsigma \tr_g(K)(f_0) 
\end{multline*}
where 
\begin{equation*}
G^{ij}= \left(g^{ij}-\frac{(f_0)^i(f_0)^j}{1+|df_0|_g^2}\right)\frac{1}{\sqrt{1+|df_0|_g^2}}.
\end{equation*}
%Hence
%\begin{multline*}
%\cL _{(f_0,s_0)}T\,(\eta,\varsigma)\\=\left(G^{ij} \Hess_{ij}\eta+\left(\nabla_j (G^{kj})+2s_0G^{kj}(f_0)^iK_{ij}(1+|df_0|_g^2)^{-\frac{1}{2}}\right)\eta_k-\varsigma %\tr_g(K)(f_0)-\tau \eta,\eta|_{\partial \Omega}-\varsigma \phi,\varsigma \right).
%\end{multline*}
By standard theory for linear elliptic equations (see e.g. \cite[Theorem 6.14]{GT}), for any $\varsigma\in [0,1]$, $\Xi\in C^{0,\alpha}(\overline{\Omega})$, and $\xi \in C^{2,\alpha}(\partial \Omega)$ there exists a unique solution $\eta \in C^{2,\alpha}(\overline{\Omega})$ to the boundary value problem
\begin{eqnarray*}
G^{ij}\partial^2_{ij}\eta+H^k\partial_k\eta-\tau\eta  = & \Xi + \varsigma \tr_g(K)(f_0)  \quad &\text{in} \quad \Omega \\
\eta  = & \varsigma \phi+\xi \quad &\text{on} \quad \partial \Omega
\end{eqnarray*}
where 
\begin{equation*}
H^k = \nabla_j G^{kj}+2s_0 G^{kj}(f_0)^iK_{ij}(1+|df_0|_g^2)^{-\frac{1}{2}} - G_{ij} \Gamma^k_{ij}.
\end{equation*}
Consequently, the map \eqref{eqLinearization} is an isomorphism and by the implicit function theorem there is an interval $I=(s_0-\delta,s_0+\delta)\subseteq [0,1]$ such that for every $s\in I$ there is a solution $f_s \in C^{2,\alpha}(\overline{\Omega})$ to the boundary value problem \eqref{eqAux1}-\eqref{eqAux2}. Hence $S$ is open.

The existence of $f \in C^{2,\alpha}(\overline{\Omega})$ satisfying \eqref{eqBVP1}-\eqref{eqBVP2} is thereby proven.  That $f\in C^{3,\alpha}(\Omega) $ follows at once by interior Schauder estimates (see e.g. \cite[Theorem 6.17]{GT}) applied to \eqref{eqJangLinEll} with $s=1$.
\end{proof}

The following elementary lemma provides an example of a domain $\Omega$ in an asymptotically hyperbolic initial data set $(M,g,K)$ such that the condition  \eqref{eqTrap} is satisfied. 

\begin{lemma}\label{lemLargeSphere}
Let $(M,g,K)$ be an asymptotically hyperbolic initial data set in the sense of Definition \ref{defAHdata}. If $R>0$ is sufficiently large, then $H_{\partial\mathcal{B}_R}-|\tr_{\partial\mathcal{B}_R}K|>0$ holds for $\mathcal{B}_R=M \setminus \{r \geq R\}$.
\end{lemma}

\begin{proof}
A computation shows that the mean curvature of $\partial \mathcal{B}_R$ is
\begin{equation*}
\begin{split}
H_{\partial\mathcal{B}_R}%&=g^{\mu\nu} |_{\partial\mathcal{B}_R} g\left((\nabla_\mu \partial_\nu)|_{\partial \mathcal{B}_R}, -\sqrt{1+R^2}\partial_r \right) \\ 
%&= -(1+R^2)^{-\frac{1}{2}} (g^{\mu\nu} \Gamma^r_{\mu\nu})|_{\partial\mathcal{B}_R}\\
%&= \tfrac{1}{2} \sqrt{1+R^2} (g^{\mu\nu} \partial_r g_{\mu\nu}) |_{\partial\mathcal{B}_R} \\
%&= \tfrac{1}{2} (R + \tfrac{1}{2} R^{-1} + O(R^{-3}))(4R^{-1} + O(R^{-4})) \\
%&
= 2 + R^{-2} + O(R^{-3})
\end{split}
\end{equation*}
and
%\begin{equation*}
%\begin{split}
%A_{\mu\nu}&=g\left((\nabla_\mu \partial _\nu\right) |_{\partial\mathcal{B}_R}, n_R)\\&=-g\left(\Gamma_{\mu\nu}^r |_{\partial\mathcal{B}_R}\partial _r, \sqrt{1+R^2}\,\partial_r\right)\\&=-\frac{1}{\sqrt{1+R^2}}\Gamma_{\mu\nu}^r|_{\partial\mathcal{B}_R}\\&=\frac{1}{2}\sqrt{1+R^2}(\partial_r g_{\mu\nu})|_{\partial\mathcal{B}_R}.
%\end{split}
%\end{equation*}
%Consequently, the mean curvature of $\partial\mathcal{B}_R$ with respect to $n_R$ is equal to
%\begin{equation*}
%\begin{split}
%H_{\partial\mathcal{B}_R}&=g^{\mu\nu}|_{\partial\mathcal{B}_R} A_{\mu\nu}\\&=\frac{1}{2}\sqrt{1+R^2}(g^{\mu\nu}\partial_r(g_{\mu\nu}))|_{\partial\mathcal{B}_R}\\&=\frac{1}{2}\left(R+\frac{1}{2R}+O(R^{-3})\right)\left(\frac{4}{R}-\frac{3\tr_{h_0}m}{R^4}+O(R^{-5})\right)\\&=2+R^{-2}+O(R^{-3}),
%\end{split}
%\end{equation*}
\begin{equation*}
\begin{split}
\tr_{\partial\mathcal{B}_R} K%=\left(g^{\mu\nu}K_{\mu\nu}\right)|_{\partial\mathcal{B}_R}
%&=g^{\mu\nu}(g_{\mu\nu}+\ctil_{\mu\nu})|_{\partial\mathcal{B}_R}\\&
%=2+g^{\mu\nu}\ctil_{\mu\nu}|_{\partial\mathcal{B}_R}\\&
=2+O(R^{-3}),
\end{split}
\end{equation*}
hence $H_{\partial\mathcal{B}_R}-|\tr_{\partial\mathcal{B}_R}K|>0$ for a sufficiently large $R>0$.
\end{proof}

Combining Theorem \ref{existenceBVP} and Lemma \ref{lemLargeSphere} we have

\begin{proposition}\label{propExBVP}
Let $(M,g,K)$ be an asymptotically hyperbolic initial data set with Wang's asymptotics as in Definition \ref{defAHdata}. Let $f_-,f_+:\{r\geq r_0\}\rightarrow \bR$ be the barrier functions as in Proposition \ref{propBarExist}. Given a sufficiently large $R>r_0$ and a sufficiently small $\tau \in (0,1)$, for any $\phi_R\in C^{2,\alpha}(\partial\mathcal{B}_R)$ such that $f_-\leq\phi_R\leq f_+$ on $\partial\mathcal{B}_R$, there exists a solution $f\in C^{2,\alpha}(\overline{\mathcal{B}}_R)\cap C^3(\mathcal{B}_R)$ to the boundary value problem 
\begin{subequations}
\begin{eqnarray}
H_g(f)-\tr_g(K)(f)=\tau f &\text{ in }& \mathcal{B}_R \label{eqBall1} \\
f= \phi_{R} &\text{ on }& \partial\mathcal{B}_R \label{eqBall2}
\end{eqnarray} 
\end{subequations}
such that $f_-\leq f \leq f_+$ on $\{r_0\leq r\leq R\}$.
\end{proposition}

\begin{proof}
The existence of $f\in C^{2,\alpha}(\overline{\mathcal{B}}_R)\cap C^3(\mathcal{B}_R)$ follows from Theorem \ref{existenceBVP} and Lemma \ref{lemLargeSphere}, so we only need to confirm that $f_-\leq f \leq f_+$ holds on $\{r_0\leq r\leq R\}$.
Note that $f_+$ and $f_-$ satisfy \eqref{eqBar2} and are bounded on $\{r_0\leq r\leq R\}$. Hence by assuming that $\tau>0$ is sufficiently  small we can ensure that
\begin{subequations}
\begin{eqnarray}
H_g(f_+)-\tr_g(K)(f_+)-\tau f_+&<& 0 \label{eqBarUp}, \\
H_g(f_-)-\tr_g(K)(f_-)-\tau f_-&>& 0 \label{eqBarLow}.
\end{eqnarray} 
\end{subequations}
Combining \eqref{eqBall1} with \eqref{eqBarUp}-\eqref{eqBarLow} we may argue as in the proof of Proposition \ref{propBarExist} (see also \cite[Proposition 3]{PMT2}) to show that $f_- \leq f \leq f_+$ on $\{r_0\leq r\leq R\}$. 
%Since $f$ is bounded on $\mathcal{B}_R$ by Lemma \ref{apriori}, we can find a nonnegative constant $L$ such that $f_+ + L>f$ on $\{r_0\leq r\leq R\}$. Let $L_0$  be the infimum of all such constants $L$, then $f_+(x_0)+L_0=f(x_0)$ for some $x_0 \in \{r_0\leq r \leq R\}$. If $x_0\in\partial\mathcal{B}_R$, then from $f(x_0)\leq f_+(x_0)$ it is clear that $L_0=0$, so $f\leq f_+$ for $r\geq r_0$. If we assume that $x_0\in \partial\mathcal{B}_{r_0}$, then by $f_+(x_0)+L_0=f(x_0)$ and $f_+(x)+L_0\geq f(x)$ on $\{r_0\leq r\leq R\}$, we conclude that $(\partial_r f)(x_0)=-\infty$, a contradiction.  It remains to consider the situation when $x_0\in \{r_0<r<R\}$. In this case $x_0$ is an interior minimum point for the function $f_+-f$ on $\{r_0\leq r\leq R\}$. Let $(x^1,x^2,x^3)$ be a coordinate chart containing $x_0$. We subtract \eqref{ball1} from \eqref{barup} and use the fact that the first order partial derivatives of $f$ and $f_+$ coincide at $x_0$. This yields the inequality
%\begin{equation*}
%(1+|df_+|_g^2(x_0))^{-\frac{1}{2}}\left(g^{ij}(x^0)-\frac{(f_+)^i(x_0)(f_+)^j(x_0)}{1+|df_+|_g^2}\right)\frac{\partial^2 (f_+-f)}{\partial x^i \partial x^j}(x_0)\leq \tau(f_+-f)(x_0)=-\tau L_0.
%\end{equation*}
%Since both $g^{ij}(x^0)-\frac{(f_+)^i(x_0)(f_+)^j(x_0)}{1+|df_+|_g^2}$ and $\frac{\partial^2 (f_+-f)}{\partial x^i \partial x^j}(x_0)$ are nonnegative definite, it follows that $L_0=0$. Hence $f\leq f_+$ on $\{r_0\leq r \leq R\}$. A similar argument using \eqref{barlow} shows that $f_-\leq f$ on $\{r_0\leq r \leq R\}$. 
\end{proof}

\section{The existence of a geometric solution}\label{secExistence}
In this section we construct a \emph{geometric} solution of the Jang equation with respect to asymptotically hyperbolic initial data $(M,g,K)$ which is assumed to have Wang's asymptotics as in Definition \ref{defAHdata} with $l\geq 5$. By a geometric solution we mean a properly embedded complete $C^3$ hypersurface $\Sigma \subset M\times \bR$ satisfying the prescribed mean curvature equation $H^\Sigma=\tr^\Sigma K$ where $K$ is extended parallelly along the $\bR$-factor as described in Section \ref{secJf}.  The existence and properties of the constructed geometric solution are summarized in Theorem \ref{thJangGraph}. The theorem is proven by suitably modifying the respective construction in the asymptotically Euclidean case that was carried out in \cite{PMT2}.  Alternatively, one could rely on the geometric measure theory based methods as in \cite{EichmairJang}. However, we choose not to discuss these methods here  as  the less technical argument of \cite{PMT2} suffices for our purposes.

The main ingredient of the proof are  the so-called local parametric estimates for graphical hypersurfaces in $M\times \bR$ whose graphing functions $f:\Omega\rightarrow\bR$ are defined by the boundary value problem
\begin{subequations}
\begin{eqnarray}
H_g(f)-\tr_g(K)(f)=\tau f &\text{ in }& \Omega \label{JangCap1}\\
f=\phi &\text{ on }& \partial\Omega \label{JangCap2}
\end{eqnarray} 
\end{subequations}
where $\Omega \subset M$ and $\phi$ are as in Theorem \ref{existenceBVP}.  These estimates are obtained in Proposition \ref{propLocParam}. In Section \ref{subsecLimit} we apply these estimates to prove the existence of a geometric solution. Let us briefly outline the main idea of the construction. From the proof of Lemma \ref{lemApriori} we know that if $f$ is a solution of \eqref{JangCap1}-\eqref{JangCap2} then $\tau|f|\leq \max\{C_1, \tau C_2\}$, where $C_1$ depends only on $(M,g,K)$ while $C_2$ might also depend on $\Omega$ and $\phi$. Consequently, if we choose $\tau$ so that  $\tau \in (0, C_2^{-1})$ then $\tau |f| \leq \mu_1$ in $\Omega$ for some $\mu_1$ depending only on $(M,g,K)$. For similar reasons we may assume that $\tau|df|_g\leq \mu_2$ holds in $\Omega$ for $\mu_2$ depending only on $(M,g,K)$. Now consider a sequence  $\{R_n\}_{n\in\mathbb{N}}$  such that $R_n>r_0$ and $R_n \to \infty$ as $n\to \infty$. For every $n\in\mathbb{N}$ we choose $\phi_n \in C^{2,\alpha}(\partial \mathcal{B}_{R_n})$ so that $f_-\leq \phi_n \leq f_+$. In the view of the above discussion we can choose $\tau_n$  so that $\tau_n \searrow 0$ as $n\to \infty$, and  $\tau_n |f|\leq \mu_1$, $\tau_n |df|_g \leq \mu_2$ holds in $\mathcal{B}_{R_n}$ for $\mu_1$ and $\mu_2$ depending only on $(M,g,K)$.  Such a choice of $\tau_n$ ensures that the solutions of the boundary value problems
\begin{eqnarray*}
H_g(f)-\tr_g(K)(f)=\tau_n f &\text{ in }& \mathcal{B}_{R_n} \\
f= \phi_n &\text{ on }& \partial\mathcal{B}_{R_n}
\end{eqnarray*} 
satisfy  the local parametric estimates of Proposition \ref{propLocParam} with uniform constants depending only on $(M,g,K)$. With these estimates at hand one can study the limit of the respective solutions $f_n$ as $n\to \infty$. This limit might blow up/down inside the compact set  where the barriers are not defined, but wherever the barriers are defined the limit is graphical  and is trapped between the barriers. 

\subsection{Setup}

When $l\geq 5$ in the Definition \ref{defAHdata} the manifold $(M\times\bR, \ghat = g+dt^2)$ admits \emph{uniformly controlled normal coordinates}, see e.g. \cite[Lemma V.3.4]{SY-LN}. More specifically, there exists $\rho_0>0$ such that at every point $p\in M\times \bR$ there is a normal coordinate chart
\begin{equation}\label{eqNorCoor}
\varphi: M\times \bR \supset B_{\rho_0}^4(p)\rightarrow B^4_{\rho_0}(0) \subset \bR^4: q \mapsto (y^1(q),y^2(q),y^3(q),y^4(q)),
\end{equation}
such that $\varpi = \varphi_*\ghat - \delta$ satisfies 
\begin{equation}\label{eqUnifEucl}
\sup_{x\in B^4_{\rho_0}(0)} \left(|x|^{-2}|\varpi|+|x|^{-1}|\partial \varpi|+|\partial^2 \varpi| + |\partial^3 \varpi| \right)\leq C
\end{equation}
for a constant $C>0$ independent of $p$, where $\delta$ denotes the Euclidean metric on $\bR^4$ and $\partial$ stands for the respective coordinate derivatives. Let $\ghat_{ab}$ denote the components of $\ghat$ in the described normal coordinates, that is $\varphi_*\ghat = \sum_{a,b=1}^4 \ghat_{ab} dy^a dy^b$, and we will write $\partial_a =\partial_{y^a}$. In this section we let the indices $a,b,\ldots$ run from 1 to 4, and the indices $i,j,\ldots$ run from 1 to 3. 

Given a $C^{3,\alpha}_{loc}$ graphical hypersurface $\Sigma \subset M\times \bR$ and  $p\in \Sigma$ we may without loss of generality assume that the tangent space to $\Sigma$ at $p$ corresponds to the coordinate slice $\{y^4=0\}$. In this case, $\Sigma$ can be locally written as the graph of a function $w=w(y)$ where $y=(y^1,y^2,y^3)$. We will call $w$ a \emph{local defining function} and denote its domain by $D_w$. \emph{Local parametric estimates} to be obtained in Section \ref{subLocParam} are certain uniform estimates for defining functions.

%\caution To avoid excessive notation, in this section we use usual unhatted symbols like $R$, $\nabla$, $\Gamma$ etc to denote various quantities associated with the metric $\ghat=dt^2+g$.

\subsection{Local parametric estimates}\label{subLocParam}

A key ingredient for deriving local parametric estimates is the $C^0$-bound on the second fundamental form of $\Sigma$.

\begin{proposition}\label{propBound2FF}
Let $\Sigma$ be a hypersurface given as the graph of $f:\Omega\rightarrow\bR$, where $f$ is a solution to the boundary value problem \eqref{JangCap1}-\eqref{JangCap2}, and suppose that $\tau|f|\leq \mu_1$,  $\tau|df|_g\leq \mu_2$, where $\mu_1$ and $\mu_2$ depend only on $(M,g,K)$. Let $A$ denote the second fundamental form of $\Sigma$. Then for any sufficiently small $\rho>0$ there exists a constant $C>0$ depending only on $\rho$ and $(M,g,K)$
such that for any $p=(x,f(x))\in \Sigma$ with $\dist (x,\partial \Omega)\geq \rho$ we have $|A|^2(p)\leq C$.
\end{proposition}

\begin{proof}
See \cite[Appendix E]{SakovichThesis} where the proof of \cite[Proposition 1]{PMT2} is adapted to the current setting. Since the required modifications are minor we choose not to include this rather lengthy proof here.
\end{proof}

The following result is stated in \cite{PMT2} in the case when $\Omega=M$. Even though this result appears to be standard, we include its proof as it seems difficult to find it in the literature, and since we will refer to it later in the text.

\begin{lemma}\label{lemDefFuncIneq}
For every sufficiently small $\rho>0$ and $\rho_0>0$ there exists a constant $C>0$ depending only on $(M,g,K)$, $\rho$ and $\rho_0$ such that the inequality 
\begin{equation}\label{eqDefFuncIneq}
C(1+|\partial w|^2)^3 \geq \sum_{i,j=1}^3 (\partial_i \partial _j w)^2
\end{equation}
holds  on $D_w \cap \{|y|<\rho_0\}$ for every $p=(x,f(x))\in \Sigma$ such that $\dist (x,\partial \Omega)\geq \rho$. 
\end{lemma}

\begin{proof} 
Assume that $p\in \Sigma$ and let
\begin{equation*}
\{(y,y^4): y^4 = w(y), \, y\in D_w \cap \{|y|<\rho_0\}\}
\end{equation*}
be the local graphical parametrization of $\Sigma$ near $p$. In this case, the vectors $e_i=\partial_i + (\partial_i w) \partial_4$ are tangent to $\Sigma$. Let $\gbar_{ij}=\ghat(e_i,e_j)$ be the respective components of the induced metric on $\Sigma$. In what follows we tacitly assume that all computations are carried out at a fixed point $q\in\Sigma$ covered by the above local parametrization, and we let $C$ denote a generic constant that may vary from line to line but depends only on $(M,g)$.

Let $\Theta$ be the largest eigenvalue of $\gbar=\{\gbar_{ij}\}$ and let $X=X^i e_i$ with $(X^1)^2 + (X^2)^2 + (X^3)^2 = 1$ be the respective eigenvector. We set $Y=X^i \partial_i$ and let $\Lambda$ denote the largest eigenvalue of $\ghat = \{\ghat_{ab}\}$ with respect to the Euclidean metric $\delta$. Relying on  \eqref{eqUnifEucl} we estimate
\begin{equation*}
\begin{split}
\Theta &=  \gbar(X,X)\\ 
&=\ghat_{ij} X^i X^j +2(\partial_i w) \ghat _{4j}X^i X^j+(\partial_i w)(\partial_j w) \ghat_{44} X^i X^j\\
&=\ghat(Y,Y)+2 \ghat (\partial_4, Y) dw(Y)+ \ghat_{44} (dw(Y))^2\\
&\leq \ghat(Y,Y) +2|\partial_4|_{\ghat} |dw|_{\ghat} |Y|^2_{\ghat}+|\ghat_{44}||dw|^2_{\ghat}|Y|^2_{\ghat}\\
&\leq C \ghat(Y,Y)(1+|dw|_{\ghat})^2\\
&\leq C \Lambda (1+|dw|^2_{\ghat})\\
&\leq C (1+|\partial w|^2)
\end{split}
\end{equation*}
which yields a lower bound for the smallest eigenvalue $\Theta^{-1}$ of $\gbar^{-1} = \{\gbar^{ij}\}$. We note for the record that the lowest eigenvalue of $\gbar$ is uniformly bounded in terms of the lowest eigenvalue of $\ghat$, which gives the uniform upper bound for the largest eigenvalue of $\gbar^{-1}$.  % smth like 1/2 of the lowest eigenvalue of \ghat

In the rest of the proof we identify  all bilinear forms with their matrices in the basis $\{e_1,e_2,e_3\}$. Let $O$ be the orthogonal matrix such that $O \gbar^{-1}O^{T}=D$, where $D$ is diagonal, and let $\widetilde{A}=O A O^{T}$. Then
\begin{equation*}
\begin{split}
|A|^2=\tr (\gbar^{-1} A \gbar^{-1} A)=\tr (D\widetilde{A} D \widetilde{A})\geq \Theta^{-2} \tr{\widetilde{A}^2}= \Theta^{-2} \tr{A^2}\geq \frac{C\tr A^2}{(1+|\partial w|^2)^2}. 
\end{split}
\end{equation*}
Set  $W(y, y^4) = y^4 - w(y)$. Using \eqref{eqUnifEucl} it is straightforward to check that
\begin{equation*}
\begin{split}
\tr{A^2} &= \sum_{i,j=1}^3 A(e_i,e_j)^2 \\ & =\sum_{i,j=1}^3 (\Hess W (e_i,e_j))^2|dW|^{-2}_{\ghat} \\ & \geq \frac{(1-C \varepsilon)\sum_{i,j} (\partial _i \partial _j w)^2-C\varepsilon (1+|\partial w|^2)^3}{C(1+|\partial w|^2)}
\end{split}
\end{equation*}
where $\varepsilon>0$ can be assumed to be as small as we want up to decreasing $\rho_0$ if necessary.

Hence
\begin{equation*}
|A|^2\geq \frac{(1-C \varepsilon)\sum_{i,j} (\partial _i \partial _j w)^2-C\varepsilon (1+|\partial w|^2)^3}{C(1+|\partial w|^2)^3},
\end{equation*}
and \eqref{eqDefFuncIneq} follows at once by Proposition \ref{propBound2FF}.
\end{proof}

\begin{lemma}\label{lemDefFunc2Der}
If $\rho>0$ is sufficiently small then there exists  $\rho'>0$ depending only on $(M,g,K)$ and $\rho$ such that for every $p=(x,f(x))\in \Sigma$ with $\dist(x, \partial\Omega)\geq\rho$ the local defining function $w$ is defined on $\{|y|\leq \rho'\}$. Moreover, there exists a constant $C>0$ depending only on $(M,g,K)$ and $\rho$ such that 
\begin{equation}\label{EqDefFunc2Der}
\sup_{|y|\leq \rho'} (|w(y)|+|\partial w(y)|+|\partial \partial w(y)|)\leq C.
\end{equation}
\end{lemma}

\begin{proof}
The proof is outlined in \cite{PMT2} and we include it here only for the sake of completeness. We assume that $\rho$ and $\rho_0$ are such that the conclusion of Lemma \ref{lemDefFuncIneq} holds true. Let $\xi$ be a Euclidean unit vector in the $y^1y^2y^3$-space. For any $0\leq \rhotil \leq \rho_0$ we define the function
\begin{equation*}
S_{\xi}(\rhotil)=\max_{0 \leq \lambda \leq \rhotil} s_{\xi}(\lambda), \text{ where } s_{\xi}(\lambda) =\sum_{i=1}^3(\partial_i w)^2(\lambda \xi) = |\partial w|^2 (\lambda \xi).
\end{equation*} 
Since $(\partial w) (0) = 0$, by the mean value theorem we can write $s_\xi (\lambda)=s'_{\xi} (\theta \lambda)\lambda$ for some $0\leq \theta \leq 1$. Using  the Cauchy-Schwartz inequality, \eqref{eqDefFuncIneq} and the fact that $|\xi|=1$, we may estimate $s'_{\xi} = 2\sum_{i,j=1}^3 (\partial_i w) (\partial_i \partial_j w)\xi^j$ as
\begin{equation*}
 |s'_{\xi}| \leq C|\partial w|(1+|\partial w|^2)^{\frac{3}{2}} \leq C (1+|\partial w|^2)^{\frac{5}{2}}.
\end{equation*}
Here and in the rest of the proof $C>0$ is a generic constant that depends only on the quantities mentioned in the statement of the lemma. Combining the above estimates one can check that
\begin{equation*}
S_\xi (\rhotil)\leq C (1+S_\xi (\rhotil ))^{\frac{5}{2}}\rhotil,
\end{equation*}
or, equivalently,
\begin{equation*}
S_\xi (\rhotil)(1+S_\xi (\rhotil ))^{-\frac{5}{2}}\leq C \rhotil.
\end{equation*}
In this case, it is clear that there exists $\rho'>0$ depending only on $C$ such that $S_\xi (\rhotil)$ remains uniformly bounded as long as $0\leq \rhotil < \rho'$. This, in particular, implies that $w$ is defined on $\{|y|\leq \rho'\}$ and that $\sup_{|y|<\rho'} |\partial w|<C$ for a uniform constant $C$. We conclude the proof by noting that the bound on $|\partial \partial w|$ follows from \eqref{eqDefFuncIneq}, and the bound on $|w|$ is a simple consequence of the mean value theorem.
\end{proof}

With \eqref{EqDefFunc2Der} at hand, one can finally obtain the local parametric estimates. The following result is essentially Proposition 2 in \cite{PMT2}, but we nevertheless include the proof so that we can refer to some intermediate steps in the later sections.

\begin{proposition}\label{propLocParam}
If $\rho>0$ is sufficiently small then there exists  $\rhotil>0$ depending only on $(M,g,K)$ and $\rho$ such that for every $p=(x,f(x))\in \Sigma$ with $\dist(x, \partial\Omega)\geq\rho$ the local defining function $w$ is defined on $\{|y|\leq \rhotil\}$ and the following holds. 
\begin{itemize}
\item For any $\alpha\in (0,1)$ there exists a constant $C$ depending only on $(M,g,K)$, $\rho$, and $\alpha$ such that 
\begin{equation}\label{eqDefFuncC3Alpha}
\|w\|_{C^{3,\alpha}(\{|y|\leq \rhotil\})}<C.
\end{equation}

\item Let $\nu$ be the downward pointing unit normal to $\Sigma$ and let $v=-\partial_t$. Then the following Harnack type inequality
\begin{equation}\label{eqHarnack}
\sup_{\Sigma\cap B^4_{\rhotil}(p)}\langle \nu,v \rangle\leq  C \inf_{\Sigma\cap B^4_{\rhotil}(p)}\langle \nu,v \rangle
\end{equation} 
holds for a constant $C$ depending only on $(M,g,K)$ and $\rho$. 

\item We have $\Sigma\cap B^4_{\rhotil}\subseteq \{y^4=w(y)\}$. 
\end{itemize}
\end{proposition}

\begin{proof}
Let $\rho'$ be as in Lemma \ref{lemDefFunc2Der}. Set $W(y, y^4) = y^4 - w(y)$. Since the bilinear form $\left\{\ghat^{ab}-\frac{W^a W^b}{|dW|^2_{\ghat}}\right\}$ is degenerate in the direction of $dW$ and is equal to $\{\ghat^{ab}\}$ when restricted to the cotangent space of $\Sigma$, as a consequence of \eqref{JangCap1} $W$ satisfies 
\begin{equation*}
\left(\ghat^{ab}-\frac{W^a W^b}{|dW|^2_{\ghat}}\right)\left(\frac{\Hess _{ab}W}{|dW|_{\ghat}}-K_{ab}\right)=\tau t_{|_\Sigma},
\end{equation*}
where $t_{|_\Sigma}$ is the coordinate along the $\bR$-factor in $M\times \bR$ restricted to $\Sigma$.
As a consequence, the local defining function $w$ satisfies an equation of the form
\begin{equation*}
B^{ij}(y,w,\partial w) \partial_i \partial_j w= D(y, w, \partial w)
\end{equation*}
 on $\{|y|\leq\rho'\}$. By  the eigenvalue estimates from the proof of Lemma \ref{lemDefFuncIneq} and Lemma \ref{lemDefFunc2Der}, it follows that the differential operator in the left hand side is strictly elliptic, and that the coefficients of the equation are H\"older continuous functions of $y$. The estimate \eqref{eqDefFuncC3Alpha} follows at once for any $\rhotil\in(0,\rho')$ by standard arguments combining Lemma \ref{lemDefFunc2Der},  Schauder estimates, and a simple bootstrap. 

We shall now focus on proving the Harnack type inequality \eqref{eqHarnack}. Recall from \cite[equation (2.28)]{PMT2} that the function $\eta = \langle v, \nu \rangle \geq 0$ satisfies
\begin{equation*}
\Delta^{\gbar} \eta + \left( \tr_{\Sigma} \widehat{R}(\nu,\cdot,\nu,\cdot)+\nu H+|A|^2 \right) \eta = 0,
\end{equation*}
where $\widehat{R}$ is the curvature tensor of the metric $\ghat$ and $H = H_\Sigma = \tr_\Sigma A$ is the mean curvature of $\Sigma$.
Using the notations as in the proof of Lemma \ref{lemDefFuncIneq} we may rewrite this as the following equation for $\eta=\eta(y)$:
\begin{equation}\label{eqEta}
\alpha^{ij} \,\partial_i\partial_j \eta+\beta^k \,\partial_k \eta+\gamma \,\eta=0, 
\end{equation}
where 
\begin{eqnarray*}
\alpha^{ij} &=& \gbar^{ij},\\
\beta^k &=& -\gbar^{ij}\left(\hat{\Gamma}_{ij}^k+2(\partial_i w)\hat{\Gamma}_{4j}^k+(\partial_i w)(\partial_j w) \hat{\Gamma}^k_{44}\right),\\
\gamma &=&  \tr_{\Sigma} \widehat{R}(\nu,\cdot,\nu,\cdot)+\nu H+|A|^2.
\end{eqnarray*}
Using the formulae $\nu=-|d W|_{\ghat}^{-1}\nabla^{\ghat} W$,
\begin{equation*}
|A|^2 = \left(\ghat^{ac}-\frac{W^a W^c}{|dW|^2_{\ghat}}\right) \left(\ghat^{bd}-\frac{W^b W^d}{|dW|^2_{\ghat}}\right)\left(\frac{\Hess_{ab}W}{|dW|_{\ghat}}\right)\left(\frac{\Hess_{cd}W}{|dW|_{\ghat}}\right),
\end{equation*}
and 
\begin{equation}\label{eqDerMC}
 H=\left(\ghat^{ab}-\frac{W^a W^b}{|dW|^2_{\ghat}}\right)\frac{\Hess_{ab}W}{|dW|^2_{\ghat}},
\end{equation}
it is straightforward to rewrite $\gamma$ in terms of the defining function $w$. By  the eigenvalue estimates from the proof of Lemma \ref{lemDefFuncIneq} the differential operator in the left hand side is strictly elliptic, and combining \eqref{eqUnifEucl} with \eqref{eqDefFuncC3Alpha} we conclude that it has uniformly bounded coefficients on $\{|y|\leq\rhotil \}$. 
Applying \cite[Corollary 8.21]{GT} we conclude that $\eta$ satisfies the Harnack inequality
\begin{equation}\label{eqHarnack2}
\sup_{|y|\leq 3\rhotil/4}\eta\leq C \inf_{|y|\leq 3\rhotil/4}\eta
\end{equation}
with $C$ depending only on $\rho$ and $(M,g,K)$. Redefining $\rhotil$ as $3\rhotil/4$, \eqref{eqHarnack} follows. 

Finally, we prove the last claim of the proposition. In fact, by slightly refining the arguments above, one can see that the coefficients of equation \eqref{eqEta} are uniformly bounded in $C^{0,\alpha}$ norm on $\{|y|\leq\rhotil\}$.
Standard interior elliptic estimates %\cite[Problem 6.1]{GT} 
then imply that
\begin{equation*}
\sup_{|y|\leq\rhotil /2}|\partial \eta|\leq C \sup_{|y|\leq 3 \rhotil /4}|\eta|
\end{equation*}
Combining this estimate with \eqref{eqHarnack2} and bounds on the eigenvalues of $\gbar$, we get
\begin{equation*}
\sup_{|y|\leq\rhotil /2}|d\eta|_{\gbar}\leq C \inf_{|y|\leq \rhotil/2}\eta, 
\end{equation*}
hence 
\begin{equation}\label{eqAtPoint}
\sup_{|y|\leq\rhotil /2}|d(\ln \eta)|_{\gbar}\leq C.
\end{equation}

Let $\Omega_\rho$ be the set of $x\in \Omega$ such that $\dist(x,\partial \Omega)\geq \rho$ and set $\Sigma_\rho\definedas \{(x,f(x)): x\in \Omega_\rho\}$.
Since the constant $C$ in \eqref{eqAtPoint} does not depend on $p\in \Sigma$, we have 
\begin{equation}\label{eqBoundLn}
\sup_{\Sigma_\rho}|d(\ln \eta)|_{\gbar}\leq C,
\end{equation}
where $C$ depends only on $(M,g,K)$ and $\rho$. In fact, a simple computation in an orthonormal frame (see the derivation of \cite[equation (2.24)]{PMT2}) shows that $|\nabla^{\ghat}_\nu\, \nu|_{\ghat}^2 =|d(\ln \eta)|^2_{\gbar}$. Thus \eqref{eqBoundLn} amounts to $|\nabla^{\ghat}_\nu \,\nu|_{\ghat}^2\leq C$ which in combination with Proposition \ref{propBound2FF} gives 
\begin{equation}\label{eqAmbient}
|\nabla^{\ghat} \nu|_{\ghat} \leq C, % = \sum_a |\nabla_{e_a} e_4|^2 
\end{equation}
which holds on $\Sigma_\rho$, and, more generally, in $\Omega_\rho \times \bR$. With this estimate at hand one can prove the last claim of the proposition using implicit function theorem. For more details,  see the proof of Corollary \ref{corGraph} below where a version of this argument is used.
\end{proof}

\subsection{Passing to the limit}\label{subsecLimit}

We finally prove the existence of a geometric solution of the Jang equation.
\begin{theorem}\label{thJangGraph}
Let $(M,g,K)$ be an asymptotically hyperbolic initial data set with Wang's asymptotics as in Definition  \ref{defAHdata}. Then there exists a properly embedded complete $C^3$ hypersurface $\Sigma\subset M\times \bR$ such that
\begin{itemize}
\item[1)] $\Sigma$ is the boundary of some open set $O\subset M\times \bR$. Moreover, $H_{\Sigma}-\tr_{\Sigma} K=0$ where $H_{\Sigma}$ is the mean curvature of $\Sigma$ computed as the tangential divergence of the normal pointing out of $O$. 

\smallskip

\item[2)] $\Sigma$ consists of finitely many connected components $\widetilde{\Sigma}$. Each component is either a cylinder of the form $\mathcal{E}\times \bR$, where $\mathcal{E}$ is a closed properly embedded $C^3$ hypersurface in $M$, or it is a graph of a $C^3$ function $f_{\widetilde{\Sigma}}$ whose domain $U_{\widetilde{\Sigma}}$ is an open subset of $M$. The function $f_{\widetilde{\Sigma}}$ is a solution of the Jang equation $H_g(f_{\widetilde{\Sigma}})-\tr_g(K)(f_{\widetilde{\Sigma}})=0$ on $U_{\widetilde{\Sigma}}$.

\smallskip

\item[3)] The boundary of the domain $U_{\widetilde{\Sigma}}$ for every graphical component $\graph(f_{\widetilde{\Sigma}}, U_{\widetilde{\Sigma}})$ of $\Sigma$ is a closed properly embedded $C^3$ hypersurface in $M$. In fact, $\partial U_{\widetilde{\Sigma}}$ consists of two disjoint unions of components $\mathcal{E}^+$ and $\mathcal{E}^-$ such that $f_{\widetilde{\Sigma}}(x)\rightarrow \pm \infty$ as $x\rightarrow \mathcal{E}^\pm$. We have $H_{\mathcal{E}^{\pm}}\mp \tr_{\mathcal{E}^{\pm}} K=0$, where the mean curvature is computed as the tangential divergence of the unit normal pointing out of $U_{\widetilde{\Sigma}}$. Furthermore, the hypersurfaces $\graph(f_\Sigma - C, U_\Sigma) \subset M\times \bR$
converge locally uniformly in $C^{3,\alpha}$ to the cylinder $\mathcal{E}^\pm \times \bR$ when $C\to \pm \infty$. 

\smallskip

\item[4)] $\Sigma$  has a graphical component $\widetilde{\Sigma}_0 =\graph(f_{\widetilde{\Sigma}_0},U_{\widetilde{\Sigma}_0})$ such that the domain $U_{\widetilde{\Sigma}_0}$ contains the region $\{r\geq r_0\}$. We have 
\begin{equation}\label{eqFallOff3D}
f_{\widetilde{\Sigma}_0}=\sqrt{1+r^2}+\alpha \ln r+\psi(\theta, \varphi)+O_3(r^{-1+\varepsilon})
\end{equation}
in $U_{\widetilde{\Sigma}_0}$
for a sufficiently small $\varepsilon \in (0,1)$.
\end{itemize} 
\end{theorem}

\begin{remark}
Although
\begin{equation}\label{eqWithoutDeriv}
f_{\widetilde{\Sigma}_0}=\sqrt{1+r^2}+\alpha \ln r+\psi(\theta, \varphi)+O(r^{-1+\varepsilon})
\end{equation}
follows directly from the construction, proving  \eqref{eqFallOff3D} requires quite a bit of work. Therefore in the current section  we only prove the first three claims of the theorem. The lengthy and technical proof of \eqref{eqFallOff3D} is carried out in Section \ref{secJangAE}. 
\end{remark}

\begin{proof}
Let $R_n$ and $\tau_n$ be positive real numbers such that $R_n\rightarrow \infty$ and $\tau_n \rightarrow 0$ as $n\rightarrow \infty$. 
By Proposition \ref{propExBVP} for each sufficiently large $n$ we can solve  the boundary value problem
\begin{eqnarray*}
H_g(f)-\tr_g(K)(f)=\tau _n f &\text{ in }& \mathcal{B}_{R_n} \\
f= \phi_n &\text{ on }& \partial\mathcal{B}_{R_n}
\end{eqnarray*} 
where $\phi_n$ is a function on $\partial\mathcal{B}_{R_n}$ such that $f_-\leq \phi_n \leq f_+$. Let the respective solution be denoted by $f_{n}$, and let $\Sigma_{n}$ be its graph. As discussed in the beginning of Section \ref{secExistence}, we may without loss of generality  assume that $\tau_n$ is chosen so that $\tau_n|f_n|\leq \mu_1$, and $\tau_n |df_n|_g\leq \mu_2$, where $\mu_1$ and $\mu_2$ depend only on $(M,g,K)$, so that the results of Section \ref{subLocParam} apply to $\Sigma_n$.

Let us study the convergence of $\Sigma_{n}$ when $n \rightarrow \infty$. The argument is standard, see e.g. \cite[Section 4]{PerezRos}. We fix some small $\rho>0$ and choose $\rhotil>0$ as in Proposition \ref{propLocParam} so that the estimate \eqref{eqDefFuncC3Alpha} holds for any $p=(x,f_{n}(x))\in \Sigma_n$ where $x\in \mathcal{B}_{R_n-\rho}$. Since $f_-\leq f_{n} \leq f_+$ holds on $\{r\geq r_0\}$, it is obvious that the sequence $\{\Sigma_{n}\}_n$ has accumulation points in $M\times \bR$. We choose a countable dense set $\{p_1,p_2,\ldots\}$ in $M\times \bR$ and proceed as follows. 

Consider the geodesic ball $B^4_{\rhotil/2}(p_1)$. Suppose that this ball contains an accumulation point $q_1$ of the sequence $\{\Sigma_{n}\}_n$. In this case we consider the ball $B^4_{\rhotil}(q_1)\supset B^4_{\rhotil/2}(p_1)$.  Without loss of generality, we assume that there is a sequence of points $q_{1,n}\in \Sigma_{n}$ such that $q_{1,n}\rightarrow q_1$ as $n\rightarrow \infty$. Let $\nu_n(q_{1,n})$ be the (downward pointing) unit normal to $\Sigma_{n}$ at $q_{1,n}$. Since $S^3$ is compact, we can choose a subsequence of $q_{1,n}$ denoted by the same notation such that the respective normals $\nu_n(q_{1,n})$ converge to some unit vector $\nu(q_1)$ when $n\rightarrow \infty$. In fact, we can assume that  $\nu_n(q_{1,n})=\nu(q_1)$ without violating the uniform estimate $\|w_n\|_{C^{3,\alpha}(\{|y|\leq \rhotil\})}<C(\alpha)$ which holds for the defining functions $w_{n}$ of $\Sigma _{n}$ such that $w_{n}(0)=q_{1,n}$. This allows us to apply the Arzela-Ascoli theorem and extract a subsequence of these defining functions converging in $C^3$ on $\{|y|\leq \rhotil\}$ to a function $y^4=w(y)$. We thus obtain a subsequence $\{\Sigma_{n,1}\}_n$ converging with multiplicity one to a $C^3$ hypersurface $\Sigma$ in $B^4_{\rhotil/2}(p_1)$. 

If $B^4_{\rhotil/2}(p_1)$ contains no accumulation points of $\Sigma_{n}$ then we can instead take $\{\Sigma_{n,1}\}_n$ to be a subsequence of $\Sigma_{n}$ such that $\Sigma_{n,1}\cap B^4_{\rhotil/2}(p_1)=\emptyset$.

We repeat this procedure with the sequence $\{\Sigma_{n,1}\}_n$ in $B^4_{\rhotil/2}(p_2)$ and extract a subsequence $\{\Sigma_{n,2}\}_n$, which either converges with multiplicity one to a $C^3$ hypersurface $\Sigma$ in $B^4_{\rhotil/2}(p_2)$, or satisfies $\Sigma_{n,2}\cap B^4_{\rhotil/2}(p_2)=\emptyset$. Iterating this process, we see that the diagonal subsequence $\{\Sigma_{n,n}\}_n$ converges to a properly embedded complete $C^3$ hypersurface $\Sigma\subset M \times \bR$. If each $\Sigma_{n}$ is viewed as the boundary of the set $\{(x,t): t>f_{n}(x), x\in \mathcal{B}_{R_n}\}$, then it is clear that $\Sigma$ is the boundary of some open subset $O\subset M\times \bR$, and that $\Sigma$ satisfies $H_{\Sigma}-\tr_{\Sigma}K=0$ with respect to the normal pointing out of $O$. By the Harnack inequality \eqref{eqHarnack}, each connected component of $\Sigma$ is either graphical or cylindrical. We can view the union of the graphical components of $\Sigma$ as the graph of $f_{\widetilde{\Sigma}}:U_{\widetilde{\Sigma}}\rightarrow \bR$, where $\{r\geq r_0\}\subset U_{\widetilde{\Sigma}}$ where $U_{\widetilde{\Sigma}}$ might be disconnected. It is clear that $f_{\widetilde{\Sigma}}$ solves the Jang equation $H_g(f_{\widetilde{\Sigma}})-\tr_g (K)(f_{\widetilde{\Sigma}})=0$ on $U_{\widetilde{\Sigma}}$, and that $f_-\leq f_{\widetilde{\Sigma}} \leq f_+$ holds on $\{r\geq r_0\}$, which implies that it has the asymptotic behavior as in \eqref{eqWithoutDeriv}. It is also obvious that when we approach a connected component $\mathcal{E}$ of $\partial U_{\widetilde{\Sigma}}$ the graph of  $f_{\widetilde{\Sigma}}$ asymptotes the cylinder $\mathcal{E}\times \bR$. Taking the limit $C\to \infty$ of $f_{\widetilde{\Sigma}} \pm C$ we see that $\mathcal{E}\times \bR$ is a geometric solution of the Jang equation. From this it is easy to conclude that $H_{\mathcal{E}}\mp \tr_{\mathcal{E}}K=0$ with respect to the normal pointing out of $U_{\widetilde{\Sigma}}$, the sign depending on whether $f_{\widetilde{\Sigma}}\rightarrow +\infty$ or  $f_{\widetilde{\Sigma}}\rightarrow -\infty$ as we approach $\mathcal{E}$.  Finally, we note that $\Sigma$ has finitely many connected components, since the region $\mathcal{B}_{r_0}$ (where multiple graphical or cylindrical components might occur) is precompact, and since by Proposition \ref{propLocParam} there is a uniform $\rhotil$ such that $\Sigma\cap B^4_{\rhotil}\subset \{y^4=w(y)\}$ holds over this region.
\end{proof}

\section{The Jang graph is an  asymptotically Euclidean manifold}\label{secJangAE}
The goal of this section is to show that the graphical component 
\begin{equation*}
\widetilde{\Sigma}_0=\graph(f_{\widetilde{\Sigma}_0},U_{\widetilde{\Sigma}_0}),  \quad \text{where} \quad \{r \geq r_0\} \subseteq U_{\widetilde{\Sigma}_0},
\end{equation*} 
of the geometric solution of the Jang equation constructed in Theorem \ref{thJangGraph} is an asymptotically Euclidean manifold in the sense of Definition \ref{defAEManifolds}. 
For this, we need to obtain information about the derivatives of $f_{\widetilde{\Sigma}_0}$; more specifically, we need to confirm that \eqref{eqFallOff3D} holds. Note that the function  $f_{\widetilde{\Sigma}_0}$ is defined on an asymptotically hyperbolic manifold. As scalar multiplication is not a homothety for the hyperbolic metric we cannot directly rely on the rescaling technique which was used for similar purposes in \cite{PMT2}.  Instead, we will first show that near infinity we may view $\Sigmatil_0$ as the graph of a function defined on an asymptotically Euclidean manifold (roughly speaking, the graph of the lower barrier $\Sigma_-$  that was constructed in  Section \ref{secBarriers}). Applying the rescaling technique to the equation that the graphing function satisfies, we will show that its derivatives fall off sufficiently fast for concluding that $\Sigmatil_0$ is an asymptotically Euclidean manifold. We will then rewrite these estimates in terms of $f_{\widetilde{\Sigma}_0}$ thereby establishing \eqref{eqFallOff3D}.

\subsection{Setup}\label{secFermi}

We will use the notation $M_{R} = \{r \geq R\}$ for any $R\geq r_0$. Recall from Section \ref{secExistence} that $f_-\leq f \leq f_+$ holds  in $M_{r_0}$ where $f_-$ and $f_+$ are barriers for the Jang equation constructed in Section \ref{secBarriers}. These barriers are defined implicitly by using solutions of certain initial value problems. For this reason it is not very convenient to use them for the purposes of this section. At the same time, the properties of the barriers established in Section \ref{secBarriers} allow us to pick a  sufficiently large $r_1>r_0$ and two functions\footnote{Denoted by the same notation as the actual barriers, since the later will not be used in the rest of this paper.} $f_-: M_{r_1} \to \bR$ and $f_+: M_{r_1} \to \bR$ such that 
\begin{equation}\label{eqNewBarriers}
f_\pm = \phi_{\pm}(r) + \psi(\theta, \varphi) = \sqrt{1+r^2}+\alpha \ln r+\psi(\theta, \varphi)+O_4(r^{-1+\varepsilon})
\end{equation}
and $f_- \leq f \leq f_+$ on $M_{r_1}$. These two functions are defined  on a potentially smaller neighborhood of infinity than the actual barriers but the asymptotic behavior of their derivatives is more explicit. The graphs of these two functions are denoted by $\Sigma_-$ and $\Sigma_+$ respectively. Note that the submanifolds $(\Sigma_-, g_{\Sigma_-})$ and $(\Sigma_+, g_{\Sigma_+})$ of $M\times\bR$ are asymptotically Euclidean by Lemma \ref{lemma1}.

Rather than using the standard product coordinates on $M\times \bR$, in this section we will work in the so called \emph{Fermi (or normal geodesic) coordinates} adapted to the submanifold $\Sigma_-$. To ensure that these coordinates have good properties (more specifically, that Proposition \ref{propLevel} below holds), in this section we work under the assumption that $(M,g,K)$ is as in Definition \ref{defAHdata} with $l\geq 5$. To avoid excessive notation, in this section we use usual unhatted symbols like $R$, $\nabla$, $\Gamma$ etc to denote various quantities associated with the metric $\ghat=dt^2+g=\langle \cdot, \cdot \rangle$. We will also drop $\ghat$ in the norms. Furthermore, we will write $\Sigma=\graph(f,U)$ instead of $\Sigmatil_0=\graph(f_{\widetilde{\Sigma}_0},U_{\widetilde{\Sigma}_0})$.  

Let $u=(u^1,u^2,u^3)$ be an asymptotically Euclidean Cartesian coordinate system on $\Sigma_-$ obtained from the natural polar coordinate system (see e.g. Lemma \ref{lemma2}) by the usual (spherical coordinates) transformation. Since $\Sigma_-$ has bounded second fundamental form, we may argue as in Section \ref{secExistence} to conclude that $\Sigma_-$ is a uniformly embedded submanifold in the manifold $(M\times \bR, \ghat)$ that has bounded geometry, hence there exists a normal neighborhood $N_{\gamma}(\Sigma_-)$ of radius $\gamma>0$, see \cite[Chapter 2]{Eldering}. We then define the coordinates $y$ on $N_{\gamma}(\Sigma_-) \cong \Sigma_-\times (-\gamma,\gamma) $ such that $y(\cdot,0)=u$ and $\frac{\partial y}{\partial \rho} = \nu^{\Sigma_\rho}$  with $\nu^{\Sigma_\rho}$ being the upward pointing unit normal to $\Sigma_\rho \definedas y(\cdot,\rho)$. Note that in these coordinates we may write $\ghat = d\rho^2 + g_\rho$, where $g_\rho$ is the induced metric on $\Sigma_\rho$. In what follows we let $A_\rho$ denote the second fundamental form of $\Sigma_\rho$ given by $(A_\rho)_{ij}=\langle \nabla_{\partial_i} (-\partial_\rho),\partial_j \rangle$ where 
$\partial_i$  for $i \in \{1,2,3\}$ denote the respective tangent vectors to $\Sigma_\rho$. 

The following result is proven in Appendix \ref{secLevelFermi}.

\begin{proposition}\label{propLevel}
There exist constants $\rho_0>0$ and $C>0$ such that $|A_\rho|<C$ and $\frac{1}{C} \delta_{ij} \leq (g_\rho)_{ij} \leq C \delta_{ij}$ for any $0 \leq \rho \leq \rho_0$. Furthermore, all partial  derivatives of $(g_\rho)_{ij}$ and $(A_\rho)^i_{\phantom{i}j}$ up to order 3   in the Fermi coordinates are bounded. 
\end{proposition}

We note that the proof of this result contains a few important equations that we will use below, most notably  \eqref{eqRiccati} and \eqref{eqMetricEvolution}.

\subsection{The height function: existence and a priori estimates}\label{secAprioriHeight}

The aim of this section  is to show that near infinity $\Sigma$ is given as the graph of a function  $h: \Sigma_- \to [0,\gamma)$, that is 
\begin{equation}\label{eqGraphFermi}
\Sigma=\graph h = \{ y(q, h(q)): q\in \Sigma_-\},
\end{equation}
and to obtain some a priori estimates for this function. In what follows, we will refer to $h$ as the \emph{height function} of $\Sigma$ with respect to $\Sigma_-$. 

Using the fact that $\Sigma$ is ``squeezed'' between the graphs of the barrier functions and that its second fundamental form is bounded, we obtain the following estimate for its normal. 

\begin{lemma}[``Tilt-excess'' estimate for $\Sigma$]\label{lemNormals} 
Let $\nu_-$ and $\nu$ be the respective upward pointing normal vector fields to $\Sigma_-$ and $\Sigma$ extended parallelly along the $\bR$-factor in $M\times \bR$. Then there exists a constant $C>0$ such that at every $p \in M_{ 2r_1}\times \bR$ we have 
\begin{equation*}
|\nu(p)-\nu_-(p)|\leq C r(p)^{\frac{-1+\varepsilon}{2}}.
\end{equation*}
\end{lemma}

\begin{proof}
We will use the following notation: for $z\in M\times \bR$ we define $z_M \definedas \proj _M z$ where $\proj_M: M\times \bR \to M$ is the standard projection operator.

Let $p\in \Sigma$ be such that $r(p) > 2 r_1$. We shift $\Sigma_-$ vertically so that it intersects $\Sigma$ at $p$. The resulting hypersurface, which we denote by $\overline{\Sigma}$, is the graph of the function $\bar{f}: M_{r_1} \to \bR$ given by
\[
\bar{f}=f_-+(f(p_M)-f_-(p_M)).
\] 
Define $F_-: M_{r_1} \times \bR \to \bR$ by $F_-(x,t) = t - \bar{f}(x)$. Then we have 
\[
\begin{aligned}
\frac{\nabla F_-}{|\nabla F_-|} =  \nu_-  &\hspace{0.5cm} \text{ in }\hspace{0.5cm} M_{r_1}\times \bR,\\ 
 F_-=0 & \hspace{0.5cm}\text{ on }\hspace{0.5cm} \overline{\Sigma}.
\end{aligned}
\]  
For a point $q\in \Sigma$, let $\gamma$ be a unit speed 
geodesic in $\Sigma$ such that $\gamma(0)=p$ and $\gamma(s)=q$. Since $F_-(p)=0$, for some $\theta\in [0,1]$ we may write 
\begin{equation}\label{eqTaylor}
\begin{split}
F_-(q) = dF_-(\dot{\gamma}(0)) s + \frac{s^2}{2} (\Hess^{\Sigma} F_-) (\dot{\gamma}(\theta s), \dot{\gamma}(\theta s)).
\end{split}
\end{equation}
The claim will be proven by making a suitable choice of $\dot{\gamma}(0)$ and $s=\dist_\Sigma (p,q)$ in this formula.

From \eqref{eqNewBarriers} we know that there exists a constant $C_0>0$ such that $0 \leq (f_+ - f_-)(r) \leq C_0 r^{-1+\varepsilon}$ on $M_{r_1}$. Set $\delta \definedas 3 C_0 r(p)^{-1 + \varepsilon}$ and let $q$ be such that $\dist_\Sigma (p,q)=\sqrt{\delta}$.  We claim that in this case we may without loss of generality  assume that $\frac{r(p)}{2} \leq r(q) \leq 2 r(p)$. Indeed, if we for instance assume that $r(q) < \frac{r(p)}{2}$ then a computation using the fact that $(M,g)$ is asymptotically hyperbolic with Wang's asymptotics shows that
\begin{equation*}
\begin{split}
     \sqrt{3 C_0 r(p)^{-1 + \varepsilon}}
										 &= \dist_{\Sigma} (p,q)\\ 
										 &\geq \dist_M (p_M, q_M) \\
										 & \geq \int_{r(q)} ^{r(p)} \frac{dr}{\sqrt{1+r^2}} \\
										 & \geq \int_{\tfrac{r(p)}{2}} ^{r(p)} \frac{dr}{\sqrt{1+r^2}}\\
										 & \geq \frac{r(p)}{2\sqrt{1+(r(p))^2}} \\
										 & = \frac{1}{2\sqrt{(r(p))^{-2} + 1}},										
\end{split}
\end{equation*}
which cannot be true for a sufficiently large $r_1>0$. Similarly, one reaches a contradiction in the case when $r(q) > 2r(p)$. 

Since $\frac{r(p)}{2} \leq r(q) \leq 2 r(p)$, we have in particular $r(q) \geq r_1$ so that $f_- (q_M)$ and $f_+(q_M)$ are well-defined. Let $q_1\in \overline{\Sigma}$ be such that $(q_1)_M=q_M$. In this case we have
\begin{equation}\label{eqDistBound}
\begin{split}
\dist_{M\times \bR}(q_1,q)&=|\bar{f}(q_M)-f(q_M)|\\
                          &=|(f_-(q_M)-f(q_M))+(f(p_M)-f_-(p_M))|\\
													&\leq |f_+ (q_M) - f_-(q_M)| + |f_+(p_M)-f_-(p_M)| \\
												  &\leq C_0 (r(q_M))^{-1+\varepsilon} + C_0(r(p_M))^{-1+\varepsilon}\\
													&\leq 3 C_0 (r(p))^{-1+\varepsilon}\\
													&=\delta.
\end{split}
\end{equation} 
Since $F_-(q_1) = 0$ and since $\nabla F_-$ is constant along $\bR$-factor in $M\times \bR$, we may now estimate the left hand side of \eqref{eqTaylor} as follows 
\begin{equation*}
F_-(q) =|F_-(q) - F_-(q_1)| \leq  \delta |\nabla F_-| (q).
\end{equation*}
As for the right hand side, note that$|\Hess^\Sigma F_-| \leq |\Hess F_-| + |\nabla F_-||A^\Sigma| \leq C |\nabla F_-|$ for some $C>0$, since the second fundamental forms $A^\Sigma$  and $A^{\Sigma_-}=\frac{\Hess F_-}{|\nabla F_-|}$   are bounded. 
Consequently, choosing $s=\sqrt{\delta}$ in \eqref{eqTaylor} and estimating the left hand side and the right hand side as described above we obtain
\begin{equation}\label{eqTaylor2}
\delta \, |\nabla F_-| (q) \geq  \sqrt{\delta} \, dF_-(\dot{\gamma}(0)) - C \delta \sup_{0\leq \theta \leq 1} |\nabla F_-|(\gamma(\theta s)). 
\end{equation}
Since $|\nabla F_-| = r + O(1)$ and since $\gamma$ is the geodesic such that $\gamma(0)=p$ and $\gamma(s)=q$ where $\frac{r(p)}{2} \leq r(q) \leq 2 r(p)$ we can also estimate 
 \begin{equation*}
\sup_{0\leq \theta \leq 1} |\nabla F_-|(\gamma(\theta s)) \leq 2 \max\{r(p),r(q)\} < 4 r(p) <8 |\nabla F_-| (p). 
\end{equation*}
Applying this estimate to \eqref{eqTaylor2}, after division by $\sqrt{\delta} |\nabla F_-| (p)$ we obtain
\begin{equation*}
\left\langle \nu_- (p), \dot{\gamma}(0) \right\rangle \leq C \sqrt{\delta},
\end{equation*}
possibly for a larger constant $C>0$. Finally, let $\dot{\gamma}(0) = \frac{\nabla^\Sigma F_- (p)}{|\nabla^\Sigma F_-(p)|}$. Then at the point $p$ we have
\begin{equation*}
\begin{split}
C\sqrt{\delta} & \geq \left\langle \nu_-, \frac{\nabla^\Sigma F_-}{|\nabla^\Sigma F_-|}\right\rangle\\ 
               &= \frac{\langle \nu_-, \nabla F_- - \langle \nabla F_-, \nu\rangle \nu\rangle}{|\nabla^\Sigma F_-|} \\
							 &= \frac{\langle \nu_-, \nabla F_- \rangle - \langle \nabla F_-, \nu\rangle \langle\nu, \nu_-\rangle}{|\nabla^\Sigma F_-|} \\
							 &= \frac{\langle \nabla F_-, \nu_-  - \langle\nu, \nu_-\rangle \nu\rangle}{|\nabla^\Sigma F_-|} \\
							 &= \frac{|\nabla F_-|\langle \nu_-, \nu_-  - \langle\nu, \nu_-\rangle \nu\rangle}{|\nabla^\Sigma F_-|} \\
							 &= \frac{(1 - \langle \nu, \nu_- \rangle^2)|\nabla F_-|}{|\nabla F_- - \langle \nabla F_-, \nu\rangle \nu|} \\
							 &= \frac{1 - \langle \nu ,\nu_- \rangle^2}{|\nu_- - \langle \nu ,\nu_- \rangle \nu|}\\
							 &= \sqrt{1 - \langle \nu ,\nu_- \rangle^2}.
\end{split}
\end{equation*}
Thus $\langle \nu, \nu_-\rangle^2 (p) = 1+O(\delta)$, and $|\nu(p)-\nu_-(p)|=O(\sqrt{\delta})$. Recalling the definition of $\delta$ the claim follows.
\end{proof}

With Lemma \ref{lemNormals} at hand we can prove the existence of the height function $h$.

\begin{cor}[Existence of height function]\label{corGraph}
There exists a $C^3_{loc}$-function $h: \Sigma_- \to \bR_{\geq 0}$ and $r_2>0$ such that $\Sigma \cap (M_{r_2} \times \bR) = \graph h$ in the Fermi coordinates as described in Section \ref{secFermi}.
\end{cor}

\begin{proof}
We use the same notations as in Lemma \ref{lemNormals}, in particular we  let $\nu$ and $\nu_-$ denote the upward pointing unit normal vector fields to $\Sigma$ and $\Sigma_-$ extended parallelly along the $\bR$-factor in $M\times \bR$. Let $F: M_{2r_1} \times \bR \to \bR$ be given by $F(x,t) = t - f(x)$. Then we have 
\[
\begin{aligned}
\frac{\nabla F}{|\nabla F|} = \nu  &\hspace{0.5cm} \text{ in }\hspace{0.5cm} M_{2 r_1}\times \bR,\\ 
 F=0 & \hspace{0.5cm}\text{ on }\hspace{0.5cm} \Sigma.
\end{aligned}
\]  
 We will show that $\partial_\rho F$ is bounded away from zero on $\Sigma \cap (M_{r_2} \times \bR)$ provided that $r_2>0$ is sufficiently large. The claim will then follow by  the implicit function theorem.

Fix $q \in \Sigma \cap (M_{3r_1} \times \bR )$. We let $q_-$ denote the orthogonal projection of $q$ on $\Sigma_-$ and we let $q_M$ denote the vertical projection of $q \in M \times \bR$ on $M$ as in the proof of Lemma  \ref{lemNormals}. The same type of argument as in the proof of Lemma  \ref{lemNormals} shows that we may without loss of generality assume that $\tfrac{1}{2}r(q)\leq r(q_-)\leq 2r(q)$. Since $\partial_\rho  = \nu_-$  on $\Sigma_-$, by Lemma \ref{lemNormals} we have 
\[
\frac{\partial_\rho F (q_-)}{|\nabla F (q_-)|} = \langle \nu(q_-) , \nu_-(q_-) \rangle   = 1 + O(r(q_-)^{\frac{-1+\varepsilon}{2}}).
\]
Recalling \eqref{eqAmbient} we obtain 
\begin{equation*}
\begin{split}
 \left|\frac{\partial_\rho F (q)}{|\nabla F (q)|} - \frac{\partial_\rho F (q_-)}{|\nabla F (q_-)|}\right| 
   & \leq \left|\frac{\nabla F (q)}{|\nabla F (q)|} - \frac{\nabla F (q_-)}{|\nabla F (q_-)|}\right|\\
   & \leq |\nu(q) - \nu(q_-)|\\
	 & \leq \sup |\nabla \nu| \dist_{M \times \bR} (q, q_-) \\
	 & \leq C (f(q_M)-f_-(q_M)) \\
          & = O(r(q)^{-1+\varepsilon}),
\end{split}
\end{equation*}
hence
\[
\frac{\partial_\rho F (q)}{|\nabla F (q)|} = 1 + O(r(q)^{\frac{-1+\varepsilon}{2}}).
\]
Finally, since $|\nabla F| = \sqrt{1 + |df|^2_g} \geq 1$, we conclude that
\[
\partial_\rho F (q) \geq \tfrac{1}{2} |\nabla F (q)| \geq \tfrac{1}{2},
\]
provided that $r(q)\geq r_2$ for a sufficiently large $r_2>0$. 
\end{proof}

Estimating the ``vertical gap''  between the barriers it is straightforward to see that the height function satisfies $h=O(r^{-1+\varepsilon})$. In  the following lemma we refine this estimate to $h=O(r^{-2+\varepsilon})$ by estimating the ``horizontal gap'' instead.  We also obtain some preliminary estimates for the coordinate derivatives of $h$.

\begin{lemma}[A priori estimates for the height function]\label{lemPropertiesHeight}
Let  $h: \Sigma_- \to \bR$ be the height function of $\Sigma$ as described in Corollary \ref{corGraph}.  %For any $q=(\overline{y},h(\overline{y}))\in \Sigma$ with $q_-=\proj_{\Sigma_-} q \in \Sigma_-\cap \{r\geq r_1\}$ set $r=r(q_-)$. 
Then $h=O(r^{-2+\varepsilon})$, $|\partial h|=O(r^{-1+\varepsilon})$, and $|\partial \partial h|=O(1)$. 
\end{lemma}

\begin{proof} We address each estimate separately. Recall the following notation: for  any $z\in M\times \bR$ we denote $z_M = \proj _M z$ where $\proj_M: M\times \bR \to M$ is the standard projection operator.\\

\emph{Proving that $h=O(r^{-2+\varepsilon})$.} By considering sufficiently large $r$ we may assume that the functions $\phi_\pm(r)=f_\pm(r,\theta,\varphi) - \psi(\theta,\varphi)= \sqrt{1+r^2}+\alpha \ln r + O(r^{-1+\varepsilon})$  are both increasing. Let  $p\in\Sigma_+$ and $q\in \Sigma_-$  be such that $q$ is the orthogonal projection of $p$ on $\Sigma_-$.  We define $z\in \Sigma_-$ so that $(\theta(z),\varphi(z))=(\theta(p),\varphi(p))$ and $f_+(p_M)=f_-(z_M)$. Clearly, we have $h(q) \leq \dist_{M\times\bR} (p,q)\leq \dist_{M\times\bR} (p,z)$, so we want to estimate $\dist_{M\times\bR} (p,z)$. 

We denote $r_p= r(p_M)$, $r_z = r(z_M)$, $r_q=r(q_M)$. Since $\phi_-(r_p)<\phi_+(r_p) =\phi_-(r_z)$, we have $r_p< r_z$. We also have $\phi_+(r_p)=\phi_-(r_p)+O(r_p^{-1+\varepsilon})$, hence
\begin{equation*}
\phi_-(r_z)-\phi_-(r_p)=\phi_+(r_p)-\phi_-(r_p)= O(r_p^{-1+\varepsilon}).
\end{equation*} 
As a consequence, there exists $\beta\in [0;1]$ such that 
\begin{equation*}
\phi'_-\left(\beta r_z+(1-\beta)r_p\right)(r_z-r_p)=\phi_-(r_z)-\phi_-(r_p)=O(r_p^{-1+\varepsilon}). 
\end{equation*} 
Since $\phi'_-(r)=1+O(r^{-1})$ we conclude that $r_z-r_p=O(r_p^{-1+\varepsilon})$ so
\begin{equation*}
\dist_{M\times \bR}(p,z) = \int_{r_p}^{r_z} \frac{dr}{\sqrt{1+r^2}} \leq \frac{r_z-r_p}{\sqrt{1+r_p^2}}=O(r_p^{-2+\varepsilon}). 
\end{equation*} 

In order to prove the claim it only remains to replace $O(r_p^{-2+\varepsilon})$ by $O(r_q^{-2+\varepsilon})$ in the right hand side of this inequality. This can be achieved by estimating
\begin{equation*}
\begin{split}
C r_p^{-2+\varepsilon}&\geq \dist_{M\times \bR}(p,z)\\ 
                      & \geq \dist_{M\times \bR}(p,q) \\ 
									    & \geq \dist_{M}(p_M,q_M)\\
											& \geq \left|  \int_{r_p}^{r_q} \frac{dr}{\sqrt{1+r^2}} \right|\\ 
											& \geq \frac{|r_q-r_p|}{\sqrt{1+(r_p + r_q)^2}}\\
											& = \frac{\left| r_q r_p^{-1} -1 \right|}{\sqrt{r_p^{-2} + (r_q r_p^{-1} + 1)^2}} \\
											& \geq \frac{\left| r_q r_p^{-1} -1 \right|}{\sqrt{2} (r_q r_p^{-1} + 1)},
\end{split}
\end{equation*}
which clearly implies that $r_q r_p^{-1}$ is bounded when $r_p\to \infty$. We conclude that $h(q)=O(r_q^{-2+\varepsilon})$. \\

\emph{Proving that $\partial h=O(r^{-1+\varepsilon})$.} We slightly modify the argument in the proof of Lemma \ref{lemNormals}. We fix a point $p\in \Sigma$, and let $p_0$ denote the orthogonal projection of $p$ on $\Sigma_-$. For $\rho_0 = h(p_0)$ consider the function $\Phi= \Phi(\rho) \definedas  \rho_0  - \rho$. Arguing as in the proof of Lemma  \ref{lemNormals}, we conclude that $|\Hess^\Sigma \Phi| \leq C$. Then if $\gamma$ is a unit speed geodesic in $\Sigma$ such that $\gamma(0)=p$ and $\gamma(s)=q$, we have
\begin{equation}\label{eqBoundHeight2}
%\begin{split}
\Phi(q) \geq d \Phi(\dot{\gamma}(0)) \dist_\Sigma (p,q) - C (\dist_\Sigma (p,q))^2,
%\end{split}					
\end{equation}		
cf. \eqref{eqTaylor}. We have already proven that there exists $C_0$ such that $h(z) \leq C_0 r(z)^{-2 + \varepsilon}$ for any $z\in \Sigma_-$ so we set $s = \sqrt{\delta}$, where $\delta \definedas 5 C_0 r(p)^{-2 + \varepsilon}$.  

Let $q_0$ denote the orthogonal projection of $q$ on $\Sigma_-$.  Again, it is straightforward to show that  $\frac{r(p_0)}{2} \leq r(q_0) \leq 2 r(p_0)$ provided that $r(p)$ is sufficiently large. The left hand side of \eqref{eqBoundHeight2} can then be estimated as follows
\begin{equation}\label{eqDistBound2}
\begin{split}
|\Phi(q)| &  =| h(p_0)-h(q_0)| \\
						& \leq C_0 r(p_0)^{-2+\varepsilon} + C_0 r(q_0)^{-2+\varepsilon}	\\
						& \leq 5C_0 r(p_0)^{-2+\varepsilon}\\ 
						& = \delta.
\end{split}
\end{equation}
Consequently, it follows from \eqref{eqBoundHeight2} that $d\Phi (\dot{\gamma}(0)) \leq C \sqrt{\delta}$ for some $C>0$. If we now choose $\dot{\gamma}(0) = \frac{\nabla^\Sigma \Phi}{|\nabla^\Sigma \Phi|}$ then at the point $(p_0,\rho_0) = (p_0, h(p_0))$ we have
\begin{equation}
\begin{split}
C \sqrt{\delta} & \geq \frac{\langle \nabla \Phi, \nabla^\Sigma \Phi\rangle}{|\nabla^\Sigma \Phi|} \\
                & =\frac{\langle\partial_\rho,\partial_\rho - \langle  \nu,\partial_\rho\rangle\nu \rangle }
								    {\sqrt{1-\langle \nu, \partial_\rho\rangle^2}}\\
								& = \sqrt{1-\langle \nu, \partial_\rho\rangle^2},
\end{split}
\end{equation} 
where, as before, 
\[
\nu= \frac{\partial_\rho-\nabla^{g_\rho} h}{\sqrt{1 + |dh|^2_{g_\rho}}}
\]
is the upward pointing unit normal of $\Sigma$. It follows that
\[
1 - \frac{1}{1 + |dh|^2_{g_{\rho}}} = O(\delta),
\]
and hence 
$|dh|^2_{g_{\rho}}=O(\delta)$ at the point $(p_0,\rho_0) = (p_0, h(p_0))$. Recalling the definition of $\delta$ and Propostion \ref{propLevel} the second claim follows. \\

\emph{Proving that $|\partial \partial h|=O(1)$.} Combining the argument used in the proof of Lemma \ref{lemDefFuncIneq} with Propostion \ref{propLevel} one can obtain the following estimate for the second fundamental form: 
\[
\sum_{i,j=1}^3 A^\Sigma(e_i,e_j)^2 \leq C (1 + |\partial h|^2)^2,
\]
where $e_i\definedas \partial_i +(\partial_i h) \partial_\rho$ for $i=1,2,3$. It is also straightforward to check that
\[
A^\Sigma(e_i, e_j)^2 = \left(\partial_i \partial_j h - (\Gamma_\rho)^k _{ij} h_k + (A_\rho)_{ij}
				         +2 (A_\rho)_{\phantom{i} i}^k h_j h_k \right)^2 (1 + |dh|^2_{g_\rho})^{-1},					
\]
see Section \ref{secHeight} for details. Using the inequality $(a+b)^2 \geq \frac{a^2}{2} -  b^2$, Proposition \ref{propLevel} and the fact that $|\partial h| = O(r^{-1+\varepsilon})$, the last claim of the proposition follows. 
\end{proof}

\subsection{The height function: a posteriori estimates}\label{secHeight}

We begin this section by rewriting the Jang equation in terms of the height function using Fermi coordinates. 
For this purpose it is convenient to  think of  $\Sigma$ as the level set $\{F = 0\}$ of the function  $F(u,\rho)=h(u)-\rho$. A computation shows that
\begin{align*}
\Hess_{\rho \rho} F & =0, \\ 
\Hess_{\rho i} F & = (A_\rho)^{k}_{\phantom{i} i} h_k, \\
\Hess_{ij} F & = \Hess^{g_\rho}_{ij} h + (A_\rho)_{ij},
\end{align*}
where we, as before, use the notation $h_i= \partial_i h$, and tacitly assume that $i,j,k\in\{1,2,3\}$, and that the indices are raised with respect to the metric $g_\rho$.  We remind the reader that our sign convention for the second fundamental form of the surfaces $\{\rho=\const\}$ is $(A_\rho )_{ij}= \langle \nabla_{\partial_i} (-\partial_\rho), \partial_j \rangle$.

In this setting, the vector $-\partial_\rho+\nabla^{g_\rho} h$ is normal, and the vectors $e_i = \partial_i+(\partial_i h)\partial_\rho$ are tangent to $\Sigma$ at the point with Fermi coordinates $(u,\rho)=(u, h(u))$. The induced metric on $\Sigma$ has components 
\begin{equation*}
\overline{g}_{ij} \definedas \ghat (e_i,e_j)=(g_\rho)_{ij}+h_i h_j,
\end{equation*}  
and its inverse is
\begin{equation}\label{eqFermiInverse}
\gbar^{ij}= (g_\rho)^{ij}-\frac{h^i h^j}{1+|dh|^2_{g_\rho}}.
\end{equation}
The mean curvature of $\Sigma$ is then given by
\begin{equation}\label{eqMeanCurvFermi}
\begin{split}
H^\Sigma & = \bar{g}^{ij} A^\Sigma (e_i, e_j) \\
         & = \frac{\bar{g}^{ij}\Hess F \left(\partial_i+(\partial_i h)\partial_\rho, \partial_j+(\partial_j h)\partial _\rho\right)}{|\nabla F|}\\
         & = \frac{\bar{g}^{ij}\left(\Hess^{g_\rho}_{ij} h + (A_\rho)_{ij}+2 (A_\rho)_{\phantom{i} i}^k h_j h_k \right)}{\sqrt{1+|dh|^2_{g_\rho}}},
\end{split}
\end{equation}
and the trace of $K$ with respect to the induced metric on $\Sigma$ is given by
\begin{equation}\label{eqTrKFermi}
%\begin{split}
                   \tr^\Sigma K  = \bar{g}^{ij} K(e_i,e_j)
                                 =  \bar{g}^{ij} (K_{ij}+2 h_i K_{\rho j}+h_i h_j K_{\rho \rho}).
%\end{split}
\end{equation} 
Note that all quantities in the equations \eqref{eqMeanCurvFermi} and \eqref{eqTrKFermi}  are computed at the point with Fermi coordinates $(u,\rho)=(u, h(u))$. 

We may now rewrite the Jang equation $H^\Sigma  - \tr^\Sigma K = 0$ in terms of the height function as follows.

\begin{proposition}
The height function $h$ satisfies the equation
\begin{equation}\label{eqJangHeight}
a^{ij}\partial_i\partial_j h + b^k \partial_k h =c,
\end{equation}
with the coefficients given by
\[
a^{ij}  =  \frac{\bar{g}^{ij}}{\sqrt{1+|dh|_{g_\rho}^2}}, \hspace{1cm}
  b^k   = - \frac{\bar{g}^{ij}(\Gamma_\rho)^k_{ij}}{\sqrt{1+|dh|_{g_\rho}^2}} - 2 \bar{g}^{ik} K_{i\rho},
\]
\[
c= \bar{g}^{ij} \left(-\frac{(A_\rho)_{ij}+2 (A_\rho)_{\phantom{i} i}^k h_j h_k}
		         {\sqrt{1+|dh|^2_{g_\rho}}}+K_{ij} 
						+ h_i h_j K_{\rho \rho}  \right),
\]
where $(\Gamma_\rho)^k_{ij}$ are the Christoffel symbols of the metric $g_\rho$, and $\gbar^{ij}$ is given by \eqref{eqFermiInverse} .
\end{proposition}

Applying standard elliptic theory and rescaling technique to \eqref{eqJangHeight} we will obtain our a posteriori estimates for the height function. It will be convenient to use the following definition (see e.g. \cite{Bartnik}, \cite{ChruscielDelayOnMapping}, \cite{Meyers}).

\begin{definition}
 Let $B$ be a closed ball in $\bR^n$ with center at the origin. For every $k\in\{0,1,2,\ldots\}$, $\alpha \in (0,1)$ and $\tau \in \bR$ we define the weighted H\"older space $C^{k,\alpha}_\tau(\bR^n \setminus B)$ as the collection of $f\in C^{k,\alpha}_{loc}(\bR^n \setminus B)$ with 
\[
\sum_{|I|\leq k} \sup_{x\in \bR^n \setminus B} |x|^{|I|+\tau} |(\partial^{(I)} f)(x)| + \sup_{x\in \bR^n \setminus B} |x|^{k+\tau+\alpha} \sup_{4|x-y|<|x|} \frac{|\partial^k f(x)-\partial^k f(y)|}{|x-y|^\alpha} < \infty.
\]
\end{definition}

\begin{remark}
This definition extends in a standard way (see e.g. \cite[Definition 1]{EHLS}) to define the weighted H\"older space $C^{k,\alpha}_\tau(M)$ on a $C^k$ manifold $M$ which outside of a compact set is diffeomorphic to $\bR^n \setminus B$ as well as to the case of tensor bundles on $M$. In what follows, we will write $C^{k,\alpha}_\tau$ instead of $C^{k,\alpha}_\tau(M)$ whenever the context is clear and denote by $O^{k,\alpha} (r^{-\tau})$ a tensor in the weighted H\"older space $C^{k,\alpha}_\tau$.
\end{remark}

\begin{proposition}\label{propFallOffh}
The height function $h$ satisfies $h=O^{2,\alpha} (r^{-2 + \varepsilon})$ and $|\partial\partial\partial h |= O^\alpha(r^{-4+\varepsilon})$ for some $\alpha\in (0,1)$.
\end{proposition}

\begin{remark}
The positive constant $\varepsilon$ may be assumed to be arbitrarily small by choosing an appropriate $r_0>0$ in Lemma \ref{lemBarrierAsympt} and a sufficiently small $\beta>0$ in the proof below. Since we are not interested in the explicit  form of $\varepsilon$, in what follows we will mostly let $\varepsilon>0$ denote a generic constant possessing the above properties.
\end{remark}

\begin{proof}
We prove the proposition by completing the following steps. \\

\emph{Proving that $\partial h=O(r^{-2+\varepsilon})$.} From Lemma \ref{lemPropertiesHeight} we know that $h=O(r^{-2+\varepsilon})$, $\partial h = O(r^{-1+\varepsilon})$, $\partial \partial h = O(1)$. Consequently, Proposition \ref{propLevel} implies that $a^{ij}$ is bounded in $C^1$ norm, that $b^k$ is bounded, and that the equation is uniformly elliptic. It is also clear that
\begin{equation}\label{eqF}
c  =  -H_\rho + \tr^{g_\rho} K 
+ O(|\partial h|^2),
\end{equation}
where $H_\rho$ is the mean curvature of $\Sigma_\rho$. 

In order to estimate the coefficient $c$ more accurately, recall that $A_\rho$ satisfies the Mainardi equation
\begin{equation*}
-\partial_\rho (A_\rho)^i_{\phantom{i}j} + (A_\rho)^i_{\phantom{i}k} (A_\rho)^k_{\phantom{k}j}=R^i_{\phantom{i} \rho \rho j},
\end{equation*}
see Appendix \ref{secLevelFermi}  for details. Taking the trace, we obtain 
\begin{equation*}
\partial_\rho H_{\rho} = \ric (\partial_\rho, \partial_\rho) + |A_\rho|^2.
\end{equation*}
Differentiating with respect to $\rho$ one more time, we get 
\begin{equation*}
\begin{split}
\partial^2_{\rho \rho} H_\rho & = 2 \partial_\rho (A_\rho)^i_{\phantom{i}k} (A_\rho)^k_{\phantom{k}i} + \partial _\rho \ric (\partial_\rho, \partial_\rho)\\
                        & = 2 (A_\rho)^i_{\phantom{i}l}(A_\rho)^l_{\phantom{l}k} (A_\rho)^k_{\phantom{k}i} - 2 R^i_{\phantom{i}\rho \rho k} (A_\rho)^k_{\phantom{k}i}  +                             \partial _\rho \ric (\partial_\rho, \partial_\rho).
\end{split}
\end{equation*}
Since  $\nabla_{\partial_\rho} \partial_\rho=0$, we have $\partial _\rho \ric (\partial_\rho, \partial_\rho) =\left(\nabla_{\partial_\rho} \ric\right) (\partial_\rho, \partial_\rho)$. Consequently, 
\begin{equation*}
|\partial^2_{\rho \rho}H_\rho| \leq 2 |A_\rho|^3 + 2|A_\rho||R| + |\nabla \ric|,  
\end{equation*}
which is bounded for all $\rho \in [0,\rho_0]$ by Proposition \ref{propLevel} and by our assumptions on the initial data. As a consequence, by Lemma  \ref{lemma2}, for $\rho=h(u)=O(r^{-2+\varepsilon})$ we obtain 
\begin{equation*}
\begin{split}
H_\rho & =  H_0+ (\partial_\rho H_\rho)_{|_{\rho = 0}} \rho +O(\rho^2)\\
			  & = H^{\Sigma_-} +\left(\ric (\nu_-, \nu_-) + \left|A^{\Sigma_-}\right|^2\right)\rho+O(r^{-4+\varepsilon})\\
				& = H^{\Sigma_-} + O(r^{-4+\varepsilon}).
\end{split}
\end{equation*}

We also need to estimate $\tr^{g_\rho} K$. For this we note that 
\[
\tr^{g_\rho} K = \tr^{\ghat} K - K_{\rho \rho}.
\]
Again, in the view of $\nabla_{\partial_\rho} \partial_\rho=0$ we have
\begin{equation*}
\partial_\rho (\tr^{g_\rho} K) = \partial_\rho (\tr^{\ghat} K) - (\nabla_{\partial_\rho} K)(\partial_\rho, \partial_\rho)
\end{equation*} 
and
\begin{equation*}
\partial^2_{\rho\rho} (\tr^{g_\rho} K) = \nabla_{\partial_\rho} \nabla_{\partial_\rho} (\tr^{\ghat} K) - (\nabla_{\partial_\rho} \nabla_{\partial_\rho} K)(\partial_\rho, \partial_\rho).
\end{equation*}
In particular, we see that $\partial^2_{\rho\rho} (\tr^{g_\rho} K)$ is bounded for any $\rho \in [0,\rho_0]$. As a consequence,  using the asymptotic properties of $K$ (see Section \ref{secPrelim}), we obtain 
\begin{equation*}
\begin{split}
\tr^{g_\rho} K & = \tr^{g_0} K + \partial_\rho (\tr^{g_\rho} K) _{|_{\rho=0}} \rho + O(\rho^2)\\
               & = \tr^{\Sigma_-} K + \left(\nabla_{\nu_-} (\tr^{\ghat} K) - \left( \nabla_{\nu_-} K \right) (\nu_-, \nu_-)\right) \rho +O(\rho^2)\\
						   & = \tr^{\Sigma_-} K + O(r^{-4+\varepsilon}).
\end{split}												
\end{equation*}
Recall now that $\Sigma_-$ is a graphical hypersurface such that \eqref{eqAlphaMass} and \eqref{eqPsi} hold. It follows from \eqref{eqJangEqExpansion} that $H^{\Sigma_-} - \tr^{\Sigma_-} K = O(r^{-4+\varepsilon})$, which implies
\begin{equation}\label{eqF1}
c=O(r^{-4+\varepsilon}) + O(|\partial h|^2) = O(r^{-2 + \varepsilon}).
\end{equation}
Applying elliptic regularity in the balls of fixed radius followed by Sobolev embedding we conclude from \eqref{eqJangHeight} that $|h |_{C^{1,\alpha} (B_2 (p))}=O(r(p)^{-2+\varepsilon})$ for any $p\in \Sigma_-$ with sufficiently large $r(p)$. The estimate  $\partial h=O(r^{-2+\varepsilon})$ follows.\\

\emph{Proving that $h=O^{2,\alpha} (r^{-2 + \varepsilon})$.} Note that interior Schauder estimates and a standard bootstrap argument, in the view of our assumptions on the initial data and Lemma  \ref{propLevel}, yield $|h|_{C^{3,\alpha} (B_{3/2} (p))} = O(r(p)^{-2+\varepsilon})$. In order to improve this estimate we fix a point $p_0\in \Sigma_-$ with asymptotically Euclidean coordinates $u_0=(u_0^1,u_0^2,u_0^3)$ and define the coordinates $\tilde{u}=\frac{u-u_0}{\sigma}$, where $\sigma=r_0/2$ for $r_0 =r(u_0)= |u_0|$. In terms of $\tilde{u}$, our equation becomes
\begin{equation}\label{eqRescaled}
a^{ij} \tilde{\partial}_i \tilde{\partial}_j h + \sigma b^k \tilde{\partial}_k h = \sigma^2 c,
\end{equation} 
where $\tilde{\partial}_i =  \partial_{\tilde{u}^i}$ for $i=1,2,3$. We will consider this equation in $U_{3/2} \definedas \{|\tilde{u}|< 3/2$\}. We will use the notation  $U_R= \{|\tilde{u}|< R \}$ throughout the proof. % We will now show that the $C^{0,\beta}(U_1)$ norm of the coefficients of this equation is bounded, and that $C^{0,\beta}(U_1)$ norm of its right hand side is $O(r_0^{-2+\varepsilon})$ for some $0<\beta<1$.

Recall that the coefficients $a^{ij}$, $b^k$ and $c$ of the equation \eqref{eqJangHeight} are computed at the point $(u,\rho) = (u,h(u))$, so the chain rule must be applied whenever  these coefficients are differentiated with respect to $u^i$, $i=1,2,3$. For instance, in the case of $a^{ij}$ we have
\begin{equation}\label{eqChain}
\partial_k (a^{ij}(u, h(u))) = (\partial_k a^{ij} + \partial_\rho a^{ij} \partial_k h) (u, h(u)),
\end{equation}
where 
\begin{equation*}
\begin{split}
\partial_k a^{ij} + \partial_\rho a^{ij} \partial_k h & = \partial_k (g_\rho) ^{ij} + \partial_\rho (g_\rho)^{ij} \partial_k h + O(r^{-4+\varepsilon})\\ 
                                                      & = \partial_k (g_0) ^{ij} + O(r^{-2 + \varepsilon}) \\
																											& = O(r^{-2 + \varepsilon})
\end{split}
\end{equation*}
in the view of the above estimate for $h$, Proposition \ref{propLevel} and Lemma \ref{lemma2}. It follows that $\max_{U_{3/2}} |\tilde{\partial}_k a^{ij}|=O(r_0^{-1+\varepsilon})$. 

Let $b^k_1 = - 2 \bar{g}^{kj} K_{\rho j}$ and $b^k_2 = - a^{ij} (\Gamma_\rho)^k_{ij}$ so that $b^k = b^k_1 + b^k_2$. In order to estimate the $C^{0,\beta}(U_{3/2})$-norm of $b_1^k$ we first note that 
\begin{equation*}
\partial_\rho (dt(\partial_\rho)) = (\nabla_{\partial_\rho} dt) (\partial_\rho) + dt (\nabla_{\partial_\rho} \partial_\rho)=0,
\end{equation*}
hence $(dt(\partial_\rho))(u,\rho) = (dt(\partial_\rho))(u,0) = (dt(\nu_-))(u)$. In the view of Lemma \ref{lemma2} we then have
\begin{equation}\label{eqExtraTerm}
\begin{split}
K(\partial_\rho, \partial_j) & = \ghat (\partial_\rho, \partial_j) - dt(\partial_\rho) dt(\partial_j)+ (K-g)(\partial_\rho, \partial_j) \\
                             & = - dt(\nu_-) dt(\partial_j)+ (K-g)(\partial_\rho, \partial_j) \\
														 & = O(r^{-1})
\end{split}
\end{equation}
where we have also used the fact that $r(p)$ and $r(p_-)$ are comparable, cf. the proof of Corollary \ref{corGraph}. It follows that $b_1^k = O(r^{-1})$. Furthermore, using \eqref{eqExtraTerm}, we also obtain
\begin{equation*}
\begin{split}
\partial_\rho K(\partial_\rho, \partial_j) & = (\nabla_{\partial_\rho} K)(\partial_\rho, \partial_j) + K(\partial_\rho, \nabla_{\partial_\rho}\partial_j)\\
                                           & = - K_{\rho k} (A_\rho)^k_{\phantom{k} j} + O(r^{-4}) \\
																					 & = O(r^{-1}).
\end{split}
\end{equation*}
Similarly, differentiating  \eqref{eqExtraTerm} and using Proposition \ref{propLevel}, Lemma \ref{lemma2}, and the fact that 
$\nabla_{\partial_l} \partial_j =( \Gamma_\rho)^k_{lj} \partial_k$ where  $(\Gamma_\rho)_{lj}^k = (\Gamma_0)_{lj}^k+O(r^{-2+\varepsilon})=O(r^{-2+\varepsilon})$ we conclude that
\begin{equation*}
\begin{split}
\partial_l K(\partial_\rho, \partial_j) %= & \partial_l  \left( - dt(\partial_\rho) dt(\partial_j)+ (K-g)(\partial_\rho, \partial_j) \right) \\
                               = & - \partial_l (dt(\nu_-) ) dt(\partial_j) - dt(\nu_-) dt(\nabla_{\partial_l} \partial_j)\\ & \quad+ \nabla_{\partial_l}(K-g)(\partial_\rho, \partial_j)  - (K-g)(\nabla_{\partial_l} \partial_\rho, \partial_j)  - (K-g)( \partial_\rho, \nabla_{\partial_l}\partial_j) \\
														  =& O(r^{-2}).
\end{split}
\end{equation*}
Applying the chain rule as in  \eqref{eqChain}, we conclude that $\max_{U_{3/2}} |\tilde{\partial}_l b_1^k| = O(r_0^{-1})$.  

A similar argument shows that $b_2^k= O(r^{-2+\varepsilon})$ and $\partial_l b_2^k = O(r^{-2+\varepsilon})$ hence $\tilde{\partial}_l b_2^k = O(r_0^{-1 + \varepsilon})$ on $U_{3/2}$. This gives us the  estimate 
\begin{equation}\label{eqHolder}
\begin{split}
\frac{\sigma |b_2^k(\tilde{x})-b_2^k(\tilde{y})|}{|\tilde{x}-\tilde{y}|^{\beta}} & = \sigma \frac{|b_2^k(\tilde{x})-b_2^k(\tilde{y})|^\beta}{|\tilde{x}-\tilde{y}|^{\beta}}|b_2^k(\tilde{x})-b_2^k(\tilde{y})|^{1-\beta} \\ & \leq C r_0^{1+(-1+\varepsilon) \beta + (-2 + \varepsilon)(1 - \beta)} \\ & = C r_0^{-1 + \varepsilon +\beta}
\end{split}
\end{equation}
for $\tilde{x}, \tilde{y} \in U_1$. Hence 
\[
\|\sigma b_2^k\|_{C^{0,\beta}(U_{3/2})}=O(r_0^{-1 + \varepsilon +\beta}).
\]

Further, in the view of $\partial h=O(r^{-2+\varepsilon})$ the estimate \eqref{eqF1} improves  and we obtain $c = O(r^{-4+2\varepsilon})$. Combining  the formulas that we obtained when proving the estimate $\partial h=O(r^{-2+\varepsilon})$ with Lemma \ref{lemma2} and \eqref{eqJangEqExpansion}, we also find that
\begin{equation*}
\begin{split}
\partial_l c= &- \partial_l (H_\rho - \tr^{g_\rho} K) + O(r^{-4+\varepsilon}) \\
= &-\partial_l (H^{\Sigma_-} - \tr^{\Sigma_-} K )-\partial_l \left(\ric (\nu_-, \nu_-) + \left|A^{\Sigma_-}\right|^2\right)\rho \\ & \quad + \partial_l \left(\nabla_{\nu_-} (\tr^{\ghat} K) - \left( \nabla_{\nu_-} K \right) (\nu_-, \nu_-)\right) \rho+O(r^{-4+2\varepsilon})\\
= & O(r^{-4+2\varepsilon})
\end{split}
\end{equation*}
and 
\begin{equation*}
\begin{split}
\partial_\rho c= - \partial_\rho(H_\rho - \tr^{g_\rho} K) _{|_{\rho = 0}} + O(r^{-2+\varepsilon}) 
=  O(r^{-2+\varepsilon}).
\end{split}
\end{equation*}
Applying the chain rule as in  \eqref{eqChain} and estimating as in \eqref{eqHolder} we find that
\[
\|\sigma^2 c\|_{C^{0,\beta}(U_{3/2})}=O(r_0^{-2 + 2 \varepsilon +\beta}).
\]
We are now in a position to apply interior  Schauder estimates which gives
\[
\|h\|_{C^{2,\beta}(U_{5/4})}=O(r_0^{-2+2 \varepsilon + \beta}).
\]

Changing back to the unrescaled coordinates $u=(u^1,u^2,u^3)$, %we see that $h(u_0) = O_2 (r_0^{-2+\varepsilon + \beta})$. It follows that $h=O_2 (r^{-2+\varepsilon})$ up to redefining $\varepsilon$.
the estimate follows up to redefining $\varepsilon$.\\

\emph{Proving that $|\partial\partial\partial h| = O^\alpha(r^{-4+\varepsilon})$.}
% $\|\partial\partial\partial h\|_{C^{0,\beta}(B_{r(x)/2}(x))}=O(r(x)^{-4+\varepsilon-\beta})$.}
Recall that $\max_{U_{3/2}} |\tilde{\partial}_k a^{ij}|=O(r_0^{-1+\varepsilon})$ and that the second derivatives of $a^{ij}$ with respect to unrescaled coordinates $u_k$ are bounded.  Estimating as in  \eqref{eqHolder} we conclude that $\|a^{ij}\|_{C^{1,\beta}(U_{5/4})} = O(r_0^{-1+\varepsilon})$. Further,  in the view of $h=O_2 (r^{-2+\varepsilon})$ we have $\partial_l b_2^k = O(r^{-3+\varepsilon})$. Again, the second derivatives of $b_2^k$ with respect to unrescaled coordinates  are bounded so it follows that  $\|\sigma b_2^k\|_{C^{1,\beta}(U_{5/4})} $ is bounded along the lines of  \eqref{eqHolder}. Furthermore, one can check that %is bounded  relying on the earlier estimate   $\max_{U_1} |\tilde{\partial}_l b^{k}|=O(r_0^{-1})$ and on the fact that 
the second order derivatives of $b_1^k$ in the unrescaled coordinates are of order $O(r^{-3})$ which implies boundedness of $\|\sigma b_1^k\|_{C^{1,\beta}(U_{5/4})} $. Finally, using the earlier estimate $\partial_l c = O(r^{-4+\varepsilon})$ and the boundedness of $|\partial\partial c|$ we obtain
\[
\|\sigma^2 c\|_{C^{1,\beta}(U_{5/4})}=O(r_0^{-1 +\varepsilon})
\]
up to redefining $\varepsilon$. The desired estimate follows by applying  interior  Schauder estimates and changing back to the unrescaled coordinates.  
\end{proof}

\begin{remark}
 Note that the above method does not allow us to prove the expected estimate $|\partial\partial\partial h| = O^\alpha(r^{-5+\varepsilon})$ due to the fact that the estimate $\partial_l c = O(r^{-4+\varepsilon})$ cannot be improved to $\partial_l c = O(r^{-5+\varepsilon})$ unless we include more terms in the Taylor expansion. At the same time, a much weaker estimate   $|\partial\partial\partial h|=O^\alpha(r^{-2+\varepsilon})$ would suffice for our purposes, as one can see by inspecting the proofs below.
\end{remark}

%\begin{remark}
%In particular,  the estimate $\|\partial\partial\partial h\|_{C^{0,\beta}(B_{r(x)/2}(x))}=O(r(x)^{-4+\varepsilon-\beta})$ implies that the third order partial derivatives of $h$ are in the weighted H\"older space $C^{0,\beta}_{-4+\varepsilon}$.
%\end{remark}

\begin{cor}\label{corAEh}
The induced metric $\bar{g}$ on the Jang graph is asymptotically Euclidean such that
\begin{equation}\label{eqFallOffMetricAE2}
\bar{g} = g_{\Sigma_-} + O^{2,\beta}(r^{-2+\varepsilon}).
\end{equation}
In particular, the ADM masses of the metrics $\bar{g}$ and $g_{\Sigma_-}$ are equal:
\begin{equation}\label{eqADMMasses}
\mathcal{M}(\bar{g})=\mathcal{M}(g_{\Sigma_-} )=\alpha = 2E.
\end{equation}
\end{cor}

\begin{proof}
We perform the computation in the asymptotically Euclidean coordinate chart $\Psi$ as described in  Lemma \ref{lemma1}.  Let $e=\bar{g} - g_{\Sigma_-}$. 
Using $h=O_2(r^{-2+\varepsilon})$ we compute as in the proof of Proposition \ref{propFallOffh} that
\begin{equation*}
\begin{split}
e_{ij}  = (g_\rho)_{ij} - (g_0)_{ij} +h_i h_j =\partial_\rho (g_\rho)_{ij}\vert_{\rho=0}\,\rho  +  O(r^{-4+\varepsilon}) = O(r^{-2+ \varepsilon}),
\end{split}
\end{equation*}
where $\partial_\rho (g_\rho)_{ij} = - 2 (A_\rho)^k_{\phantom{k}i} (g_\rho)_{kj}$. Similarly, in the view of Lemma \ref{lemma2} we obtain
\begin{equation*}
\begin{split}
\partial_l e_{ij}  & = \partial_l\left((g_\rho)_{ij} -(g_0)_{ij}\right)  + \partial_\rho \left((g_\rho)_{ij} -(g_0)_{ij}\right) h_l + O(r^{-7+\varepsilon}) \\ & = \partial_l \left(\partial_\rho (g_\rho)_{ij}\vert_{\rho=0}\right)\rho + O(r^{-3+\varepsilon}) \\ & = -2\partial_l (A_0)_{ij} \rho + O(r^{-3+\varepsilon}) \\ &= O(r^{-3+\varepsilon}).
\end{split}
\end{equation*}
Recalling $\partial \partial \partial h=O(r^{-4+\varepsilon})$, by Lemma \ref{lemma2} we also have
\begin{equation*}
\begin{split}
\partial_k \partial_l e_{ij}  & =   \partial_k \partial_l \left((g_\rho)_{ij} -(g_0)_{ij}\right)  + \partial_\rho \partial_l \left((g_\rho)_{ij} -(g_0)_{ij}\right) h_k \\ &\qquad + \partial_k \partial_\rho \left((g_\rho)_{ij} -(g_0)_{ij}\right) h_l +  \partial^2_\rho \left((g_\rho)_{ij} -(g_0)_{ij}\right) h_l  h_k \\ &\qquad + \partial_\rho\left((g_\rho)_{ij} -(g_0)_{ij}\right) \partial_k \partial_l h + O(r^{-7+\varepsilon}) \\
& =  -2 \partial_k \partial_l (A_0)_{ij} \rho -2 \partial_l (A_0)_{ij}h_k -2 \partial_k (A_0)_{ij}h_l +O(r^{-4+\varepsilon}) \\
& =  O(r^{-4+\varepsilon}). 
\end{split}
\end{equation*}

It follows that $\bar{g} = g_{\Sigma_-} + O_2(r^{-2+\varepsilon})$. In particular, this implies that $\mathcal{M}(\bar{g})=\mathcal{M}(g_{\Sigma_-} )$ so \eqref{eqADMMasses} follows by Lemma \ref{lemma1}.

To complete the proof it remains to show that $\partial_k \partial_l  e_{ij}=O^{\beta}(r^{-4+\varepsilon})$. For this,
 we write  $e_{ij}= (e_{ij} -h_i h_j) + h_i h_j$. % we write like this because we cannot differentiate $h$ four times, otherwise we could do as before.
 The third order coordinate derivatives of the first term are bounded and the second order coordinate derivatives fall off as  $O(r^{-4+\varepsilon})$, see the above computation. Arguing as in the  the proof of Proposition \ref{propFallOffh}), we conclude that $\partial_k \partial_l  (e_{ij} -h_i h_j) =O^{\beta}(r^{-4+\varepsilon})$. That $\partial_k \partial_l (h_i h_j) =O^{\beta}(r^{-4+\varepsilon})$ is a direct consequence of Proposition \ref{propFallOffh}.
\end{proof}

We have now all ingredients ready for proving  \eqref{eqFallOff3D}.

\begin{proof}[Proof that $f=\sqrt{1+r^2}+\alpha \ln r+\psi(\theta, \varphi)+O_3(r^{-1+\varepsilon})$.]
We write 
\[
f =\sqrt{1+r^2}+\alpha \ln r+\psi(\theta, \varphi)+\eta(r,\theta, \varphi), 
\]
where $\eta = O(r^{-1+\varepsilon})$. On the one hand, \eqref{eqFallOffMetricAE2} implies that
\begin{equation*}
\begin{split}
\bar{g}_{rr} & = g_{rr} + (\partial_r f_-)^2 + O(r^{-2+\varepsilon})\\ & = \frac{1}{1+r^2} +  \left( \frac{r}{\sqrt{1+r^2}} + \frac{\alpha}{r} + O(r^{-2+\varepsilon})\right)^2 + O(r^{-2+\varepsilon}) \\ & = \frac{1}{1+r^2} +  \left( \frac{r}{\sqrt{1+r^2}} + \frac{\alpha}{r} + O(r^{-2+\varepsilon})\right)^2.
\end{split}
\end{equation*}
On the other hand, we have
\begin{equation*}
\bar{g}_{rr} = g_{rr} + (\partial_r f)^2 = \frac{1}{1+r^2} + \left( \frac{1}{\sqrt{1+r^2}} + \frac{\alpha}{r} + \eta'_r \right) ^2.
\end{equation*}
It follows that $\eta'_r = O(r^{-2+\varepsilon})$. Note that when comparing the two expressions we have tacitly relied on the fact that $r(p)=r(p_-) + O(r(p_-)^{-2+\varepsilon})$.

With this estimate at hand, one finds that $\eta'_\mu = O(r^{-1+\varepsilon})$ by computing the components of $\bar{g}_{r\mu} $ in two different ways as discussed above. 

Estimates for the second and third order derivatives follow in a similar way.
\end{proof}

\section{The conformal structure of the Jang graph}\label{secConformal}
As in Section 6, we denote by  $(\Sigma, \bar{g})$  the graphical component of the geometric solution of the Jang equation. The graphing function is denoted by $f$, and it is assumed that its domain  $U$ contains the region $\{r \geq r_0\}$. The goal of this section is to show that $\Sigma$ admits a metric satisfying the conditions of positive mass theorem for asymptotically Euclidean manifolds, that is, a complete metric with nonnegative scalar curvature. This metric is constructed  mostly following \cite{PMT2} (see also \cite{EichmairJang}), although we need to take care of some additional complications arising from the fact that $\gbar - \delta$ has a somewhat slower fall-off rate as $r\to\infty$ in our setting. 

In this section $0<\varepsilon<1$ and $C>0$ are generic constants that may vary from line to line. The particular value is not important. 

\begin{proposition}\label{propConfStructure1}
The metric $\bar{g}=g+df\otimes df$ on $U\subset M$ is complete and $C^{2,\beta} _{loc}$. Its scalar curvature satisfies 
\begin{equation}\label{eqScalarJang}
\scal^{\bar{g}} = \frac{2 \Delta^{\bS^2} \psi}{r^3} + O(r^{-4+\varepsilon})
\end{equation}
and the integral inequality
\begin{equation}\label{eqSYIntegrated0}
\int_{\Sigma} \left( \scal^{\gbar} \varphi^2 + 2|d\varphi|^2_{\gbar} \right) \, d\mu^{\gbar} \geq \int_\Sigma \left( 2(\mu-|J|_g)\varphi^2 + |A-K|^2_{\gbar} \varphi^2 \right) \; d\mu^{\gbar}
\end{equation}
holds for $\varphi\in C^1_c(\Sigma)$.
 
As a consequence, if the strict dominant energy condition holds near $\partial U$ then the spectrum of the operator $-\Delta^{\gamma_i} + \frac{1}{8} \scal^{\gamma_i}$ is positive on each connected component $\partial U_i$ of $\partial U$, $i=1,\ldots, l$. In particular, each of $(\partial U_i, \gamma_i \definedas g_{|_{\partial U_i}})$, $i=1,\ldots, l$, is topologically a sphere. 
\end{proposition} 

\begin{proof}
We recall that the scalar curvature of $(\Sigma, \bar{g})$ can be computed using the Schoen and Yau identity \cite[(2.25)]{PMT2}: 
\begin{equation}\label{eqScalarJang1}
\scal^{\bar{g}} = 2(\mu - J (w)) + |A-K|^2_{\bar{g}} + 2 |q|^2_{\bar{g}} - 2 \divg^{\bar{g}} q, 
\end{equation}
where the 1-form $q$ is as defined in Lemma \ref{lemma2} and the vector field $w$ is such that $|w|_g<1$. Since $f$ satisfies \eqref{eqFallOff3D} the asymptotics of all terms in the right hand side of \eqref{eqScalarJang1} can be made precise using Lemma \ref{lemma2} and Definition \ref{defAHdata}. In particular, we see  that $\divg^{\bar{g}} q = - r^{-3} \Delta^{\bS^2}\psi  + O(r^{-4})$, while the remaining terms are of order $O(r^{-4+ \varepsilon})$ or lower. This proves \eqref{eqScalarJang}. 

It is also straightforward to check that \eqref{eqSYIntegrated0} holds by integrating \eqref{eqScalarJang1} against $\varphi^2$, where $\varphi\in C^1_c(\Sigma)$, and using a simple estimate
\[
 -\int_\Sigma 2 \varphi^2 \divg^{\bar{g}} q \, d\mu^{\gbar} 
= \int_\Sigma 4 \varphi q (\nabla^{\gbar} \varphi) \, d\mu^{\gbar} 
\geq - \int_\Sigma \left( 2|q|^2_{\gbar}\varphi^2 + 2|d\varphi|^2_{\gbar} \right) \, d\mu^{\gbar},
\]
together with the fact that $|w|_{g}< 1$.

The second part of the claim follows from the same separation of variables argument as in \cite[p. 254-255]{PMT2}. Suppose that $2(\mu-|J|_g) > \lambda > 0$ near $\partial U$, then using the fact that $\Sigma$ has ends that are $C^{3,\alpha}$ asymptotic to $(\partial U_i \times \bR, \gamma_i + dt^2)$, $i=1,\ldots, l$,  we obtain from \eqref{eqSYIntegrated0} the inequality
\begin{equation}\label{eqSYIntegrated12}
\int_{ \partial U\times \bR} \left( \scal^{\gamma} \varphi^2 + 2|d\varphi|^2_{\gamma+dt^2} \right) \; d\mu^{\gamma+dt^2} \geq \lambda \int_{\partial U \times \bR} \varphi^2 \; d\mu^{\gamma+dt^2},
\end{equation}
where $\gamma=g_{|_{\partial U}}$.
Now let $\varphi = \xi \chi$ where $\xi: \partial U \to \bR$ and $\chi: \bR \to \bR$ is a cutoff function such that $\chi(t)=1$ for $|t|\leq T$, $\chi(t)=0$ for $|t|\geq T+1$, and $|\partial_t \chi| \leq 2$. For this choice of $\varphi$ in \eqref{eqSYIntegrated12} we obtain 
\begin{equation*}
\begin{split}
\int_{\partial U}  \xi^2 \scal^{\gamma} \, d\mu^{\gamma}\int_{\bR} \chi^2 dt + 2\int_{\partial U}|d\xi|^2_{\gamma} \, d\mu^{\gamma} \int_{\bR} \chi^2 dt
 + 8\int_{ \partial U} \xi^2  \, d\mu^{\gamma} &  \\  \geq  \lambda \int_{\partial U}  \xi^2 \, d\mu^{\gamma} \int_{\bR} \chi^2 \,dt. &
\end{split}
\end{equation*}
Dividing by $\int_{\bR} \chi^2 dt$ and letting $T\to \infty$ we get 
\begin{equation*}
\begin{split}
\int_{ \partial U}  \xi^2 \scal^{\gamma} \, d\mu^{\gamma} +  2 \int_{\partial U}|d\xi|^2_{\gamma} \, d\mu^{\gamma} \geq \lambda \int_{\partial U} \xi^2 \, d\mu^{\gamma}.
\end{split}
\end{equation*}
Applying this with $\xi$ that vanishes on all components of $\partial U$ except for $\partial U_i$ shows that for every $i=1, \ldots, l$ the operator $-\Delta^{\gamma_i} + \frac{1}{8} \scal^{\gamma_i}$ on $\partial U_i$ has positive spectrum. In particular, if $\xi = 1$ on $\partial U_i$ and zero elsewhere, we conclude by Gauss-Bonnet theorem that $\partial U_i$ is topologically a sphere.
\end{proof}

\begin{proposition}\label{propExactCylindrical}
Let $f: U \to \mathbb{R}$ be as described in the beginning of this section. Assume that $U \neq M$, that the dominant energy condition $\mu \geq |J|_g$ holds on $U$ and that this inequality is strict near $\partial U$. For every sufficiently large number $T_0$ that is a regular value for both $f$ and $-f$  there exists a complete Riemannian metric $\gtil$ on $\Sigma \subset M \times \bR$ such that
\begin{itemize}
\item[(1)] There is a compact set $\Omega \subset \Sigma$ such that its complement $\Sigma \setminus \Omega$ has finitely many components $C_1, \ldots, C_l$ and $N$. The induced metric on $N$ is the asymptotically Euclidean metric $\gtil_{|_N} = \gbar_{|_N}$, and each $(C_i, \gtil)$ is isometric to a half-cylinder $(\partial U_i \times (T_0,\infty), \gamma_i \times dt^2)$, where $\partial U_i$, $i=1,\ldots, l$, are the connected components of $\partial U$. The metric $\gtil$ is uniformly equivalent to $\gbar$ on all of $\Sigma$. 
\item[(2)] For every $\varphi \in C_c^1 (\Sigma)$ we have 
\begin{equation}\label{eqSYIntegrated1}
\int_{\Sigma} \left(|d\varphi|^2_{\gtil}  + \tfrac{1}{8}\scal^{\gtil} \varphi^2 \right) \, d\mu^{\gtil} \geq \tfrac{1}{8}\int_N  |A-K|^2_{\gbar} \varphi^2 \, d\mu^{\gbar} + \tfrac{3}{4}\int_{\Sigma} |d\varphi|^2_{\gtil} \, d\mu^{\gtil}.
\end{equation}
\end{itemize}
\end{proposition}

\begin{proof}
Just as in \cite{PMT2} we may slightly perturb the metric $\gbar$ so that the asymptotically cylindrical ends $\Sigma \cap \{|t|>T_0\}$ for a sufficiently large $T_0>0$ become exactly cylindrical. Since $\mu-|J|_g>0$ near $\partial U$, in the view of \eqref{eqScalarJang1} we may ensure that the perturbed metric $\gtil$ satisfies
\begin{equation}\label{eqScalarJang2}
\tfrac{1}{2}\scal^{\gtil} - \tfrac{1}{2}|A-K|^2_{\bar{g}}  - |q|^2_{\gtil} + \divg^{\gtil} q \geq \tfrac{1}{2}(\mu - J(w)) \geq \tfrac{1}{2}(\mu - |J|_g).
\end{equation}
Integrating this against $\varphi^2$ with respect to the measure $d\mu^{\gtil}$ and arguing as in the proof of Proposition \ref{propConfStructure1} the claim follows in the view of the dominant energy condition.
\end{proof}

From now on we will refer to $N$ as an asymptotically Euclidean end of $(\Sigma, \gtil)$ and to $C_1,\ldots, C_n$ as its cylindrical ends. If $U=M$ then we take $\gtil = \gbar$.

\begin{remark}\label{remModification}
Note that in the asymptotically Euclidean setting of \cite{PMT2} and \cite{EichmairJang} the inequality \eqref{eqSYIntegrated1} is satisfied not only for $\varphi \in C^1_c(\Sigma)$ but for all $\varphi \in C^1(\Sigma)$ such that $(\mathrm{spt} \, \varphi) \cap C_i$, $i=1,\ldots, n$, is compact. In particular, it applies to $\varphi$ vanishing outside of a compact set in the asymptotically cylindrical ends and satisfying $\varphi \to 1$ in the asymptotically Euclidean end of $\Sigma$. This is not the case in the  asymptotically hyperbolic setting, as we merely have $\scal^{\gtil} = O (r^{-3})$ by \eqref {eqScalarJang}. This becomes important when analyzing the asymptotic behavior of certain conformal factors, see Proposition \ref{propAsymptAnalysis} below.
\end{remark}

We start with the metric $\gtil$ with exactly cylindrical ends on $\Sigma$, as described in Proposition \ref{propExactCylindrical}, and deform it into the metric satisfying the conditions of the positive mass theorem for asymptotically Euclidean manifolds that was proven in \cite{PMT1}. For this we essentially follow the same steps as in \cite{PMT2} and \cite{EichmairJang}, apart from some adjustments needed to deal with the fact that the asymptotics of the asymptotically Euclidean metric $\gbar$ are slightly worse than in the setting of \cite{PMT2} and \cite{EichmairJang}. Describing how the mass changes in this deformation process requires careful bookkeeping. The argument proceeds as follows:

\begin{itemize}

\item[(1)] In Proposition \ref{propStep1} we make a conformal change to zero scalar curvature in the cylindrical ends. More specifically, we construct a conformal factor $\Psi > 0$  that ``conformally closes'' the cylindrical ends $C_i$, $i=1,\ldots,l$, and yields an incomplete asymptotically Euclidean metric $\gtil_\Psi = \Psi^4 \gtil$ with $l$ conical singularities. We have $\gtil_\Psi = \gtil = \gbar$ in $N$, in particular, the mass of the metric is preserved.

\smallskip
 
\item[(2)] In Proposition \ref{propYamabe} we construct a conformal factor $u>0$ such that the metric $\gtil_{u\Psi} = u^4 \gtil_\Psi = (u\Psi)^4 \gtil$ has zero scalar curvature everywhere. This conformal transformation may change the mass, in which case the mass of $\gtil_{u\Psi}$ is at least a half of the mass of $\gtil$, see Proposition \ref{propAsymptAnalysis}.

\smallskip

\item[(3)] In Proposition \ref{propDeform} the metric $\gtil_{u\Psi}$ is deformed to a metric $\ghat$  which is asymptotically Schwarzschildean in the sense of Definition \ref{defAEManifolds} and has zero scalar curvature. The mass changes arbitrarily little. This step is not needed in the asymptotically Euclidean setting of  \cite{PMT2} or \cite{EichmairJang}. 

\smallskip

\item[(4)] Finally, in Proposition \ref{propOpeningEnds} we construct a conformal factor that we will later use for ``opening up'' the conformally compactified asymptotically Euclidean ends while changing the mass arbitrarily little. As we will see in Section \ref{secPositivity}, this deformation results in a complete metric with nonnegative scalar curvature to which the positive mass theorem of \cite{PMT1} can be applied. 
\end{itemize}

\begin{proposition}\label{propStep1}
There is a conformal factor $\Psi>0$ such that $\gtil_\Psi \definedas \Psi^4 \gtil$ has vanishing scalar curvature $\scal^{\gtil_\Psi}=0$ on each cylindrical end. Further, for each compactly supported $\varphi \in C^1 (\Sigma)$ we have
\begin{equation}\label{eqSYIntegrated2}
\begin{split} 
\int_\Sigma \left( |d\varphi|^2_{\gtil_\Psi} + \tfrac{1}{8} \scal^{\gtil_\Psi} \varphi^2 \right) \; d\mu^{\gtil_\Psi} \geq 
 \tfrac{3}{4} \int_\Sigma \Psi^{-2} |d(\varphi\Psi)|^2_{\gtil_\Psi}  \,d\mu^{\gtil_\Psi} + \tfrac{1}{8} \int_N |A-K|^2_{\gbar}\varphi^2 \, d\mu^{\gbar}.
\end{split}
\end{equation}
\end{proposition}                                                              
\begin{proof}
Let $(C_i, \gtil) = (\partial U_i \times (T_0, \infty), \gamma_i + dt^2)$ be one of the exact cylindrical ends of $(\Sigma, \gtil)$. (If $(C_i, \gtil) = (\partial U_i \times (-\infty,- T_0), \gamma_i + dt^2)$, replace $t$ by $-t$ in the argument below.) Let $0<\phi_i\in C^{2,\beta}(\partial U_i)$ be the first eigenfunction of the operator $-\Delta^{\gamma_i} + \frac{1}{8} \scal^{\gamma_i}$, so that
\[
-\Delta^{\gamma_i} \phi_i + \frac{1}{8} \scal^{\gamma_i} \phi_i = \lambda_i \phi_i
\]
for $\lambda_i > 0$. If we set $\Psi_i = e^{-\sqrt{\lambda_i}t}\phi_i$, then the scalar curvature of the metric $\Psi^4_i (\gamma_i + dt^2)$ vanishes on $\partial U_i \times \bR$. Let $s_i = \frac{1}{2\sqrt{\lambda_i}}e^{-2\sqrt{\lambda_i}t}$, then $(C_i,\Psi_i^4(\gamma_i + dt^2))$ is isometric to  $(\partial U_i \times (0,\frac{1}{2\sqrt{\lambda_i}}e^{-2\sqrt{\lambda_i}T_0}), \phi_i^4 (4\lambda_i s_i^2 \gamma_i + ds_i^2))$, in particular, it is  uniformly equivalent to the cone $(\partial U_i\times (0,\frac{1}{2\sqrt{\lambda_i}}e^{-2\sqrt{\lambda_i}T_0}), s_i^2 \gamma_i + ds_i^2)$. Fix a function $\Psi>0$ such that $\Psi = \Psi_i$ on $C_i$ and $\Psi=1$ on $N$, and let $\gtil_\Psi \definedas \Psi^4 \gtil$. The scalar curvature $\scal^{\gtil_\Psi}$ vanishes on each cylindrical end of $\Sigma$. 

In order to obtain \eqref{eqSYIntegrated2}, we first replace $\varphi$ by $\varphi \Psi$ in \eqref{eqSYIntegrated1}, which gives 
\[
\int_{\Sigma} \left( \scal^{\gtil} (\varphi \Psi)^2 + 2|d(\varphi\Psi)|^2_{\gtil} \right) \,d\mu^{\gtil} \geq \int_N |A-K|^2_{\gbar} \varphi^2\, d\mu^{\gbar}.
\]
Further, we note that 
\[
\begin{split}
\int_\Sigma \scal^{\gtil} (\varphi\Psi)^2 \, d\mu^{\gtil} & = \int_\Sigma (8 \Psi \Delta^{\gtil} \Psi + \scal^{\gtil_\Psi} \Psi^6)  \varphi^2 \, d\mu^{\gtil} \\ 
\end{split}
\]
where 
\[
\int_\Sigma \scal^{\gtil_\Psi} \Psi^6 \varphi^2 \, d\mu^{\gtil} = \int_\Sigma \scal^{\gtil_\Psi} \varphi^2 \, d\mu^{\gtil_\Psi}
\] 
and
\[
\begin{split}
8 \int_\Sigma \varphi^2 \Psi \Delta^{\gtil} \Psi \, d\mu^{\gtil} 
& = - 8 \int_\Sigma \left( \varphi^2 |d\Psi|^2_{\gtil}+ 2 \varphi \Psi \gtil(\nabla^{\gtil} \Psi, \nabla^{\gtil} \varphi)  \right) \, d\mu^{\gtil}\\     
& = - 8 \int_\Sigma \left( |d(\varphi\Psi)|^2_{\gtil} - \Psi^2 |d\varphi|^2_{\gtil} \right) \, d\mu^{\gtil} \\
& = - 8 \int_\Sigma \left( \Psi^{-2}|d(\varphi\Psi)|^2_{\gtil_\Psi} - |d\varphi|^2_{\gtil_\Psi}\right) \, d\mu^{\gtil_\Psi} .
\end{split}
\]
Similarly, we have 
\[
2 \int_{\Sigma} |d(\varphi\Psi)|^2_{\gtil}  \,d\mu^{\gtil} = 2 \int_{\Sigma} \Psi^{-2} |d(\varphi\Psi)|^2_{\gtil_\Psi}  \,d\mu^{\gtil_\Psi}.
\]
Summing up, we obtain \eqref{eqSYIntegrated2}.
\end{proof}
                                                                                  
Following \cite{EichmairJang} we may now introduce a new distance function $s=s(x)$ such that $0<s\in C^{3,\beta}(\Sigma)$, $s=r$ on $N$, and $s=s_i$ on $C_i$. When $U=M$ we just set $s=r$ everywhere on $\Sigma$.  One may now add a point at infinity to each of the asymptotically cylindrical ends of $(\Sigma,\gtil)$ and extend the new distance function $s$ to these virtual singular points by zero. In this way each cylindrical end of $(\Sigma,\gtil)$  corresponds to a conical singularity of $(\Sigma,\gtil_\Psi)$.

\begin{remark}\label{remHarmonicCapacity}
 These conical singularities have \emph{vanishing harmonic capacity}, as explained in \cite{EichmairJang}: Take a smooth cut off function $\chi_\varepsilon$ such that $0\leq \chi_\varepsilon \leq 1$, $\chi_\varepsilon = 0$ for $0\leq s \leq \varepsilon$, $\chi_\varepsilon = 1 $ for $ s \geq 2\varepsilon$, and $|\nabla^{\gtil_\Psi} \chi_\varepsilon| \leq C \varepsilon^{-1}$ where $C$ does not depend on $\varepsilon$.  Then
\begin{equation*}
\int_\Sigma |d\chi_\epsilon|^2_{\gtil_\Psi}\, d\mu^{\gtil_\Psi} = O(\varepsilon).
\end{equation*} 
\end{remark}
%for the functions $\chi_\epsilon \in C^{3,\beta}_{loc}(\Sigma)$ defined by 
%\[
 %\chi_{\epsilon}(x) = 
 %\begin{cases}
 %0,  &  0 <  s(x) < \epsilon \\
 %-1 + \epsilon^{-1} s(x),  &  \epsilon \leq s(x) \leq 2 \epsilon \\
%	1,  &  2\epsilon < s(x).
 %\end{cases}
%\]
%are compactly supported in the cylindrical ends, are equal to 1 outside of the cylindrical ends and converge to $1$ locally uniformly on all of $\Sigma$ as $\epsilon \to 0$. 

\begin{proposition}\label{propYamabe}
There exists $u \in C^{2,\beta}_{loc}(\Sigma)$ such that 
\begin{equation}\label{eqConfFactor}
 -\Delta^{\gtil_\Psi} u + \tfrac{1}{8} \scal^{\gtil_\Psi} u = 0 \hspace{12pt}\text{on}\hspace{12pt} \Sigma,
\end{equation}
$u\to 1$ as $r\to\infty$, and $c^{-1} \leq u \leq c$ for some $c\geq 1$.
As a consequence, the metric $\gtil_{u\Psi}\definedas u^4 \gtil_\Psi$ has zero scalar curvature.
\end{proposition}

\begin{proof}
Here we essentially repeat a part of the proof of \cite[Proposition 12]{EichmairJang} which in turn is based on \cite[Lemma 4]{PMT2}, for the reader's convenience. Let $\sigma_0$ be as small as to ensure that $\scal^{\gtil_\Psi} = 0$ for $0<s<2\sigma_0$. For $\sigma < \sigma_0$ consider a sequence of Dirichlet problems
\begin{eqnarray*}
-\Delta^{\gtil_\Psi} v_\sigma + \tfrac{1}{8} \scal^{\gtil_\Psi} v_\sigma = - \tfrac{1}{8} \scal^{\gtil_\Psi} &\text{ in }& \{\sigma < s < \sigma^{-1}\},\\
v_\sigma = 0 &\text{ on }& \{s=\sigma\} \cup \{s=\sigma^{-1}\}.
\end{eqnarray*} 
The solution $v_\sigma$ exists and is unique as \eqref{eqSYIntegrated2} implies that the respective homogeneous problem only has a zero solution. Extending each $v_\sigma$ by zero to be a compactly supported Lipschitz function on $\Sigma$ we obtain 
\[
\begin{split}
\left(\int_{\{s\geq \sigma_0\}} |v_\sigma|^6 \, d\mu^{\gtil_\Psi} \right)^{1/3} 
% & = \left(\int_{s\geq \sigma_0} |v_\sigma \Psi|^6 \Psi^{-6}\, d\mu^{\gtil_\Psi} \right)^{\frac{1}{3}} \\
& \leq C \left(\int_{\{s\geq \sigma_0\}} |v_\sigma \Psi|^6 \, d\mu^{\gtil_\Psi} \right)^{1/3} \\
& \leq C \int_{\{s\geq \sigma_0\}} |d(v_\sigma \Psi)|^2_{\gtil_\Psi} \, d\mu^{\gtil_\Psi} \\
& \leq C \int_{\{s\geq \sigma_0\}} \Psi^{-2} |d(v_\sigma \Psi)|^2_{\gtil_\Psi} \, d\mu^{\gtil_\Psi} \\
& \leq C \int_{\{\sigma \leq s \leq \sigma^{-1}\}} \left( |dv_\sigma|^2_{\gtil_\Psi} + \tfrac{1}{8} \scal^{\gtil_\Psi} v_\sigma^2 \right) \; d\mu^{\gtil_\Psi}\\
& = C \int_{\{\sigma \leq s \leq \sigma^{-1}\}} v_\sigma \left( - \Delta^{\gtil_\Psi} v_\sigma + \tfrac{1}{8} \scal^{\gtil_\Psi} v_\sigma \right) \; d\mu^{\gtil_\Psi}\\
& \leq C \int_{\{s \geq \sigma_0\}} |\scal^{\gtil_\Psi}||v_\sigma| \; d\mu^{\gtil_\Psi}\\
& \leq C \left(\int_{\{s \geq \sigma_0\}} |\scal^{\gtil_\Psi}|^{6/5} \; d\mu^{\gtil_\Psi}\right)^{5/6} \left(\int_{\{s \geq \sigma_0\}} |v_\sigma|^{6} \; d\mu^{\gtil_\Psi}\right)^{1/6},
\end{split} 
\]
where the constant $C>0$ may vary from line to line, but is independent of $\sigma$. In the first line we  relied on the fact that $\Psi$ is bounded away from zero on $\{s \geq \sigma_0\}$. In the second line we used the Sobolev inequality in the form of \cite[Lemma 18]{EichmairJang}. The third line is a consequence of the fact that $\Psi$ is bounded from above on $\{s \geq \sigma_0\}$. In the fourth line we  used the fact that $v_\sigma$ vanishes outside of $\{\sigma \leq s \leq \sigma^{-1}\}$ and applied  \eqref{eqSYIntegrated2} with $\varphi = v_\sigma$.  In the fifth line we performed integration by parts. In the sixth line we made use of the equation that $v_\sigma$ satisfies together with the fact that $\scal^{\gtil_{\Psi}} = 0$ for $0 \leq s \leq 2\sigma_0$. We conclude by applying the H\"older inequality in the last line.

Since $\scal^{\gtil_{\Psi}} = 0$ for $0 \leq s \leq 2\sigma_0$ and  $|\scal^{\gtil_{\Psi}}|^{6/5} = O(r^{-18/5})$ in $N$, it follows that $v_\sigma$ are uniformly bounded in $L^6$ on $\{s\geq \sigma_0\}$.  Applying elliptic regularity in the balls of fixed radius followed by the Sobolev embedding it follows that $|v_\sigma|<C$ on $\{s\geq 2\sigma_0\}$ for a constant $C>0$ independent of $\sigma$. Further, note that $v_\sigma$ are harmonic on $\{\sigma \leq s \leq 2 \sigma_0\}$ and vanish on $\{s=\sigma\}$. Since harmonic functions attain their maximum and minimum on the boundary, it follows that $|v_\sigma|<C$ on $\{\sigma \leq s \leq 2 \sigma_0\}$ as well. All in all, we obtain the uniform bound $|v_\sigma| < C$ on $\{\sigma \leq s \leq \sigma^{-1}\}$. A standard diagonal subsequence extraction argument gives a subsequence of $u_\sigma \definedas v_\sigma + 1$ that converges to a solution $u \in C^{2,\beta}_{loc}$ of \eqref{eqConfFactor} as $\sigma\searrow 0$. Note that the above discussion shows that $|u_\sigma| <c$ for some $c>1$.

In order to show that $u$ is bounded away from zero, we will first show that $u_\sigma > 0$ on $\{\sigma < s < \sigma^{-1}\}$. From the definition of $v_\sigma$ it is clear that this is true in a neighborhood of the boundary of this set. Let $\varepsilon>0$ be a sufficiently small regular value of $-u_\sigma$, then $\min\{u_\sigma+\varepsilon,0\}$ is a Lipschitz continuous function with support in $\{\sigma < s < \sigma^{-1}\}$. Using it as a test function in  \eqref{eqSYIntegrated2} we obtain
\begin{equation*}
\begin{split}
&\tfrac{3}{4} \int_{\{u_\sigma < -\varepsilon\}} \Psi^{-2} |d((u_\sigma + \varepsilon)\Psi)|^2_{\gtil_\Psi}  \,d\mu^{\gtil_\Psi} \\
& \qquad \leq \int_{\{u_\sigma < -\varepsilon\}} \left( |d(u_\sigma + \varepsilon)|^2_{\gtil_\Psi} + \tfrac{1}{8} \scal^{\gtil_\Psi} (u_\sigma + \varepsilon)^2 \right) \; d\mu^{\gtil_\Psi} \\
& \qquad  \leq \int_{\{u_\sigma < -\varepsilon\}} (u_\sigma + \varepsilon)\left(- \Delta^{\gtil_\Psi} (u_\sigma + \varepsilon) + \tfrac{1}{8} \scal^{\gtil_\Psi} (u_\sigma + \varepsilon) \right) \; d\mu^{\gtil_\Psi} \\
& \qquad  \leq \varepsilon \int_{\{u_\sigma < -\varepsilon\}} \tfrac{1}{8} \scal^{\gtil_\Psi} (u_\sigma + \varepsilon) \; d\mu^{\gtil_\Psi},
\end{split}
\end{equation*}
where we used the equation that $u_\sigma$ satisfies in the last line. Letting $\varepsilon \searrow 0$, we see that $\Psi u_\sigma = \const$ on $\{u_\sigma < 0\}$, hence $\{u_\sigma < 0\} = \emptyset$.  As $u_\sigma=1$ on $\{s=\sigma\} \cup \{s=\sigma^{-1}\}$ we have $u_\sigma > 0$ by Harnack theory, thus $u \geq 0$ everywhere on $\Sigma$. Combining the fact that the subsequential limit $v$ of $v_\sigma$ satisfies $\int_{\{s\geq \sigma_0\}} |v|^6 \, d\mu^{\gtil_\Psi} < C$ with standard elliptic theory for the equation that $v$ satisfies %(see also the proof of Proposition \ref{propAsymptAnalysis} below) 
we conclude that $u\to 1$ as $r\to \infty$. Again, by Harnack theory it follows that $u>0$ on $\Sigma$. Since $u_\sigma$ are harmonic on $\{\sigma < s < 2 \sigma_0\}$ and uniformly approach $u>0$ on a neighborhood of $\{s=2\sigma_0\}$, it follows that they are uniformly bounded away from zero on $\{\sigma < s < 2\sigma_0\}$ by some constant independent of $\sigma$. Combining this with the fact that $u>0$ is bounded away from zero for large $r$, we conclude that $u > c^{-1}$ for some $c>1$ everywhere in $\Sigma$, which completes the proof.
\end{proof}       

We recall that $(N, \gtil_\Psi) $ is (a part of) the graphical component of the geometric solution of the Jang equation and that the graphing function $f: U \to \bR$  satisfies 
\begin{equation*}
f(r,\theta,\varphi)=\sqrt{r^2+1} + \alpha \ln r + \psi (\theta,\varphi)+O_3(r^{-1+\varepsilon})
\end{equation*}
as $r\to \infty$, where $\alpha = 2E$  is twice the energy of the initial data set $(M,g,K)$ and the function $\psi:\mathbb{S}^2\to\bR$ defined by the equation \eqref{eqPsi} is such that $\int_{\mathbb{S}^2} \psi \, d\mu^\sigma = 0$. We shall now see how these quantities enter the asymptotics of the conformal factor $u$ constructed in  Proposition \ref{propYamabe}. 

\begin{proposition}\label{propAsymptAnalysis}
Let $u$ be as in Proposition \ref{propYamabe}. Then 
\begin{equation}\label{eqExpansionU}
u = 1 + (A + \tfrac{1}{4} \psi) r^{-1} + O^{2,\beta}(r^{-2+\varepsilon}) \quad \text{ in } \quad N
\end{equation}
where  the constant $A$ satisfies $A \leq -\tfrac{\alpha}{4}$. Consequently, $\gtil_{u\Psi} \definedas u^4 \gtil_\Psi = (u\Psi)^4 \gtil$ is a (possibly incomplete) asymptotically Euclidean metric, and its mass satisfies $\mathcal{M} (\gtil_{u\Psi}) \leq \alpha/2 = E$.
\end{proposition}

\begin{proof}
By Corollary \ref{corAEh} we have $\gbar=g_{\Sigma_-} + O^{2,\beta}(r^{-2+\varepsilon})$.  Recalling \eqref{eqNewBarriers} it is straightforward to verify that $\scal^{g_{\Sigma_-} }= \tfrac{2 \Delta^{S^2} \psi}{r^3} + O^{\beta}(r^{-4+\varepsilon})$, either by a direct computation or using the full version of the Schoen and Yau identity \cite[(2.25)]{PMT2}. It follows that $\scal^{\gbar} = \tfrac{2 \Delta^{S^2} \psi}{r^3} + O^{\beta}(r^{-4+\varepsilon})$.  If we now set $u = 1 + \tfrac{1}{4}\psi\, r^{-1} + u_0$, then as a consequence of \eqref{eqConfFactor}, we see that $u_0$ in $N$ satisfies the equation 
\[
\Delta^{\gbar} u_0 = \tfrac{1}{4} r^{-3} \Delta^{\bS^2} \psi \,   u_0 + O^\beta(r^{-4+\varepsilon}).
\] 
A standard argument using the fact that is $\gbar=\delta + O^{2,\beta}(r^{-1})$ and \cite[Theorem 2]{Meyers} yields $u_0 = A \, r^{-1} + O^{2,\beta}(r^{-2+\varepsilon})$, where $A\in \bR$, which implies \eqref{eqExpansionU}.

Our goal now is to estimate the constant $A$ from above. In contrast to \cite{PMT2} and \cite{EichmairJang} we cannot use \eqref{eqSYIntegrated2} for this purpose, see Remark \ref{remModification}. Instead we will work directly with \eqref{eqScalarJang2}, and rely on the fact that $u$, as a consequence of \eqref{eqConfFactor}, satisfies
\begin{equation}\label{eqYamabeRed}
\scal^{\gtil} = 8 (u\Psi)^{-1} \Delta^{\gtil} (u\Psi).
\end{equation}
In this case we have
\[
4 u\Psi \Delta^{\gtil} (u\Psi) + (u\Psi)^2 \divg^{\gtil} q - \tfrac{1}{2} (u\Psi)^2 |A-K|^2_{\gbar} - (u\Psi)^2 |q|^2_{\gtil}  \geq \tfrac{1}{2}  (u\Psi)^2 (\mu - |J|_g),
\]
which yields  
\begin{equation}
\begin{split}
\divg^{\gtil} (4 u\Psi d (u\Psi) + (u\Psi)^2 q) - 4 |d(u\Psi)|^2_{\gtil} - 2 u\Psi q(\nabla^{\gtil}(u\Psi)) - (u\Psi)^2 |q|^2_{\gtil}  & \\ \geq \tfrac{1}{2}(u\Psi)^2 |A-K|^2_{\gbar} +\tfrac{1}{2}  (u\Psi)^2 (\mu - |J|_g) &.
\end{split}
\end{equation}
Applying Cauchy-Schwartz inequality in the left hand side we obtain
\begin{equation}\label{eqBeforeIntegration}
\begin{split}
\divg^{\gtil} (4 u\Psi d (u\Psi) + (u\Psi)^2 q) - 3 |d(u\Psi)|^2_{\gtil}  \geq \tfrac{1}{2}(u\Psi)^2 |A-K|^2_{\gbar} + \tfrac{1}{2}(u\Psi)^2(\mu - |J|_g)  .
\end{split}
\end{equation}
We intend to integrate this inequality with respect to the measure $d\mu^{\gtil}$ on $\Sigma$, so we first need to clarify why such an integration makes sense. 

Note that as a consequence of \eqref{eqScalarJang1} and \eqref{eqYamabeRed} in the asymptotically Euclidean end, where $\Psi \equiv 1$, we have 
\begin{equation}\label{eqGoodTerm}
\begin{split}
\divg^{\gtil} (4 u\Psi d (u\Psi) + (u\Psi)^2 q) & =  \divg^{\gbar} (4 u \, d u + u^2 q) \\&  =  4 |du|^2_{\gbar} + 2u q(\nabla^{\gbar}u) + u^2 |q|^2_{\gbar} \\  & \qquad + \tfrac{1}{2} u^2 |A-K|^2_{\gbar}  + u^2 (\mu - J (w)) \\ & =  O(r^{-4+\varepsilon}), 
\end{split}
\end{equation}
and all other terms in \eqref{eqBeforeIntegration} are $O(r^{-4+\varepsilon})$ in the asymptotically Euclidean end as well.\footnote{It is actually the main advantage of  \eqref{eqBeforeIntegration} that the terms with slow fall off arising from $\scal^{\gtil}$ and $\divg^{\gtil} q$ in  \eqref{eqScalarJang2} are combined together in one quickly decaying term \eqref{eqGoodTerm}.} In the case when $(\Sigma, \gtil)$ has cylindrical ends all terms in the right hand side of \eqref{eqBeforeIntegration} fall off exponentially as $t\to\infty$ as they all include $\Psi$ and since all other quantities appearing in these terms are bounded. As for the left hand side of \eqref{eqBeforeIntegration}, we note that  
\begin{equation*}
\begin{split}
&|\divg^{\gtil} (4 u\Psi d (u\Psi) + (u\Psi)^2 q) |\\ & =4  |d (u\Psi) |^2_{\gtil} + \tfrac{1}{2} (u\Psi)^2 |\scal^{\gtil}| + (u\Psi)^2|\divg^{\gtil} q| + 2 |u\Psi q(\nabla^{\gtil}(u\Psi))| \\&  \leq  5 |d (u\Psi) |^2_{\gtil}+ (u\Psi)^2 |q|^2_{\gtil}+ \tfrac{1}{2} (u\Psi)^2 |\scal^{\gtil}| + (u\Psi)^2|\divg^{\gtil} q|  \\&  \leq  10\Psi^2 |d u|^2_{\gtil}+ 10 u^2 |d \Psi |^2_{\gtil}+ (u\Psi)^2 |q|^2_{\gtil}+ \tfrac{1}{2} (u\Psi)^2 |\scal^{\gtil}| + (u\Psi)^2|\divg^{\gtil} q|,
\end{split}
\end{equation*}
where we used the Cauchy-Schwartz inequality in the third and fourth line. Clearly, all terms  in the last line, except for possibly the first one, are integrable on $\{s\leq \sigma_0\}$.  Thus, in order to be able to integrate \eqref{eqBeforeIntegration}  we only need to show that
\begin{equation}\label{eqGradBoundCylEnd}
\int_{\{s\leq \sigma_0\}}  \Psi^2 |d u|^2_{\gtil} \, d\mu^{\gtil} = \int_{\{s\leq \sigma_0\}}  |d u |^2_{\gtil_\Psi} \, d\mu^{\gtil_\Psi} < \infty.
\end{equation}
This has actually been explained in \cite[Proof of Proposition 12]{EichmairJang}. For the reader's convenience, we briefly recall this argument. 

We have $-\Delta^{\gtil_\Psi} u + \frac{1}{8} \scal^{\gtil_\Psi} u = 0$, where $\scal^{\gtil_\Psi} \equiv 0$ on $\{s\leq \sigma_0\}$. Consequently, for any function $\xi\in C^1$ that has compact support in $\{s\leq \sigma\}$ for $0<\sigma \leq \sigma_0$ we have
\begin{equation}\label{eqTest}
\int_{\{s\leq \sigma\}} du (\nabla^{\gtil_\Psi }\xi) \, d\mu^{\gtil_\Psi} = \int_{ \{s=\sigma\}} \xi \nu_{\gtil_\Psi} (u) \, d\mu^{\gtil_\Psi},
\end{equation}
where $\nu_{\gtil_\Psi}$ is the outward pointing unit normal with respect to the metric $\gtil_\Psi $. Applying this identity with $\xi=u\chi_\varepsilon^2$ for $\chi_\varepsilon$ as in the Remark \ref{remHarmonicCapacity}  and letting $\varepsilon \searrow 0$ the desired bound \eqref{eqGradBoundCylEnd} follows in the view of the $L^\infty$-bound on $u$. With this bound at hand, using test functions $\xi=u\chi_\varepsilon$ in \eqref{eqTest}, we also obtain
\begin{equation}\label{eqDivTh}
 \int_{\{s<\sigma \}} |du|^2_{\gtil_\Psi} \, d\mu^{\gtil_\Psi} = \int_{\{s=\sigma \}} u \,du (\nu_{\gtil_\Psi})  \, d\mu^{\gtil_\Psi} .
\end{equation}

As a consequence, integrating \eqref{eqBeforeIntegration} over $(\Sigma, \gtil)$ and performing integration by parts, we obtain
\begin{equation}\label{eqIntByParts4}
\begin{split}
0 & \leq  \tfrac{1}{2}\int_N (u\Psi)^2 (\mu - |J|_g) \, d\mu^{\gbar} + \tfrac{1}{2}  \int_N (u\Psi)^2 |A-K|^2_{\gbar}  \, d\mu^{\gbar} + 3 \int_{\Sigma} |d(u\Psi)|^2_{\gtil}\, d\mu^{\gtil} \\
& \leq \int_\Sigma \divg^{\gtil} (4 u\Psi d (u\Psi) + (u\Psi)^2 q) \, d\mu^{\gtil}\\
 & = \lim_{\sigma \to 0} \int_{\{s=\sigma^{-1}\}} (4u \,du + u^2 q)(\nu_{\gtil})  \, d\mu^{\gtil} + \lim_{\sigma \to 0} \int_{\{s=\sigma \}} (4u\Psi \,d(u\Psi) + (u\Psi)^2 q)(\nu_{\gtil})  \, d\mu^{\gtil}
\end{split}
\end{equation}
where $\nu_{\gtil}$ is the outward pointing unit with respect to $\gtil$ normal to the domain $\{\sigma < s < \sigma^{-1}\}$. Using the exponential fall off of $\Psi$, \eqref{eqDivTh}, and the finiteness of $\int_\Sigma |du|^2_{\gtil_\Psi}\, d\mu^{\gtil_\Psi}$ it is straightforward to check that
\begin{equation*}
\begin{split}
\lim_{\sigma \to 0} \int_{\{s=\sigma \}} (4u\Psi \,d(u\Psi) + (u\Psi)^2 q)(\nu_{\gtil})  \, d\mu^{\gtil} & = \lim_{\sigma \to 0} \int_{\{s=\sigma \}} 4\Psi^2 u\, du (\nu_{\gtil})  \, d\mu^{\gtil}  \\ & = \lim_{\sigma \to 0} \int_{\{s=\sigma \}} 4u \,du (\nu_{\gtil_\Psi})  \, d\mu^{\gtil_\Psi}  \\ & = \lim_{\sigma \to 0} \int_{\{s<\sigma \}} 4|du|^2_{\gtil_\Psi} \, d\mu^{\gtil_\Psi} =0.
\end{split}
\end{equation*}
 Further, using asymptotic expansions in the asymptotically Euclidean end and recalling that $\int_{\mathbb{S}^2} \psi \, d\mu^\sigma=0$ we obtain
 \begin{equation*}
\begin{split}
\lim_{\sigma \to 0} \int_{\{s=\sigma^{-1}\}} (4u \,du + u^2 q)(\nu_{\gtil})  \, d\mu^{\gtil} & =-  \int_{\mathbb{S}^2}(4A + \psi + \alpha) \, d\mu^\sigma \\ & = - (4A + \alpha ) \, \mathrm{vol}(\mathbb{S}^2).
\end{split}
\end{equation*}
The desired estimate $A \leq -\alpha /4$ follows by \eqref{eqIntByParts4}. 

Finally, we compute the mass of the asymptotically Euclidean metric $\gtil_{u\Psi}$: 
\[
\begin{split}
\mathcal{M} (\gtil_{u\Psi}) & = \mathcal{M} (u^4\gbar) \\
                            & =\frac{1}{16\pi}\lim_{R \to \infty} \int_{\{r=R\}}\left(\operatorname{div}^\delta (u^4\gbar) 
														   - d \tr^{\delta} (u^4 \gbar)\right) (\partial_r) \,d\mu^\delta \\
														& = \mathcal{M}(\gbar) + \frac{1}{16\pi}\lim_{R \to \infty} \int_{\{r=R\}} 4u^3\left(\gbar (\nabla^\delta u,\partial_r) 
														   - (\tr^\delta \gbar)\, \partial_r u \right) \, d\mu^\delta \\
														& = \mathcal{M}(\gbar) + \frac{1}{16\pi} \lim_{R \to \infty} \int_{\{r=R\}} (-8 \partial_r u + o(r^{-2})) \, d\mu^\delta \\
														& = \alpha + 2 A \\
														& \leq \alpha/2 = E,
\end{split}
\]
where we used Corollary \ref{corAEh} in the last two lines.
\end{proof}

While the metric $\gtil_{u\Psi}$ is asymptotically Euclidean with zero scalar curvature, it may fail to satisfy the assumptions of the Riemannian positive mass theorem in  \cite{PMT1}, since it might have conical singularities and since it does not approach the Euclidean metric sufficiently fast. In the view of these potential issues, we first adapt a well-known construction from \cite{PMT1.5} to ``improve'' the asymptotics of the metric (Proposition \ref{propDeform}) and then we ``open up''   the previously conformally closed cylindrical ends (Proposition \ref{propOpeningEnds}). This results in a complete metric with nonnegative scalar curvature to which the Riemannian positive mass theorem of \cite{PMT1} can be applied.

\begin{proposition}\label{propDeform}
For any sufficiently large $R>0$ there exists a metric $\ghat=\ghat(R)$ on $\Sigma$ such that 
\begin{itemize}
\item[1)] For $s\geq 2R$ we have $\ghat = v^4 g_{\mathrm{Schw}}$  where $g_{\mathrm{Schw}}=\left(1+\tfrac{m}{2r}\right)^4 \delta$ is the Schwarzschild metric  of the mass $m=\mathcal{M}(\gtil_{u\Psi})$. 
For $s\leq R$ we have  $\ghat = v^4 \gtil_{u\Psi}$.
\item[2)] The scalar curvature of the metric $\ghat = \ghat(R)$ is zero.
\item[3)] The conformal factor $v=v(R)\in C^{2,\alpha}_{loc}$ satisfies $c^{-1}\leq v \leq c$ in $\Sigma$ for a constant $c>1$ that is independent of $R$ and $v= 1 +ar^{-1}+ O^{2,\alpha}(r^{-2})$ in $N$ for $a=a(R) \in \bR$. 
As a consequence, the metric $\ghat=\ghat(R)$ is asymptotically Schwarzschildean in the sense of Definition \ref{defAEManifolds} with the mass $\mathcal{M}(\ghat) = \mathcal{M}(\gtil_{u\Psi})  + 2 a$. We also have
\[
\lim_{R\to\infty} a(R) = 0.
\]
\end{itemize}
\end{proposition}

\begin{proof}
Following \cite{PMT1.5},  we begin by splitting $\gtil_{u\Psi}$ into the ``Schwarzschildean'' part and the ``massless'' part. That is, we write $\gtil_{u\Psi} = h + \left(1+\tfrac{m}{2r}\right)^4 \delta$, where $m=\mathcal{M}(\gtil_{u\Psi})$, so that the contribution of the symmetric 2-tensor $h=O(r^{-1})$ to the mass of $\gtil_{u\Psi}$ is zero. Let $\chi$ be a smooth cutoff function such that $0 \leq \chi \leq 1$, $\chi=0$ for $s\leq R$, $\chi = 1$ for $s \geq 2R$, $|\nabla \chi|\leq c_1 R^{-1}$, and $|\nabla \nabla \chi|\leq c_2 R^{-2}$ for some constants $c_1$ and $c_2$ independent of $R$.  We define a new metric 
\[
\check{g}  = \gtil_{u\Psi} - \chi h = (1-\chi) h  + \left(1+\tfrac{m}{2r}\right)^4 \delta.
\] 
Note that $\scal^{\check{g}} = 0 $ for $s \in (0,R] \cup [2R, \infty)$, since both $\gtil_{u\Psi}$ and $g_{\mathrm{Schw}}=\left(1+\tfrac{m}{2r}\right)^4 \delta$  are scalar flat. Note also that $\scal^{\check{g}} = O(R^{-3})$. 

Next step is to construct a conformal factor $v$ such that the metric $\ghat = v^4\gcheck$ has zero scalar curvature everywhere. We fix $\sigma_0>0$ and note that $\scal^{\check{g}}$ vanishes on $\{s \leq 2\sigma_0\}$ when $R$ is sufficiently large.  For each $\sigma \in (0,\sigma_0)$ we consider the mixed Dirichlet-Neumann problem 
\begin{eqnarray}
-\Delta^{\gcheck} \varphi_\sigma + \tfrac{1}{8} \scal^{\gcheck} \varphi_\sigma = 0 &\text{ in }& \{\sigma < s < \sigma^{-1}\} \label{eqY1hom} \\
\varphi_\sigma = 0 &\text{ on }& \{s=\sigma^{-1}\} \label{eqY2hom}
\\
\nu_{\gcheck}( \varphi_\sigma) = 0 &\text{ on }& \{s=\sigma\} \label{eqY3hom}
\end{eqnarray} 
where $\nu_{\gcheck}$ denotes the outward $\gcheck$-unit normal to the domain  $\{\sigma < s < \sigma^{-1}\}$. Recalling that $\scal^{\check{g}} = O(R^{-3})$, and using the Sobolev inequality in the form of \cite[Lemma 18]{EichmairJang} on $\{ s \geq \sigma_0\}$, we conclude that the solutions satisfy
\begin{equation*}
\begin{split}
8\int_{\{\sigma \leq s \leq \sigma^{-1}\}} |d\varphi_\sigma|_{\gcheck}^2 \, d\mu^{\gcheck} 
& =  - \int_{\{\sigma_0 \leq s \leq \sigma^{-1}\}} \scal^{\gcheck} \varphi_\sigma^2  \, d\mu^{\gcheck}\\
& \leq \left( \int_{\{\sigma_0 \leq s \leq \sigma^{-1}\}} |\scal^{\gcheck}|^{3/2} \, d\mu^{\gcheck}\right)^{2/3} \left( \int_{\{\sigma_0 \leq s \leq \sigma^{-1}\}} \varphi_\sigma^6 \, d\mu^{\gcheck}\right)^{1/3} \\
& \leq C R^{-1} \left( \int_{\{\sigma_0 \leq s \leq \sigma^{-1}\}} \varphi_\sigma^6 \, d\mu^{\gcheck}\right)^{1/3} \\
& \leq C R^{-1} \int_{\{\sigma_0 \leq s \leq \sigma^{-1}\}} |d\varphi_\sigma|^2 \, d\mu^{\gcheck}\\
& \leq C R^{-1} \int_{\{\sigma \leq s \leq \sigma^{-1}\}} |d\varphi_\sigma|^2 \, d\mu^{\gcheck},
\end{split}
\end{equation*}
where the constant $C>0$ might vary from to line but remains independent of $\sigma$ and $R$. Choosing $R>C$ in this estimate we see that $\varphi_\sigma \equiv 0$ is the only solution of \eqref{eqY1hom}-\eqref{eqY3hom}. 

Consequently, for each $\sigma \in (0,\sigma_0)$ there exists a unique solution $\varphi_\sigma$ to the mixed Dirichlet-Neumann problem
\begin{eqnarray*}
-\Delta^{\gcheck} \varphi_\sigma + \tfrac{1}{8} \scal^{\gcheck} \varphi_\sigma = - \tfrac{1}{8} \scal^{\gcheck} &\text{ in }& \{\sigma < s < \sigma^{-1}\}
\\
\varphi_\sigma = 0 &\text{ on }& \{s=\sigma^{-1}\} 
\\
\nu_{\gcheck}( \varphi_\sigma) = 0 &\text{ on }& \{s=\sigma\}. 
\end{eqnarray*} 
We extend each  $\varphi_\sigma$ by zero to a Lipschitz continuous function on $\{s\geq \sigma\}$. Using the Sobolev inequality  \cite[Lemma 18]{EichmairJang}, and the fact that $\scal^{\gcheck}$ vanishes on $\{s \leq 2 \sigma_0\}$ we obtain that
\begin{equation*}
\begin{split}
C^{-1}\left( \int_{\{s \geq \sigma_0\}} \varphi_\sigma^6 \, d\mu^{\gcheck}\right)^{1/3} 
& \leq  8 \int_{\{s \geq \sigma_0\}} |d\varphi_\sigma|_{\gcheck}^2 \, d\mu^{\gcheck} \\
& =  \int_{\{\sigma \leq s \leq \sigma^{-1}\}} (-\scal^{\gcheck} \phi_\sigma^2-\scal^{\gcheck} \phi_\sigma)\, d\mu^{\gcheck} \\
& \leq \left( \int_{\{s\geq \sigma_0\}} |\scal^{\gcheck}|^{3/2} \, d\mu^{\gcheck}\right)^{2/3} \left( \int_{\{s \geq \sigma_0\}} \varphi_\sigma^6 \, d\mu^{\gcheck}\right)^{1/3} \\ & \qquad + \left( \int_{\{s\geq \sigma_0\}} |\scal^{\gcheck}|^{6/5} \, d\mu^{\gcheck}\right)^{5/6} \left( \int_{\{s \geq \sigma_0\}} \varphi_\sigma^6 \, d\mu^{\gcheck}\right)^{1/6}\\
& \leq C_1 R^{-1} \left( \int_{\{s \geq \sigma_0\}} \varphi_\sigma^6 \, d\mu^{\gcheck}\right)^{1/3} \\ & \qquad + C_2 R^{-1/2} \left( \int_{\{s \geq \sigma_0\}} \varphi_\sigma^6 \, d\mu^{\gcheck}\right)^{1/6}
\end{split}
\end{equation*}
for constants $C$, $C_1$ and $C_2$ independent of $\sigma$ and $R$. This shows that $\| \varphi_\sigma \|_{L^6( \{s \geq \sigma_0 \})}\leq CR^{-1/2}$ for a constant $C$ independent of $R$ and $\sigma$. Arguing as in the proof of  Proposition \ref{propYamabe} we conclude that  $\| \varphi_\sigma \|_{L^\infty ( \{s \geq 2 \sigma_0 \})}\leq CR^{-1/2}$. This in combination with the fact that $\varphi_\sigma$ are harmonic in $\{\sigma < s < 2\sigma_0\}$ %such that $\nu_{\gcheck}(\varphi_ \sigma)=0$ on $\{s=\sigma\}$ 
implies that $|\varphi_ \sigma|<CR^{-1/2}$ for the same constant $C>0$ in $\{\sigma < s < 2\sigma_0\}$ by a simple argument using the Harnack inequality.
% This is because $\varphi_\sigma$ cannot achieve its max/min on $\{s=\sigma\}$ otherwise the Neumann condition would fail, see e.g. Lemma 3.4 in Gilbarg-Trudinger. Alternatively, one can consider $\varphi_\sigma - CR^{-1\2}$. If this function is positive somewhere then by Harnack it is nonnegative everywhere, but this cannot be true as it is negative on one of the boundary components. Similarly for $\varphi_\sigma + CR^{-1\2}$. If it were negative somewhere then it would be nonpositive everywhere, but this cannot be true as it is positive on one of the boundary components. 
All in all, we obtain that $\| \varphi_\sigma  \|_{L^\infty(\{\sigma < s < \sigma^{-1}\})}\leq C R^{-1/2}$  for a constant $C$ independent of $R$ and $\sigma$. 

Let $\varphi$ be a subsequential limit of $\varphi_\sigma$ as $\sigma\searrow  0$. Then $\Delta^{\gcheck} \varphi - \tfrac{1}{8} \scal^{\gcheck} (\varphi +1) = 0$ in $\Sigma$ and  $c^{-1} \leq \varphi + 1 \leq c$ for some constant $c>1$ independent of $R$ in the view of the above uniform estimate for $\varphi_\sigma$.  Note that $\varphi$ is harmonic on $\{s \geq 2R\}$, in which case a simple asymptotic analysis as in the proof of Proposition \ref{propAsymptAnalysis} yields that $\varphi = a \,r^{-1} + O^{2,\beta}(r^{-2})$ as $r\to\infty$. To estimate the constant $a = a(R)$, we first note that $\int_{\{s=\sigma_0\}} \nu_{\gcheck}(\varphi) \, d\mu^{\gcheck}=0$ since each $\varphi_\sigma$ is harmonic on $\{\sigma <s<\sigma_0\}$ and satisfies the Neumann boundary condition on $\{s=\sigma\}$. It follows that                                            
\begin{equation}\label{eqIntbyParts}                                                                                                                         
\tfrac{1}{8} \int_{ \{s\geq \sigma_0\}} \scal^{\gcheck} (\varphi + 1) \, d\mu^{\gcheck} 
  = \int_{\{s\geq \sigma_0\}} \Delta^{\gcheck} \varphi \, d\mu^{\gcheck} = \lim_{R\to \infty} \int_{\{r=R\}} \nu_{\gcheck}(\varphi) \, d\mu^{\gcheck} 
 = - 4 \pi a.
\end{equation}
Combined with  the Taylor formula for the scalar curvature $\scal^{\gcheck}$ at $\gcheck=\delta$ (see e.g. \cite{Michel}), \eqref{eqIntbyParts} gives:
\[
\begin{split}
-32 \pi a(R) & = \int_{\Sigma} \scal^{\gcheck} (\varphi + 1) \, d\mu^{\gcheck} 
\\ & = \int_{\{R \leq r \leq 2R\}} \divg^\delta (\divg^\delta \gcheck - d \tr^\delta \gcheck) \, d\mu^\delta + O(R^{-1})\\
& =  \int_{\{r=2R\}} (\divg^\delta \gcheck - d \tr^\delta \gcheck) \left(\partial_r \right)\, d\mu^\delta \\ & \qquad \qquad -\int_{\{r=R\}} (\divg^\delta \gcheck - d \tr^\delta \gcheck) \left(\partial_r \right) \, d\mu^\delta + O(R^{-1}) \\ & = 16 \pi m - I(R) + O(R^{-1}),
\end{split}
\]
where 
\[
I(R) \definedas \int_{\{r=R\}} (\divg^\delta \gcheck - d \tr^\delta \gcheck) \left(\partial_r \right) \, d\mu^\delta \rightarrow 16 \pi m \hspace{0.5cm} \text{as} \hspace{0.5cm} R \rightarrow \infty.
\]
It follows that $\lim_{R\to\infty} a(R) = 0$, so the metric $\ghat \definedas v^4 \gcheck$ for $v=1+\varphi$ has all the required properties.
\end{proof}

\begin{proposition}\label{propOpeningEnds}
Let $\ghat$ be as in Proposition \ref{propDeform}. Then there is a positive function $w \in C^{2,\alpha}_{loc} (\Sigma)$ such that $\Delta^{\ghat} w \leq 0$ with strict inequality for large $r$, and $w = b \, r^{-1} + O^{2,\alpha} (r^{-2})$ as $r\to \infty$ for some constant $b \in \bR$. Moreover, there is a constant $c>1$  such that $c^{-1} (u v s)^{-1} \leq w \leq c (u v s)^{-1}$ as $s \to 0$.  
\end{proposition}

\begin{proof}
The proof is very similar to \cite[Proposition 13]{EichmairJang}. By Proposition \ref{propDeform}  we have $\ghat = v^4 (u \Psi)^4 \gtil$ in  $\{s<2\sigma_0\}$. Recall also that $\scal^{\ghat}=0$.  A simple computation using the definition of $\Psi$ (see the proof of Proposition \ref{propStep1}) shows that the metric $(vus)^{-4}\ghat = s^{-4} \Psi^4 \gtil $ has zero scalar curvature in $\{s<2\sigma_0\}$ hence $\Delta^{\ghat} (v u s)^{-1} = 0$ in $\{s<2\sigma_0\}$. Fix a non-negative function $w_0 \in C^{2,\alpha}_{loc} (\Sigma)$ such that it agrees with $(v u s)^{-1}$ in $\{s < 2 \sigma_0\}$, and such that $(\supp w_0) \cap \{s > \sigma_0\}$ is compact. Now fix a nonnegative function $q \in C^{2,\alpha}_{loc} (\Sigma)$ with  $(\supp q) \cap \{s < 2\sigma_0 \} = \emptyset $ and such that $q(x) = r^{-6}$ when $r=r(x)$ is large. Given $\sigma \in (0,\sigma_0)$, let $w_\sigma$ be the unique solution of  
\begin{eqnarray*}
-\Delta^{\ghat} (w_0 + w_\sigma) = q & \text{ on }& \{\sigma < s < \sigma^{-1}\},\\
w_\sigma = 0 &\text{ on }& \{s = \sigma\} \cup \{s = \sigma^{-1}\}.
\end{eqnarray*} 
Note that $w_0 + w_\sigma$ is positive by the maximum principle. We extend $w_\sigma$ by zero to a Lipschitz continuous function on all of $\Sigma$. Using the Sobolev inequality  \cite[Lemma 18]{EichmairJang}, the equation that $w_\sigma$ satisfies, and the properties of $q$ and $w_0$ we obtain
\[
\begin{split}
& C^{-1} \left(\int_{\{s \geq \sigma_0\}} |w_\sigma|^6 \, d\mu^{\ghat} \right)^{1/3} \\
& \leq \int_{\{s \geq \sigma_0\}} |d w_\sigma|^2 \, d\mu^{\ghat} \leq \int_{\{\sigma \leq s \leq \sigma^{-1}\}} |d w_\sigma|^2 \, d\mu^{\ghat}\\ & = \int_{\{\sigma \leq s \leq \sigma^{-1}\}} w_\sigma (q + \Delta^{\ghat} w_0) \, d\mu^{\ghat} \leq \int_{\{s \geq \sigma_0 \}} |w_\sigma| |q + \Delta^{\ghat} w_0| \, d\mu^{\ghat}  \\ & \leq \left(\int_{\{s \geq \sigma_0 \}} |w_\sigma|^6 \, d\mu^{\ghat} \right)^{1/6}
  \left(\int_{\{s \geq \sigma_0 \}}|q + \Delta^{\ghat} w_0|^{6/5} \, d\mu^{\ghat} \right)^{5/6}.
\end{split}
\] 
It follows that $\int_{\{s \geq \sigma_0\}} |w_\sigma|^6 \, d\mu^{\ghat}$ is bounded independently of $\sigma \in (0,\sigma_0)$. A standard argument as in the proof of Proposition \ref{propYamabe} %From this and the equation that $w_\sigma$ satisfies we obtain $C^{2,\alpha}$ estimates on $\{s(x)\geq 2\sigma_0\}$. Using that $w_\sigma$ is $\ghat_R$-harmonic on $\{\sigma < s(x) < 2 \sigma_0\}$ and that $w_\sigma$ vanishes on $s(x)=\sigma$, we obtain $L^\infty$-bound for $w_\sigma$ which is independent of $\sigma \in (0,\sigma_0)$. 
yields a uniform  $L^\infty$-bound for $w_\sigma$ and also allows us to pass to a subsequential limit when $\sigma \to 0$, thereby obtaining a nonnegative function $w \definedas w_0 + \lim_{i\to\infty} w_{\sigma_i} \in C^{2,\alpha}_{loc}(\Sigma)$ such that $-\Delta^{\ghat} w = q$. Since $w$ is a non-constant subharmonic function in $\Sigma$ we see that $w>0$ in $\Sigma$ by the Hopf maximum principle. The asymptotics of $w$ follow from the fact that $\ghat$ is asymptotically Schwarzschildean near infinity as a consequence of Proposition \ref{propDeform}.  Finally, recall that $w_0 = (v u s)^{-1}$ where  $c^{-1 }\leq uv \leq c$ for some $c>1$ in  $\{s < 2 \sigma_0\}$. Since $w>0$ is  bounded we conclude that $c^{-1} (u v s)^{-1} \leq w \leq c (u v s)^{-1}$ as $s \to 0$, up to increasing $c$ if necessary. 
\end{proof}

\section{Positive mass theorem in the asymptotically hyperbolic setting}\label{secPositivity}
In this section we prove the positivity part of Theorem \ref{thMain}.
\begin{theorem}\label{thPositivity}
 Let $(M,g,K)$ be a 3-dimensional asymptotically hyperbolic initial data set of type  
$(l,\beta, \tau,\tau_0)$ for $l\geq 6$, $0<\beta<1$, $\tfrac{3}{2} < \tau <3$  and
$\tau_0>0$. Assume that the dominant energy condition $\mu \geq |J|_g$
holds. Then the mass vector $(E,\vec{P})$ is causal future directed, that is $E \geq |\vec{P}|$.
\end{theorem}

\begin{proof}
We will first prove that $E \geq 0$ holds in the case when $(M,g,K)$ satisfies the assumptions of the theorem. Assume first that the initial data has Wang's asymptotics and satisfies the strict dominant energy condition $\mu>|J|_g$. By Proposition \ref{propYamabe} and Proposition \ref{propAsymptAnalysis} we know that in this case there is a Riemannian metric $\gtil_{u\Psi}$ which is asymptotically Euclidean (possibly with finitely many conical singularities) and such that 
\[
\mathcal{M} (\gtil_{u\Psi}) =  \alpha + 2 A \leq \alpha/2 =E.
\]
Then, by Proposition \ref{propDeform}, for any $N > 0$ there is  a radius $R_N>0$ and an asymptotically Schwarzschildean metric $\ghat_N \definedas \ghat (R_N)$ that retains the eventual conical singularities of the metric $\gtil_{u\Psi}$ and such that 
\[|
 \mathcal{M} (\gtil_{u\Psi}) -  \mathcal{M}(\ghat_N)| < 1/N. 
 \]
 Further, by Proposition \ref{propOpeningEnds} there is another asymptotically Schwarzschildean complete metric $\ghat^\varepsilon_N \definedas (1 + \varepsilon w_N)^4 \ghat_N$ that has nonnegative scalar curvature everywhere and strictly positive scalar curvature for large $r$.  Applying the Riemannian positive mass theorem of \cite{PMT1} and \cite{PMT2} \footnote{See also the two final remarks made in the proof of \cite[Proposition 14]{EichmairJang} which explain why the original proof of Schoen and Yau can be extended to account for non-asymptotically Euclidean ends such as $(C_i,\ghat^\varepsilon_N)$, $i=1,\ldots, l$.} we see that 
\[
 \mathcal{M}(\ghat^\varepsilon_N) =  \mathcal{M}(\ghat_N) + 2\varepsilon b_N \geq 0
\]
where $b_N $ is the leading order term in the expansion of $w_N$ for $r\to\infty$, see Proposition \ref{propOpeningEnds}. Since this holds for every $\varepsilon>0$ we conclude that $\mathcal{M}(\ghat_N) \geq 0$ for any $N$. Passing to the limit when $N\to\infty$ we conclude that
\begin{equation}\label{eqImportantForRigidity}
E \geq \mathcal{M}(\gtil_{u\Psi}) \geq 0.
\end{equation} 
Thus $E\geq 0$ holds when the initial data has Wang's asymptotics and the strict dominant energy condition holds. That $E\geq 0$ holds under the assumptions of the theorem follows at once by the density result of Theorem \ref{thPerturb}. 

To complete the proof, it remains to show that we have $E-|\vec{P}|\geq 0$. In fact, in the asymptotically hyperbolic case this is a straightforward consequence of $E \geq 0$. Indeed, suppose on the contrary that we have $0 \leq E < |\vec{P}|$. Since boosts of Minkowski spacetime restrict to (nonlinear) isometries of the unit upper hyperboloid, we may compose the given asymptotically hyperbolic coordinate chart with the boost of the slope $\theta\in (0,1)$  and thereby obtain another asymptotically hyperbolic coordinate chart with the same asymptotic properties defined on the complement of a compact set in $M$. Recall that the mass vector transforms equivariantly under the changes of coordinates near infinity (see e.g. \cite{Michel}), in particular the first component of the mass vector in the boosted chart is $E'=\frac{E-\theta |\vec{P}|}{1-\theta^2}$. Clearly, for any  $\theta \in \left(\frac{E}{|\vec{P}|}, 1\right)$ we have $E'<0$, which is a contradiction. Note that such an argument does not directly apply in the asymptotically Euclidean setting because boosts of Minkowski spacetime do not restrict to isometries of constant time slices, cf. the final remark in \cite{EHLS}. 
\end{proof}

\section{Rigidity}\label{secRigidity}
In this section we prove the rigidity part of Theorem \ref{thMain}. 

\begin{theorem}\label{thRigidity}
Let $(M,g,k)$ be initial data satisfying the assumptions of Theorem \ref{thPositivity}. If $(M,g,K)$ has Wang's asymptotics and $E=0$ then $(M,g)$ can be embedded isometrically into Minkowski spacetime as a spacelike graphical hypersurface with second fundamental form $K$.
\end{theorem}
\begin{remark}
This result does not seem to be optimal for the following reasons:

\begin{itemize}

\item We have to assume Wang's asymptotics, which is rather restrictive. This assumption needs to be imposed so that we can solve the Jang equation. Solving the Jang equation for general asymptotics would require the existence of ``uniform'' barriers (cf. the proof of \cite[Proposition 16]{EichmairJang}), something that our construction does not provide. 

\smallskip
 
\item In the view of the results in the asymptotically Euclidean and asymptotically anti-de Sitter setting (see e.g. \cite[Theorem 3]{HuangLeeRigidity}, \cite[Theorem 1.2]{ChruscielMaerten}), and \cite[Theorem 4]{ChruscielMaertenTod}) 
one would expect the conclusion of the theorem to hold under the weaker assumption $E=|P|$, meaning that the mass vector is future directed null. It appears that the Jang equation reduction technique is not capable of providing results of this kind, as \cite[Theorem 2]{PMT2}, \cite[Theorem 3]{EichmairJang}, and our Theorem  \ref{thRigidity} indicate. The same comment can presumably be made about spinor methods in the asymptotically hyperbolic ``hyperboloidal'' setting, see e.g. \cite[Theorem 5.1]{Maerten}. At the same time, the optimal rigidity theorem  for asymptotically hyperbolic manifolds has recently been proven in  \cite{HJM}. It is feasible that the methods of \cite{HJM} and \cite{HuangLeeRigidity} can be used to prove more general rigidity results for asymptotically hyperbolic initial data than Theorem \ref{thRigidity}.  
\end{itemize}
\end{remark}

\begin{proof}[Proof of Theorem \ref{thRigidity}] 
We denote the chart at infinity with respect to which $(M,g,K)$ has Wang's asymptotics by $\Phi$. Under the assumptions of the theorem there is a sequence of initial data $(g_i, K_i)$, $i=1,2,\ldots$, suitably asymptotically hyperbolic with respect to $\Phi$,  satisfying the assumptions of Theorem \ref{thPositivity}, and such that the strict dominant energy condition $\mu_i > |J_i|_{g_i}$ holds, see \cite[Proposition 5.2]{DahlSakovich}) for details. Furthermore, there is a sequence of charts $\Phi_i$, constructed by means of a standard procedure called \emph{adjustment}\footnote{a term coined in \cite{MassInvariants}; this procedure is also referred to as \emph{change of conformal gauge}.} such that  $(g_i, K_i)$ have Wang's asymptotics with respect to $\Phi_i$ (again, the reader is referred to the proof of \cite[Theorem 5.2]{DahlSakovich} for details). We may use the chart $\Phi_i$ to construct a geometric solution $\Sigma_i$ of the Jang equation with respect to every initial data set $(g_i, K_i)$, $i=1,2,\ldots$. In particular, inspecting the arguments of Section \ref{secBarriers}, we see that there exist uniform constants $R>0$ and $C>0$ such that for every  $i$ the barrier functions $f_{+,i}$ and $f_{-,i}$ are defined on $\{r\geq R\}$ and satisfy $|f_{\pm,i} - \sqrt{1+r^2} - \alpha_i \ln r - \psi_i| \leq C r^{-1 + \varepsilon}$ there. Here $\alpha_i= 2 E_i$ is twice  the energy of the initial data set  $(M_i,g_i,K_i)$ and $\psi_i : \mathbb{S}^2\to \bR$ such that  $\int_{\mathbb{S}^2} \psi_i\, d\mu^\sigma = 0$ is defined in terms of the asymptotic expansions of the initial data by
\begin{equation}\label{eqPsii}
\Delta^{\mathbb{S}^2} \psi_i  =\mathbf{M}_i\definedas\tfrac{1}{2} \tr^\sigma \mathbf{m}_i + \tr^\sigma \mathbf{p}_i - \alpha_i,
\end{equation}
see Section \ref{secPrelim} for details.
Note that the described asymptotics of the barrier functions $f_{\pm,i} $ are the same in either of the charts $\Phi_i$ and $\Phi$, as the adjustment will introduce only lower order corrections in this case (see the proof of \cite[Theorem 5.2]{DahlSakovich} for details).  As in  Section \ref{secExistence}, the hypersurfaces $\Sigma_i  \subset M\times \bR$ satisfy the uniform curvature estimates and we may pass to a subsequential limit as $i\to \infty$, thereby obtaining a geometric solution of the Jang equation with respect to initial data $(M,g,K)$.  Clearly, this limit has a connected component $\Sigma\subset M\times \bR$ given as the graph of a function $f$ such that its domain $U$ contains the set $\{r\geq R\}$.  To clarify the asymptotics of the function $f$, we first note that $\lim_{i\to \infty} \alpha_i=2E = 0$ holds by the continuity of the mass functional. Further, define $\psi$ such that $\int_{\mathbb{S}^2} \psi\, d\mu^\sigma = 0$ by
\begin{equation}\label{eqPsiLimit}
\Delta^{\mathbb{S}^2} \psi =\mathbf{M}\definedas \tfrac{1}{2} \tr^\sigma \mathbf{m}+ \tr^\sigma \mathbf{p},
\end{equation}
then as a consequence of \eqref{eqPsii}, \eqref{eqPsiLimit} and the Poincare inequality we obtain
\[
\begin{split}
\int_{\mathbb{S}^2} (\psi- \psi_i)^2 \, d\mu^\sigma 
& \leq C \int_{\mathbb{S}^2} |\nabla \psi - \nabla \psi_i|^2 \, d\mu^\sigma \\
& = C \int_{\mathbb{S}^2}  (\mathbf{M}_i - \mathbf{M})( \psi - \psi_i)  \, d\mu^\sigma \\
& \leq C \left( \int_{\mathbb{S}^2}  (\mathbf{M} - \mathbf{M}_i)^2 \, d\mu^\sigma\right)^{1/2} \left(\int_{\mathbb{S}^2} (\psi - \psi_i)^2 \, d\mu^\sigma \right)^{1/2}.
\end{split}
\]
Since $\mathbf{M} -\mathbf{M}_i\to 0$ uniformly on $\mathbb{S}^2$ (see the proof of \cite[Theorem 5.2]{DahlSakovich} for details)  it follows that $\psi - \psi_i$ converges to zero in $L^2 (\mathbb{S}^2)$. A standard bootstrap argument then yields  $\psi_i \to \psi$ in $C^{3,\alpha} (\mathbb{S}^2)$. We conclude in the view of the above uniform estimate for barriers that $f=\sqrt{1+r^2} + \psi + O(r^{-1 + \varepsilon})$. Arguing as in Section \ref{secJangAE} we may now show that the metric $\gbar=g+df\otimes df$ induced on $\Sigma \subset M\times \bR$  is asymptotically Euclidean, with the properties described in Corollary \ref{corAEh}.

Note however that the conclusion of Proposition \ref{propConfStructure1} might fail to hold  for the boundary components $\partial U_i$ of the domain of the graphing function $f$ as we do not necessarily have a strict inequality in the dominant energy condition $\mu \geq |J|_g$. Therefore the analysis of the conformal structure of $\Sigma$ cannot be approached directly by the methods of Section \ref{secConformal}. 

As in \cite[Proof of Proposition 16]{EichmairJang} we choose $t_0^i\nearrow \infty$ to be a sequence such that $\pm t_0^i$ are regular values for both $f_i$ and $f$. Let $\gtil_i$ be the metrics on $\Sigma_i$ as in Proposition \ref{propExactCylindrical} such that $\gtil_i = \gbar_i$ on $N_i\definedas\Sigma_i \cap (M\times (-t^i_0,t_0^i))$. Further, let $u_i \in C^{2,\alpha}_{loc}(\Sigma_i)$ be the solution of $-\Delta^{\gtil_i} u^i + \tfrac{1}{8} \scal^{\gtil_i} u^i=0$ as in Proposition \ref{propYamabe}.  Arguing as in the proof of Proposition \ref{propAsymptAnalysis} we see that
\[
0\leq \int_{N_i} |du_i|^2_{\gbar_i}  \leq -(\alpha_i + 4A_i) \vol \mathbb{S}^2.
\] 
From the above discussion we know that $\lim_{i\to \infty} \alpha_i = 0$. Furthermore, the proof of Theorem \ref{thPositivity} shows that $\alpha_i + 2A_i \geq 0$ hence $-\frac{\alpha_i}{2} \leq A_i \leq -\frac{\alpha_i}{4}$, hence $\lim_{i\to \infty} A_i = 0$. In conjunction with the Sobolev inequality and the equation that $u_i$ satisfies we see that $u_i  \to 1$ as $r\to \infty$ uniformly in $i$. Using
standard elliptic theory  we conclude that $u_i$ converges in $C^{2,\alpha}_{loc}$ to the constant function one on $\Sigma$. Inspecting the proof of Proposition \ref{propAsymptAnalysis} once more we also conclude that $\scal^{\gbar}=0$, and $A=K$ on $\Sigma$.

Now recall that the asymptotically Euclidean metric $\gbar = g+ df \otimes df$ satisfies $\gbar = \delta + O_2(r^{-1})$.  Consequently, the asymptotically Euclidean initial data set $(\Sigma,\gbar,0)$ has Sobolev type $(2,p,q,q_0,\alpha)$, as defined in \cite[Definition 1]{EHLS}, for $p>3$, $q\in(\frac{1}{2},1)$, for some $\alpha \in (0,1)$ and for every $q_0>0$. Since $\scal^{\gbar}=0$ and $\mathcal{M}(\gbar)=0$, a version of the variational argument used by Schoen and Yau in \cite{PMT1} to prove the Riemannian positive mass theorem, yields that $(\Sigma,\gbar)$ is isometric to the Euclidean space. The reader is referred to \cite[Proof of Proposition 16]{EichmairJang} where the details of this argument are provided. Combining this with the fact that  $A=K$, it follows as in \cite[p. 260]{PMT2}  that  $g$ respectively $K$ arise as the induced metric respectively the second fundamental form of the graph of the function $f:\bR^3 \to \bR$ in the Minkowski spacetime $(\bR \times \bR^3,-dt^2 + \delta)$.
\end{proof}

\appendix \label{secAppendix} 
\section{Christoffel symbols}\label{appChristoffel}

For $g$ as in Definition \ref{defAHdata} the Christoffel symbols are as follows:
\begin{equation*}
\Gamma_{rr}^r=-\frac{r}{1+r^2}, \hspace{1cm} \Gamma_{rr}^{\mu}=0, \hspace{1cm} \Gamma_{r\mu}^{r}=0,
\end{equation*}
\begin{equation*}
\Gamma_{r\mu}^{\nu}=\frac{1}{2}g^{\lambda\nu}\partial_r g_{\lambda\mu}, \hspace{1cm} \Gamma_{\mu\nu}^r=-\frac{1}{2}(1+r^2)\partial_r g_{\mu \nu},
\end{equation*}
\begin{equation*}
\begin{split}
\Gamma_{\mu\nu}^{\kappa}&=\frac{1}{2}g^{\kappa \lambda}\left(\frac{\partial g_{\lambda\nu}}{\partial x^{\mu}}+\frac{\partial g_{\mu \lambda}}{\partial x^{\nu}}-\frac{\partial g_{\mu \nu}}{\partial x^{\lambda}}\right).%\\&= ^{S^2}\Gamma_{\mu\nu}^{\kappa}+O(r^{-3}).
\end{split}
\end{equation*}

\section{The barrier method for boundary gradient estimates}\label{appBarMet} 
Here we recall barrier method for deriving boundary gradient estimates as described in \cite[Chapter 14]{GT}, applied to the boundary value problem \eqref{eqAux1}-\eqref{eqAux2}. 

\begin{proposition}\label{propBoundBar}
Suppose that in some neighborhood $U$ of $\partial \Omega$ we have two functions $\fbar,\funder\in C^2(\Omega\cap U)\cap C^1(\overline{\Omega}\cap U)$ such that 
\begin{equation*}
\hspace{1cm} H_g(\fbar)-s\tr_g(K)(\fbar)< \tau \fbar, \hspace{0.5cm} H_g(\funder)-s\tr_g(K)(\funder)> \tau \funder \hspace{0.5cm} \text{in} \hspace{0.5cm} U\cap\Omega,
\end{equation*}
and 
\begin{equation*}
\fbar=\funder=s\phi \hspace{0.5cm}\text{ on }\hspace{0.5cm} \partial \Omega.
\end{equation*}
If $f_s\in C^2(\Omega)\cap C^0(\overline{\Omega})$ is a solution of \eqref{eqAux1}-\eqref{eqAux2} such that $\funder\leq f_s \leq \fbar$ on $\partial (\Omega\cap U)$ then $|df_s|_g$ restricted to $\partial \Omega$ is bounded by a constant depending only on $\fbar$ and $\funder$.
\end{proposition}

\begin{proof}
Subtracting $H_g(f_s)-s\tr_g(K)(f_s)-\tau f_s=0$ from $H_g(\fbar)-s\tr_g(K)(\fbar)-\tau \fbar<0$ we get 
\begin{equation*}
\begin{split}
0>&\left(g^{ij}-\frac{\fbar^i\fbar^j}{1+|d\fbar|_g^2}\right)\frac{\Hess_{ij}\left(\fbar-f_s\right)}{\sqrt{1+|d\fbar|_g^2}}\\&+\left(\left(g^{ij}-\frac{\fbar^i\fbar^j}{1+|d\fbar|_g^2}\right)\frac{1}{\sqrt{1+|d\fbar|_g^2}}-\left(g^{ij}-\frac{(f_s)^i (f_s)^j}{1+|d f_s|_g^2}\right)\frac{1}{\sqrt{1+|df_s|_g^2}}\right)\Hess_{ij} (f_s)\\&+s\left(\frac{(f_s)^i (f_s)^j}{1+|d f_s|_g^2}-\frac{\fbar^i\fbar^j}{1+|d\fbar|_g^2}\right)K_{ij}-\tau(\fbar-f_s)\\=& \left(g^{ij}-\frac{\fbar^i\fbar^j}{1+|d\fbar|_g^2}\right)\frac{\Hess_{ij}\left(\fbar-f_s\right)}{\sqrt{1+|d\fbar|_g^2}}+b^i \nabla_i \left(\fbar-f_s\right)-\tau(\fbar-f_s),
\end{split}
\end{equation*}
where the existence of locally bounded functions $b_i$ follows from the mean value theorem. It is clear from the above inequality that $\fbar-f_s$ cannot have a nonpositive interior minimum in $\Omega\cap U$. Since $\fbar\geq f_s$ on  $\partial (\Omega\cap U)$, we conclude that $\fbar\geq f_s$ in $\Omega\cap U$. The same argument shows that $f_s\geq \funder$ in $\Omega\cap U$. From the fact that $\funder\leq f_s\leq \fbar$ in $\Omega\cap U$, and $\funder=f_s=\fbar$ on $\partial \Omega$ we conclude that 
\begin{equation*}
\frac{\funder(p)-\funder(p_0)}{|p-p_0|}\leq \frac{f_s(p)-f_s(p_0)}{|p-p_0|}\leq \frac{\fbar(p)-\fbar(p_0)}{|p-p_0|}
\end{equation*}
for any $p_0\in\partial\Omega$ and $p\in \Omega\cap U$. The result follows by comparing partial derivatives of $f_s$ with the respective partial derivatives of $\funder$ and $\fbar$.
\end{proof}

\section{Some basic properties of Fermi coordinates}\label{secLevelFermi} 

In this appendix we include the proof of the result which is repeatedly used in Section \ref{secJangAE}. We would like to remark that this result is the main reason behind the regularity assumptions that are made throughout the paper: as we will see, for this result to hold certain curvature bounds are required. Notations and conventions are as in Section \ref{secFermi}. 

\begin{proposition}
There exist constants $\rho_0>0$ and $C>0$ such that $|A_\rho|<C$ and $\frac{1}{C} \delta_{ij} \leq (g_\rho)_{ij} \leq C \delta_{ij}$ for any $0 \leq \rho \leq \rho_0$. Furthermore, all partial  derivatives of $(g_\rho)_{ij}$ and $(A_\rho)^i_{\phantom{i}j}$ up to order 3   in the Fermi coordinates are bounded. 
\end{proposition}
\begin{proof}
The first part of this result is proven by a standard comparison argument, cf. \cite[Chapter 5, Theorem 27]{Petersen} and  \cite[Theorem 15]{BahuaudThesis}. Recall that our convention for the second fundamental form of the hypersurfaces $\Sigma_\rho$ of constant $\rho$ is $A_\rho(X,Y)=\langle\nabla_X Y, \partial_\rho\rangle$. It is well-known that  the respective shape operator (the associated  $(1,1)$-tensor) which we denote by the same notation $A_\rho$ satisfies the Mainardi equation
\begin{equation}\label{eqRiccati}
-\partial_\rho (A_\rho)^i_{\phantom{i}j} + (A_\rho)^i_{\phantom{i}k} (A_\rho)^k_{\phantom{k}j}=R^i_{\phantom{i} \rho \rho j},
\end{equation}
see e.g. \cite{Petersen}. 
We rewrite this equation  in a simplified form as 
\begin{equation*}
-A'(\rho) + A^2 (\rho)=-R_N(\rho),
\end{equation*}
where $R_N(\rho)$ is a normal sectional curvature operator defined by $\langle R_N(\rho) V, V \rangle = \sec(V,\partial_\rho)$ for $V \bot \partial_\rho$, the prime denotes the derivative with respect to $\rho$, and we suppress the dependence on the tangential coordinates. The eigenvalues of the shape operator $A(0)$ are bounded, and we want to prove that the same is true for the eigenvalues of $A(\rho)$ for $0<\rho\leq \rho_0$.  

Let $\Lambda(\rho)$ be the largest eigenvalue of $A(\rho)$. Since $\Lambda(\rho)$ is obtained through a maximum procedure (from the Rayleigh quotient) it is Lipschitz continuous and hence differentiable almost everywhere. At a point $\tilde\rho$ where it is differentiable we pick a unit eigenvector $v$ with respect to the Euclidean metric. Then we extend it to a parallel vector field $v$ such that $v(u,\rho)=v(u,\tilde\rho)$ for  $\rho \in [0,\rho_0]$. Set $\varphi (\rho)=v^T A(\rho) v$. Then $\varphi(\tilde\rho)=\Lambda (\tilde\rho)$ and $\varphi(\rho) \leq \Lambda (\rho)$ for $\rho \in [0,\rho_0]$, hence $\varphi'(\tilde\rho)=\Lambda'(\tilde\rho)$. As a consequence, we have 
\begin{equation*}
\begin{split}
-\Lambda'(\tilde\rho) + \Lambda^2 (\tilde\rho) & = - \varphi'(\tilde\rho) + \varphi^2 (\tilde\rho)\\
                                       & = v^T (-A'(\tilde\rho) + A^2 (\tilde\rho)) v\\
																			 & = v^T (-R_N (\tilde\rho)) v.%\\
																			 %& = -\sec(\rho_0) (v,\partial_\rho).
\end{split}
\end{equation*}  
Since the curvature term in the right hand side is uniformly bounded we conclude that $\Lambda$ satisfies the differential inequality 
\[
-C_1<-\Lambda'(\rho) + \Lambda^2 (\rho) <C_1
\]
for some constant $C_1>0$ and for almost every $\rho \in [0,\rho_0]$. 

Let now $\mu (\rho)=\sqrt{C_1}\tan \left(\sqrt{C_1}\rho + \arctan \frac{C_0}{\sqrt{C_1}}\right)$ be the solution of  the initial value problem
\begin{eqnarray*}
-\mu'(\rho) + \mu^2 (\rho) &=& -C_1,\\
                    \mu(0) &=& C_0,
\end{eqnarray*}
for some  $C_0 > |\Lambda(0)|$. Up to decreasing $\rho_0$ if necessary, we may assume that  $\mu (\rho)$ is defined and bounded as long as $\rho \in [0,\rho_0]$. Furthermore, we have
\begin{equation}\label{eqDiffMonotone}
\Lambda'(\rho)-\Lambda^2(\rho) <C_1 = \mu'(\rho) - \mu^2 (\rho)
\end{equation} 
and
\begin{equation}\label{eqSumMonotone}
-(\Lambda(\rho)+ \mu(\rho) )'+(\Lambda^2(\rho)+\mu^2 (\rho) )<0
\end{equation} 
for almost every $\rho \in [0,\rho_0]$. Note also that $|\Lambda(0)|< \mu (0)$. 

We will now show that $\Lambda(\rho)<\mu(\rho)$ for all $\rho \in [0,\rho_0]$. Since $\Lambda$ is Lipschitz continuous we have $\Lambda(\rho) = \Lambda(0) + \int_0^\rho \Lambda'(\tau) \, d\tau$ for all $\rho \in [0,\rho_0]$.  
Combining this with \eqref{eqSumMonotone} we find that 
\[
\Lambda(\rho)+\mu(\rho)=\int_0^\rho (\Lambda(\tau)+\mu(\tau))'\,d\tau + \mu(0)+\Lambda(0)>0,
\]
hence $\Lambda(\rho)>-\mu(\rho)$ for $\rho\in[0,\rho_0]$. Furthermore, by \eqref{eqDiffMonotone} we have
\begin{equation}\label{eqComparisonLipschitz}
\Lambda(\rho)-\mu(\rho) = \int_0^\rho (\Lambda'(\tau) - \mu'(\tau))\, d\tau + \Lambda(0) - \mu(0)<\int _0^\rho (\Lambda^2(\tau) - \mu^2(\tau))\, d\tau 
\end{equation}
for $\rho  \in (0,\rho_0]$. Now let $\rho_*=\inf \{\rho: \Lambda(\rho)>\mu(\rho)\}$. Since  $\Lambda(0)< \mu (0)$ we have $\rho_* >0$. On the one hand, we have $\Lambda(\rho_*)=\mu(\rho_*)$. On the other hand, we have $\Lambda(\rho)<\mu(\rho)$ for $0\leq \rho<\rho_*$  and $\Lambda(\rho)>-\mu(\rho)$ for $\rho\in[0,\rho_0]$. Then $\Lambda^2(\rho)- \mu^2(\rho)<0$ for  $\rho\in[0,\rho_*]$ so \eqref {eqComparisonLipschitz} yields $\Lambda(\rho_*)<\mu(\rho_*)$, a contradiction. It follows that $\Lambda(\rho)<\mu(\rho)$ for all $\rho \in [0,\rho_0]$. 

Arguing as above one shows that the smallest eigenvalue $\lambda(\rho)$ of $A(\rho)$ satisfies for almost every $\rho  \in [0,\rho_0]$ and some constant $C_1>0$ the differential inequality 
\[
-C_1<-\lambda'(\rho) + \lambda^2 (\rho) <C_1.
\]
Then $\lambda' >\lambda^2-C_1>-C_1$ hence 
\[
\lambda(\rho)=\lambda(0)+\int_0^\rho \lambda'(\tau)\,d\tau > \lambda(0)-C_1\rho >\lambda(0)-C_1 \rho_0
\] 
holds for all $\rho  \in [0,\rho_0]$. This shows that $\lambda(\rho)$ is uniformly bounded from below. Combining this with the above estimate for $\Lambda(\rho)$ the uniform bound $|A_\rho|<C$ follows.

To obtain the metric estimate, we note that $g_\rho$ satisfies the linear equation
\begin{equation}\label{eqMetricEvolution}
\partial_\rho (g_\rho)_{ij}= - 2 (A_\rho)^k_{\phantom{k}i} (g_\rho)_{kj}.
\end{equation}
Let $\Theta(\rho)$ be the largest eigenvalue of $g_\rho$ with respect to the Euclidean metric. Again, $\Theta=\Theta(\rho)$ is Lipschitz continuous and, in the view of the shape operator estimate, from \eqref{eqMetricEvolution}  we see that whenever $\Theta$ is differentiable it satisfies $\Theta'(\rho)\leq C_1 \Theta (\rho)$, or equivalently,  $(\Theta(\rho) e^{-C_1\rho})'\leq 0$ for some $C_1>0$. Let $\Gamma(\rho)=C_0 e^{C_1 \rho}$ be the solution of the equation  $(\Gamma(\rho) e^{-C_1\rho})'= 0$ such that $\Gamma(0)=C_0 > \Theta (0)$. Then 
\[
(\Theta(\rho) - \Gamma(\rho))e^{-C_1\rho}=\int_0^\rho ((\Theta(\tau) - \Gamma(\tau))e^{-C_1\tau})'\, d\tau + \Theta(0)-\Gamma(0) <0
\]  
thus $\Theta(\rho)<\Gamma(\rho)= C_0 e^{C_1\rho}$ for all $\rho \in [0,\rho_0]$. Similar analysis applies to the lowest eigenvalue and  the desired bound $\frac{1}{C} \delta_{ij} \leq (g_\rho)_{ij} \leq C \delta_{ij}$ for some $C>0$ follows. 

Next we observe that the obtained estimates for $(g_\rho)_{ij}$ and $(A_\rho)^i_{\phantom{i}j}$ in combination with \eqref{eqRiccati} and \eqref{eqMetricEvolution} yield the required bounds on $\partial_\rho (A_\rho)^i_{\phantom{i}j}$ and $\partial_\rho (g_\rho)_{ij}$. In order to prove that $\partial_k (A_\rho)^i_{\phantom{i}j}$ and   $\partial_k (g_\rho)_{ij}$ are bounded we may argue as in \cite[Section 3]{BahuaudGicquaud}. As a consequence of \eqref{eqRiccati} and \eqref{eqMetricEvolution} we have 
\begin{subequations}
\begin{eqnarray}
\partial_\rho \partial_k (A_\rho)^i_{\phantom{i}j} & = &  \partial_k (A_\rho)^i_{\phantom{i}l} (A_\rho)^l_{\phantom{k}j}+ (A_\rho)^i_{\phantom{i}l}                                                                        \partial_k (A_\rho)^l_{\phantom{k}j} - \partial_k R^i_{\phantom{i} \rho \rho j}, 
                                                                       \label{eqTangentialDers1} \\
            \partial_\rho \partial_k (g_\rho)_{ij} & = & - 2 \partial_k (A_\rho)^l_{\phantom{k}i} (g_\rho)_{lj} - 2 (A_\rho)^l_{\phantom{k}i} 
                                                                       \partial_k(g_\rho)_{lj}\label{eqTangentialDers2}.
\end{eqnarray}
\end{subequations}
By the well-known formula relating the coordinate and covariant derivatives, using the fact that $\Gamma_{k\rho}^\rho = 0$ and the properties of the curvature tensor, we obtain
\begin{equation*}
\partial_k R^i_{\phantom{i} \rho \rho j} = \nabla_k  R^i_{\phantom{i} \rho \rho j} - \Gamma_{kl}^i R^l_{\phantom{i} \rho \rho j} + \Gamma_{k\rho}^l R^i_{\phantom{i} l \rho j} + \Gamma_{k\rho}^l R^i_{\phantom{i} \rho l j} + \Gamma_{k j}^l R^i_{\phantom{i} \rho \rho l}.
\end{equation*}
Since $\Gamma^l_{k\rho}=-(A_\rho)^l_{\phantom{l}k}$, it follows by the above estimates for $g_\rho$ and $A_\rho$ that all terms in the right hand side of this formula are bounded, possibly except for $\Gamma^i_{kl}$ and $\Gamma^l_{kj}$, which in their turn can be written as a linear combination of the first order coordinate derivatives of $g_\rho$ with bounded coefficients. As a consequence, the above system can be compactly written as
\begin{subequations}
\begin{eqnarray*}
(\partial A_\rho)' & = & K_1 \partial A_\rho + K_2 \partial g_\rho +K_3, \\
(\partial g_\rho)' & = & K_4 \partial A_\rho + K_5 \partial g_\rho,
\end{eqnarray*}
\end{subequations}
where $K_i$, $i=1,\ldots,5$, are $3^3 \times 3^3$ matrices, and $\partial A_\rho$ and $\partial g_\rho$ are treated as vectors in $\bR^{3^3}$ with the respective components $\partial_k (A_\rho)^i_{\phantom{i}j}$ and $\partial_k (g_\rho)_{ij}$. We will not need the explicit form of the matrices $K_i$, only the fact that their entries are bounded. We set $x(\rho)=|\partial A_\rho|$ and $y(\rho)=|\partial g_\rho|$. These functions are continuous and smooth as long as they are nonzero. % Note also that if either of these functions is zero on an interval the system uncouples, and a comparison argument as above shows that $\partial g_\rho$ is bounded on that interval.
 Moreover, it follows by Cauchy-Schwartz inequality that $x'\leq |(\partial A_\rho)'|$ and $y'\leq |(\partial g_\rho)'|$ whenever $x$ and $y$ are nonzero. As a consequence we have 
\begin{subequations}
\begin{eqnarray*}
x' & \leq &  c_1 x + c_2 y + c_3, \\
y' & \leq & c_4 x + c_5 y + c_6
\end{eqnarray*}
\end{subequations}
for some constants $c_i>0$, $i=1,\ldots, 6$. By \cite[Theorem 10]{BahuaudPacific} we conclude that $x<\tilde{x}$, $y<\tilde{y}$ on $[0,\rho_0]$, where $(\tilde{x},\tilde{y})$ is a smooth positive solution of the system 
\begin{subequations}
\begin{eqnarray*}
\tilde{x}' & = &  c_1 \tilde{x}+ c_2 \tilde{y}+ c_3, \\
\tilde{y}' & = & c_4 \tilde{x} + c_5 \tilde{y} + c_6
\end{eqnarray*}
\end{subequations}
for $\rho \in [0,\rho_0]$ such that $x(0)<\tilde{x}(0)$ and $y(0)<\tilde{y}(0)$. (That such a solution exists is a simple consequence of Picard-Lindel\"of theorem; note that we may need to decrease $\rho_0$ in order to ensure that the solution  remains positive in $[0,\rho_0]$.)
It follows that $\partial_k(g_\rho)_{ij}$ and $\partial_k (A_\rho)^i_{\phantom{i}j}$ are bounded for $\rho\in [0,\rho_0]$. As a consequence of these estimates  and  \eqref{eqTangentialDers1}-\eqref{eqTangentialDers2},  we see that $\partial_\rho \partial_k (g_\rho)_{ij}$ and $\partial_\rho \partial_k (A_\rho)^i_{\phantom{i}j}$ are bounded. Taking one more partial derivative of \eqref{eqRiccati} and \eqref{eqMetricEvolution} with respect to $\rho$ it also follows that $\partial_\rho \partial_\rho (g_\rho)_{ij}$ and $\partial_\rho \partial_\rho(A_\rho)^i_{\phantom{i}j}$ are bounded.

With the above estimates at hand, we take one more tangential derivative of \eqref{eqTangentialDers1}-\eqref{eqTangentialDers2}, and use standard formulae relating covariant and coordinate derivatives to conclude that
\begin{subequations}
\begin{eqnarray*}
(\partial \partial A_\rho)' & = & K_1 \partial \partial A_\rho + K_2 \partial \partial g_\rho +K_3, \\
(\partial \partial g_\rho)' & = & K_4 \partial \partial A_\rho + K_5 \partial \partial g_\rho +K_6,
\end{eqnarray*}
\end{subequations}
where $K_i$, $i=1,\ldots,6$, are $3^4 \times 3^4$ matrices with bounded entries, and $\partial \partial A_\rho$ and $\partial \partial g_\rho$ are treated as vectors in $\bR^{3^4}$ with the respective components $\partial_k \partial_l (A_\rho)^i_{\phantom{i}j}$ and $\partial_k \partial_l (g_\rho)_{ij}$. Repeating the above argument, we are again in a position to apply  \cite[Theorem 10]{BahuaudPacific}  and the boundedness of $  \partial_k \partial_l (A_\rho)^i_{\phantom{i}j}$ and $\partial_k \partial_l (g_\rho)_{ij}$ follows. 

Similar analysis yields the desired estimates for the third order coordinate derivatives of $g_\rho$ and $A_\rho$.
\end{proof}

\begin{remark}
In order to keep the proof of Proposition \ref{propLevel} as elementary as possible, we only used very rough bounds for the geometry of $(M\times\bR,g+dt^2)$. It is possible that the estimates of Proposition \ref{propLevel} can be improved if one uses more accurate bounds, cf. \cite[Section 3]{BahuaudGicquaud}.  However, as Proposition \ref{propLevel} in its current form suffices for our purposes we choose not to proceed in that direction.
\end{remark}

\section{Some asymptotic expansions}\label{secSigmaMinus}

In this article we repeatedly make use of the following two lemmas.

\begin{lemma}\label{lemma2}
Let $(M,g,K)$ be an asymptotically hyperboloidal initial data with Wang's asymptotics in the sense of Definition \ref{defAHdata} for $l\geq 3$. If $f:M\to \bR$ is such that $f=\sqrt{1+r^2} + \alpha \ln r + \psi + O_{3}(r^{-1+\varepsilon})$ then 
\begin{itemize}
\item[1)] The components of the induced metric $\gbar_{ij} = g + f_i f_j$ are given by
\[
\begin{split}
\gbar_{rr} & = 1 + 2 \alpha \, r^{-1} + O_2(r^{-2+\varepsilon}),\\
\gbar_{r\mu} & = \psi_\mu + O_2(r^{-1+\varepsilon}),\\
\gbar_{\mu\nu} & = r^{2}\sigma_{\mu\nu} + \psi_\mu \psi_\nu + O_2(r^\varepsilon).  
\end{split}
\]

\item[2)] If $\gbar^{ij}$ is given by $\gbar^{ik} \gbar_{kj} = \delta^i_j$ then
\[
\begin{split}
\gbar^{rr} & = 1 - 2 \alpha \, r^{-1} + O_2(r^{-2+\varepsilon}),\\
\gbar^{r\mu} & = - r^{-2} \sigma^{\mu\nu} \psi_\nu + O_2(r^{-3+\varepsilon}),\\
\gbar^{\mu\nu} & = r^{-2}\sigma^{\mu\nu} + O_2(r^{-5}).
\end{split}
\]
\item[3)] The components of the downward pointing unit normal $\nu$ of $\graph f \subset (M\times \bR,g+dt^2)$ satisfy 
\[
\begin{split}
\nu^t & = -r^{-1} + O_2(r^{-2}),\\
\nu^r & = r+O_2(r^{-1}),\\
\nu^\mu & = O_2(r^{-3}).  
\end{split}
\]
As a consequence, the Ricci curvature of the product metric $\ghat = g+dt^2$ satisfies $\ric(\nu, \nu) = -2 + O_1 (r^{-2})$.
\item[4)] The components of the second fundamental form $A_{ij} = \hess_{ij} f (1+|df|_g)^{-1/2}$ of $\graph f \subset M\times \bR$ are given by
\begin{equation*}
\begin{split}
A_{rr} & =r^{-2} - \alpha \, r^{-3} + O_1(r^{-4+\varepsilon}),\\
A_{r\mu} & = -r^{-2}\psi_\mu + O_1(r^{-3+\varepsilon}),\\
A_{\mu\nu} & =r^{2}\sigma_{\mu\nu}+O_1(1).
\end{split}
\end{equation*}
In particular, $|A|_{\gbar}^2=2 + O_1(r^{-2})$ and $|A-K|_{\gbar}^2 = O(r^{-4})$.
\item[5)] The components of the 1-form $q$  given by $q_i =f^j (A_{ij} - K_{ij})(1+|df|_g)^{-1/2}$
satisfy
\[
q_r = -\alpha \, r^{-2} + O_1(r^{-3+\varepsilon}), \qquad q_\mu = -r^{-1} \psi_\mu + O_1(r^{-2+\varepsilon}).
\]
We also have
\[
\divg^{\bar{g}} q = - r^{-3} \Delta^{\bS^2}\psi  + O(r^{-4}).
\]
\end{itemize}
\end{lemma}
\begin{proof}
A computation.
\end{proof}

\begin{lemma}\label{lemma1}
If $(M,g,k)$ is asymptotically hyperboloidal initial data with Wang's asymptotics and $f=\sqrt{1+r^2} + \alpha \ln r + \psi + O_3(r^{-1+\varepsilon})$ then the Jang metric $\gbar=g+df\otimes df$ is asymptotically flat in the sense of Definition \ref{defAEManifolds}, and its ADM mass is $\mathcal{M}(\gbar) = \alpha = 2E$.
\end{lemma}

\begin{proof}
It is clear that the graph of $f$ in $M\times\bR$ has an end diffeomorphic to $(R,\infty)\times \bS^2$, the coordinate diffeomorphism $\Psi$ being naturally induced by the asymptotically hyperbolic chart $\Phi : M\setminus \mathcal{C} \to (R,\infty) \times \bS^2$.  We have $\Psi_* \gbar  = \delta + O_2 (r^{-1})$ as a consequence of Lemma \ref{lemma2}.

We compute the mass $\mathcal{M} (\gbar)$ of the asymptotically Euclidean metric $\gbar$ using the formula
\begin{equation*}
\mathcal{M} (\gbar) =\frac{1}{16\pi}\int_{S_{\infty}}(\operatorname{div}^\delta \Psi_*\gbar - d \tr^{\delta} \Psi_*\gbar) (\nu) \, d\mu^\sigma.
\end{equation*}
Note that in this case we have
\begin{equation*}
\mathring{\Gamma}^{r}_{rr} = \mathring{\Gamma}^\alpha_{rr} = \mathring{\Gamma}^r_{\alpha r} = 0, \qquad \mathring{\Gamma}^r_{\alpha\beta} = -r \sigma_{\alpha\beta}, \qquad \mathring{\Gamma}^\alpha_{ \beta r} = r^{-1} \delta^\alpha_\beta, \qquad \mathring{\Gamma}^\alpha_{\beta\gamma} = (\Gamma_{\sigma})^\alpha_{\beta\gamma},
\end{equation*}
where $\mathring{\Gamma}_{ij}^{l}$ and $(\Gamma_{\sigma})^\alpha_{\beta\gamma}$ are Christoffel symbols for the metrics $\delta$ and $\sigma$ respectively, and
hence
\begin{align*}
\begin{split}
(\operatorname{div}^\delta \Psi_*g) (\nu)
& = (\operatorname{div}^\delta \Psi_* g)(\partial_{r}) \\
& = \mathring{\nabla}_r g_{rr} + r^{-2} \sigma^{\alpha\beta} \mathring{\nabla}_\beta g_{\alpha r} \\
& = \partial_r g_{rr} - 2 \mathring{\Gamma}^l_{rr} g_{lr} +  r^{-2} \sigma^{\alpha\beta}(\partial_\beta g_{\alpha r}                                            - \mathring{\Gamma}^l_{\alpha\beta} g_{lr}  - \mathring{\Gamma}^l_{\beta r} g_{\alpha l} ) \\
& = - r^{-2} \sigma^{\alpha\beta} g_{\alpha \gamma} \mathring{\Gamma}^\gamma_{\beta r} + O(r^{-3}) \\
& = -2r^{-1} + O(r^{-3}).
\end{split}
\end{align*}
Furthermore
\begin{align*}
\begin{split}
\operatorname{div}^\delta (df\otimes df) (\nu)
& = \operatorname{div}^\delta (df\otimes df)(\partial_{r}) \\
& = \mathring{\nabla}_r (f_r^2) + r^{-2} \sigma^{\alpha\beta} \mathring{\nabla}_\beta(f_\alpha f_r) \\
& = 2 f_r \mathring{\nabla}_{rr} f + r^{-2} \sigma^{\alpha\beta} (f_r \mathring{\nabla}_{\alpha\beta} f + f_\alpha \mathring{\nabla}_{\beta r} f),
\end{split}
\end{align*}
with
\begin{equation*}
2 f_r \mathring{\nabla}_{rr} f= -2 \alpha r^{-2} + O(r^{-3+\varepsilon}),
\end{equation*}
\begin{equation*}
r^{-2} \sigma^{\alpha\beta} f_\alpha \mathring{\nabla}_{\beta r} f = - r^{-2} \sigma^{\alpha\beta}f_\alpha \mathring{\Gamma}^\gamma_{\beta r} f_\gamma + O(r^{-3}) = O(r^{-3}),
\end{equation*}
and
\begin{align*}
\begin{split}
 r^{-2} \sigma^{\alpha\beta} f_r \mathring{\nabla}_{\alpha\beta} f
& = r^{-2} \sigma^{\alpha\beta} f_r \mathring{\nabla}_{\alpha\beta} \psi
+ r^{-2} \sigma^{\alpha\beta} f_r \mathring{\nabla}_{\alpha\beta} (f - \psi)\\
& = r^{-2}f_r\Delta^\sigma \psi - r^{-2} \sigma^{\alpha\beta} f_r \mathring{\Gamma}^r _{\alpha\beta} (f - \psi_r +  O(r^{-3+\varepsilon}))\\
& = r^{-2} \Delta^\sigma \psi +  2 r^{-1} (f_r)^2  +  O(r^{-3+\varepsilon}) \\
& = r^{-2} \Delta^\sigma \psi + 2 r^{-1} + 4 \alpha r^{-2} + O(r^{-3+\varepsilon}).
 \end{split}
\end{align*}
Finally, we have
\begin{align*}
\begin{split}
 (d \tr^{\delta} \Psi_* \gbar) (\nu) & = \partial_r (\gbar_{rr} + r^{-2} \sigma^{\alpha \beta} \gbar_{\alpha \beta}) \\
                                                 & = \partial_r(g_{rr} + f_r^2 + r^{-2} \sigma^{\alpha\beta} (g_{\alpha\beta} + f_\alpha f_{\beta}))\\
																								 & = \partial_r (f_r^2) + O(r^{-3})\\
																								 & = -2 \alpha r^{-2} + O(r^{-3 + \varepsilon}).
\end{split}
\end{align*}
Summing up, we conclude that
\begin{equation*}
\mathcal{M} (\gbar) = \frac{1}{16\pi}\int_{S_{\infty}}\left[(\Delta^\sigma \psi + 4 \alpha)r^{-2} +  O(r^{-3+\varepsilon})\right] \, d\mu^\sigma = \alpha = 2E.
\end{equation*}
\end{proof}

\bibliographystyle{amsalpha}
\bibliography{biblio}

\end{document}